\documentclass[12pt, a4paper]{amsart}

\usepackage[hmargin=30mm, vmargin=25mm, includefoot, twoside]{geometry}
\usepackage[bookmarksopen=true]{hyperref}

\usepackage{amscd}
\usepackage{amsfonts,amssymb,verbatim}
\usepackage{latexsym}
\usepackage{mathrsfs}
\usepackage{stmaryrd}
\usepackage{xspace}
\usepackage{enumerate, paralist}
\usepackage{graphicx}
\usepackage[all]{xy}
\usepackage{extarrows}

\usepackage[usenames,dvipsnames]{color}

\usepackage{txfonts, pxfonts}

\usepackage{amsthm}
\usepackage{amsmath}

\usepackage{todonotes}
\numberwithin{equation}{section}

\newtheorem{thm}{Theorem}[section]
 \newtheorem{cor}[thm]{Corollary}
 \newtheorem{lem}[thm]{Lemma}
 \newtheorem{prop}[thm]{Proposition}
 
 \newtheorem{Q}[thm]{Question}

\newenvironment{propbis}[1]
  {\renewcommand{\thethm}{\ref{#1}$'$}%
   \addtocounter{thm}{-1}%
   \begin{prop}}
  {\end{prop}}
  
  \newenvironment{propbiss}[1]
  {\renewcommand{\thethm}{\ref{#1}$''$}%
   \addtocounter{thm}{-1}%
   \begin{prop}}
  {\end{prop}}
  
\newtheorem{alphthm}{Theorem}			

\newtheorem{alphprop}[alphthm]{Proposition}

 \theoremstyle{definition}
  \newtheorem{defn}[thm]{Definition}

\newtheorem{alphdefn}[alphthm]{Definition}			

 \theoremstyle{remark}
 \newtheorem{rem}[thm]{Remark}
  \newtheorem{ex}[thm]{Example}

\newtheorem*{claim*}{Claim}

\def\NN{\mathbb{N}}
\def\RR{\mathbb{R}}
\def\CC{\mathbb{C}}
\def\ZZ{\mathbb{Z}}

\def\U{\mathcal{U}}
\def\H{\mathcal{H}}
\def\L{\mathcal{L}}
\def\Nd{\mathcal{N}}
\def\S{\mathcal{S}}
\def\G{\mathcal{G}}
\def\Gz{\mathcal{G}^{(0)}}
\def\P{\mathcal{P}}

\def\B{\mathfrak{B}}
\def\K{\mathfrak{K}}

\def\supp{\mathrm{supp}}
\def\ppg{\mathrm{prop}}

\def\diam{\mathrm{diam}}
\def\max{\mathrm{max}}
\def\gr{\mathrm{gr}}
\def\r{\mathrm{r}}
\def\s{\mathrm{s}}
\def\cb{\mathrm{cb}}

\def\L{{\bf L}}

\begin{document}

\title{Ghostly ideals in uniform Roe algebras}

\author{Qin Wang and Jiawen Zhang}

\address[Q. Wang]{Research Center for Operator Algebras, and Shanghai Key Laboratory of Pure Mathematics and Mathematical Practice, School of Mathematical Sciences, East China Normal University, Shanghai, 200241, China.}
\email{qwang@math.ecnu.edu.cn}

\address[J. Zhang]{School of Mathematical Sciences, Fudan University, 220 Handan Road, Shanghai, 200433, China.}
\email{jiawenzhang@fudan.edu.cn}

\date{}

\thanks{QW is partially supported by NSFC (No. 11831006, 12171156), and the Science and Technology Commission of Shanghai Municipality (No. 22DZ2229014). JZ is supported by NSFC11871342. }

\begin{abstract}
In this paper, we investigate the ideal structure of uniform Roe algebras for general metric spaces beyond the scope of Yu’s property A. Inspired by the ideal of ghost operators coming from expander graphs and in contrast to the notion of geometric ideal, we introduce a notion of ghostly ideal in a uniform Roe algebra, whose elements are locally invisible in certain directions at infinity. We show that the geometric ideal and the ghostly ideal are respectively the smallest and the largest element in the lattice of ideals with a common invariant open subset of the unit space of the coarse groupoid by Skandalis-Tu-Yu, and hence the study of ideal structure can be reduced to classifying ideals between the geometric and the ghostly ones. 
As an application, we provide a concrete description for the maximal ideals in a uniform Roe algebra in terms of the minimal points in the Stone-\v{C}ech boundary of the space. We also provide a criterion to ensure that the geometric and the ghostly ideals have the same $K$-theory, which helps to recover counterexamples to the Baum-Connes type conjectures. Moreover, we introduce a notion of partial Property A for a metric space to characterise the situation in which the geometric ideal coincides with the ghostly ideal.
\end{abstract}

\date{\today}
\maketitle

\parskip 4pt

\noindent\textit{Mathematics Subject Classification} (2020): 47L20, 46L80, 51F30.\\
\textit{Keywords: Uniform Roe algebras, Coarse groupoids, Geometric and ghostly ideals, Maximal ideals, Partial Property A}

\section{Introduction}\label{sec:intro}

Roe algebras are $C^*$-algebras associated to metric spaces, which encode the coarse geometry of the underlying spaces. They were introduced by Roe in his pioneering work on higher index theory \cite{Roe88}, where he discovered that the $K$-theory of Roe algebras serves as a receptacle for higher indices of elliptic differential operators on open manifolds. 
Hence the computation for the $K$-theory of Roe algebras becomes crucial in the study of higher index theory, and a pragmatic and practical approach is to consult the Baum-Connes type conjectures \cite{BC00, BCH94, HR95}. There is also a uniform version of the Roe algebra, which equally plays a key role in higher index theory (see \cite{STY02, Spa09}). Over the last four decades, there have been a number of excellent works around this topic (\emph{e.g.}, \cite{CWY13, HLS02, KY06, WY12, Yu00}), which lead to significant progresses in topology, geometry and analysis (see, \emph{e.g.}, \cite{Roe93, Roe96}).

On the other hand, the analytic structure of (uniform) Roe algebras reflects the coarse geometry of the underlying spaces, and the rigidity problem asks whether the coarse geometry of a metric space can be fully determined by the associated (uniform) Roe algebra. This problem was initially studied by \v{S}pakula and Willett in \cite{SW13}, followed by a series of works in the last decade \cite{BCL20, BF21, BFV20, BFV22, LSZ20}. Recently this problem is completely solved in the uniform case by the profound work \cite{BBFKVW22}, which again highlights the importance of uniform Roe algebras in coarse geometry. Meanwhile, uniform Roe algebras have also attained rapidly-growing interest from researchers in mathematical physics, especially in the theory of topological materials and topological insulators (see, \emph{e.g.}, \cite{EM19} and the references therein).

Due to their importance,
Chen and the first-named author initiated the study of the ideal structure for (uniform) Roe algebras \cite{CW01, CW04, CW04b, CW05, CW06, Wan07}. They succeeded in obtaining a full description for the ideal structure of the uniform Roe algebra when the underlying space has Yu's Property A (see \cite{CW04, CW05}). However, the general picture is far from clear beyond the scope of Property A.

Note that when the underlying space comes from a discrete group equipped with a word length metric, then the associated uniform Roe algebra has the form of a crossed product by the group itself acting on an abelian $C^*$-algebra coming from its Stone-\v{C}ech compactification (\cite{HR00, Oza00}). This class of $C^*$-algebras plays an important role in dynamic systems, and their ideal structures have also been extensively studied (see, \emph{e.g.}, \cite{KS19, Ren91, Sie10}). More generally as shown below, uniform Roe algebras for general metric spaces can be realised as reduced groupoid $C^*$-algebras. Recently, rapidly-growing interest has arisen in the study of ideal structures for general reduced groupoid $C^*$-algebras (see, \emph{e.g.}, \cite{BL20, BCS22, Ren91}), partially due to their importance in the study of purely infiniteness for $C^*$-algebras.

In the present paper, we aim to provide a systematic study on the ideal structure of uniform Roe algebras for general discrete metric spaces. To outline our main results, let us first explain some notions. 

Let $(X,d)$ be a discrete metric space of bounded geometry (see Section \ref{ssec:notions from coarse geometry} for precise definitions). Thinking of operators on $\ell^2(X)$ as $X$-by-$X$ matrices, we say that such an operator has \emph{finite propagation} if the non-zero entries appear only in an entourage of finite width (measured by the metric on $X$) around the main diagonal (see Section \ref{ssec:uniform Roe alg} for full details). The set of all finite propagation operators forms a $\ast$-subalgebra of $\B(\ell^2(X))$, and its norm closure is called the \emph{uniform Roe algebra of $X$} and denoted by $C^*_u(X)$. 

There is another viewpoint on the uniform Roe algebra based on groupoids. Recall from \cite{STY02} that Skandalis, Tu and Yu introduced a notion of coarse groupoid $G(X)$ associated to a discrete metric space $X$, and they succeeded in relating coarse geometry to the theory of groupoids. The coarse groupoid $G(X)$ is a locally compact, Hausdorff, \'{e}tale and principal groupoid (see Section \ref{ssec:coarse groupoid} for precise definitions), and the unit space of $G(X)$ coincides with the Stone-\v{C}ech compactification $\beta X$ of $X$. Moreover, the uniform Roe algebra $C^*_u(X)$ can be interpreted as the reduced groupoid $C^*$-algebra of $G(X)$ (see also \cite[Chapter 10]{Roe03}).


In \cite{CW04}, Chen and the first-named author concentrated on a class of ideals in the uniform Roe algebra in which finite propagation operators therein are dense, and they showed that these ideals can be described geometrically using the coarse groupoid. More precisely, recall that a subset $U \subseteq \beta X$ is \emph{invariant} if any element $\gamma$ in $G(X)$ with source in $U$ also has its range in $U$ (see Section \ref{ssec:groupoids}). As shown in \cite{CW04} (see also Section \ref{sec:geometric ideals}), for any ideal $I$ in $C^*_u(X)$ one can associate an invariant open subset $U(I)$ of $\beta X$, and conversely for any invariant open subset $U \subseteq \beta X$ one can associate an ideal $I(U)$ in $C^*_u(X)$. Furthermore, these two procedures provide a one-to-one correspondence between invariant open subsets of $\beta X$ and ideals in $C^*_u(X)$ in which finite propagation operators therein are dense.

Based on \cite{CW04}, the first-named author introduced the following notion in \cite[Definition 1.4]{Wan07}:

\begin{alphdefn}[Definition \ref{defn: geometric ideals}]\label{introdefn:geometric ideals}
Let $(X, d)$ be a discrete metric space of bounded geometry. An ideal $I$ in the uniform Roe algebra $C^*_u(X)$ is called \emph{geometric} if the set of all finite propagation operators in $I$ is dense in $I$.
\end{alphdefn}

As explained above, \cite[Theorem 6.3]{CW04} (see also Proposition \ref{prop: 1-1 btw ideals and open inv subsets}) indicates that the geometric ideals in $C^*_u(X)$ can be fully determined by invariant open subsets of $\beta X$, which explains the terminology. Consequently, the geometric ideals in $C^*_u(X)$ are easy to handle and they must have the form of $I(U)$ for some invariant open subset $U \subseteq \beta X$, called the \emph{geometric ideal associated to $U$} (see Definition \ref{defn:geometric ideal}). Moreover, it follows from \cite[Theorem 4.4]{CW05} that all ideals in $C^*_u(X)$ are geometric when $X$ has Yu's Property A.

However, things get complicated beyond the context of Property A. As noticed in \cite[Remark 6.5]{CW04}, when $X$ comes from a sequence of expander graphs then the ideal $I_G$ consisting of all ghost operators are \emph{not} geometric (see also \cite{HLS02}). Recall that an operator $T \in \B(\ell^2(X))$ is a \emph{ghost} if $T \in C_0(X \times X)$ when regarding $T$ as a function on $X \times X$. Ghost operators are introduced by Yu, and they are crucial to provide counterexamples to the coarse Baum-Connes conjecture (\cite{HLS02}). 

Direct calculations show that the associated invariant open subsets for $I_G$ and for the ideal of compact operators in $\B(\ell^2(X))$ are the same, both of which equal $X$ (see also Example \ref{eg: compact ideal} and \ref{eg: ghost ideal}). Hence for a general metric space $X$ and an invariant open subset $U \subseteq \beta X$, there might be more than one ideal $I$ in the uniform Roe algebra $C^*_u(X)$ satisfying $U(I)=U$. Therefore, the study of the ideal structure for $C^*_u(X)$ can be reduced to analyse the lattice (where the order is given by inclusion)
\begin{equation}\label{EQ:set of ideals}
\mathfrak{I}_U:=\{I \mbox{~is~an~ideal~in~} C^*_u(X): U(I)=U\}
\end{equation}
for each invariant open subset $U \subseteq \beta X$.

One of the main contributions of the present paper is to find the smallest and the largest elements in the lattice $\mathfrak{I}_U$. Following the discussions in \cite{CW04}, it is easy to see that $I(U(I)) \subseteq I$ for any ideal $I$ in $C^*_u(X)$, which implies that the geometric ideal $I(U)$ is the smallest element in $\mathfrak{I}_U$ (see Proposition \ref{cor: geometric ideals are smallest}). To explore the largest element, we have to include every ideal $I$ in $C^*_u(X)$ with $U(I)=U$. Inspired by the definition of $U(I)$ (see Equality (\ref{EQ: defn for U(I)})), we introduce the following key notion:

\begin{alphdefn}[Definition \ref{defn: ghostly ideals}]\label{introdefn:ghostly ideals}
Let $(X, d)$ be a discrete metric space of bounded geometry and $U$ be an invariant open subset of $\beta X$. The \emph{ghostly ideal associated to $U$} is defined to be
\[
\tilde{I}(U):=\{T \in C^*_u(X): \overline{\r(\supp_\varepsilon(T))} \subseteq U \mbox{~for~any~}\varepsilon>0\},
\]
where $\supp_\varepsilon(T):=\{(x,y)\in X \times X: |T(x,y)| \geq \varepsilon\}$ and $\r: X \times X \to X$ is the projection onto the first coordinate.
\end{alphdefn}

We show that $\tilde{I}(U)$ is indeed an ideal in the uniform Roe algebra $C^*_u(X)$ (see Lemma \ref{lem: ghostly ideals are ideals}) and moreover, we obtain the following desired result:

\begin{alphthm}[Theorem \ref{prop: ideals containment}]\label{introthm:inclusion of ideals}
Let $(X,d)$ be a discrete metric space of bounded geometry and $U$ be an invariant open subset of $\beta X$. Then any ideal $I$ in $C^*_u(X)$ with $U(I)=U$ sits between $I(U)$ and $\tilde{I}(U)$. More precisely, the geometric ideal $I(U)$ is the smallest element while the ghostly ideal $\tilde{I}(U)$ is the largest element in the lattice $\mathfrak{I}_U$ in (\ref{EQ:set of ideals}).
\end{alphthm}

Theorem \ref{introthm:inclusion of ideals} draws the border of the lattice $\mathfrak{I}_U$ in (\ref{EQ:set of ideals}), as an important step to study the ideal structure of uniform Roe algebras for general metric spaces. More precisely, once we can bust every ideal between $I(U)$ and $\tilde{I}(U)$ for each invariant open subset $U \subseteq \beta X$, then we will obtain a full description for the ideal structure of the uniform Roe algebra $C^*_u(X)$. We pose it as an open question in Section \ref{sec:open questions} and hope this will be done in some future work.

Concerning the ghostly ideal $\tilde{I}(U)$, we also provide an alternative picture in terms of limit operators developed in \cite{SW17}, showing that $\tilde{I}(U)$ consists of operators which vanish in the $(\beta X \setminus U)$-direction (see Proposition \ref{prop: geometric char for ghostly ideals}). Note that ghost operators vanish in all directions (see Corollary \ref{cor:char for ghost}), and hence operators in $\tilde{I}(U)$ can be regarded as ``partial'' ghosts, which clarifies its terminology. Thanks to this viewpoint, we discover the deep reason behind the counterexample to the conjecture in \cite{CW04}, constructed by the first-named author in \cite[Section 3]{Wan07} (see Example \ref{ex: Wan07}).

As an application, we manage to describe maximal ideals in the uniform Roe algebra. More precisely, it follows directly from Theorem \ref{introthm:inclusion of ideals} that maximal ideals correspond to minimal invariant closed subsets of the Stone-\v{C}ech boundary $\partial_\beta X:=\beta X \setminus X$. Moreover using the theory of limit spaces\footnote{Note that the theory of limit spaces and limit operators developed in \cite{SW17} only concerns strongly discrete metric spaces of bounded geometry (see Section \ref{ssec:notions from coarse geometry} for precise definitions). Although as noticed in \cite{SW17} this will not lose any generality, we put this assumption to simplify proofs.} developed in \cite{SW17}, we prove the following:

\begin{alphprop}[Proposition \ref{lem:maximal ideals using maximal open subsets}, Corollary \ref{cor:maximal ideals using minimal closed subsets} and Lemma \ref{lem:property of minimal points}]\label{introprop:maximal ideals}
Let $(X,d)$ be a strongly discrete metric space of bounded geometry and $I$ be a maximal ideal in the uniform Roe algebra $C^*_u(X)$. Then there exists a point $\omega \in \partial_\beta X$ such that $I$ coincides with the ghostly ideal $\tilde{I}(\beta X \setminus \overline{X(\omega)})$, where $X(\omega)$ is the limit space of $\omega$. 
\end{alphprop}

A point $\omega \in \partial_\beta X$ satisfying the condition in Proposition \ref{introprop:maximal ideals} is called a \emph{minimal point} (see Definition \ref{defn:minimal point}). We show that there exist a number of non-minimal points in the boundary even for the simple case of $X=\ZZ$:

\begin{alphthm}[Theorem \ref{thm:non-min point for Z}]\label{introthm:maximal ideals for Z}
For the integer group $\ZZ$ with the usual metric, there exist non-minimal points in the boundary $\partial_\beta \ZZ$. More precisely, for any sequence $\{h_n\}_{n\in \NN}$ in $\ZZ$ tending to infinity such that $|h_n -h_m| \to +\infty$ when $n+m \to \infty$ and $n\neq m$, and any $\omega \in \partial_\beta \ZZ$ with $\omega(\{h_n\}_{n\in \NN})=1$, then $\omega$ is \emph{not} a minimal point.
\end{alphthm}

We provide two approaches to prove Theorem \ref{introthm:maximal ideals for Z}. One is topological, which makes use of several constructions and properties of ultrafilters (recalled in Appendix \ref{app:ultrafilters}). The other is $C^*$-algebraic, which replies on a description of maximal ideals in terms of limit operators (see Lemma \ref{lem:simplification of maximal ideal}) together with a recent result by Roch \cite{Roc22}.

Returning to the lattice $\mathfrak{I}_U$ defined in (\ref{EQ:set of ideals}), we already notice that generally $\mathfrak{I}_U$ consists of more than one element. Hence it will be interesting and important to explore when $\mathfrak{I}_U$ has only a single element, or equivalently (thanks to Theorem \ref{introthm:inclusion of ideals}), when the geometric ideal $I(U)$ coincides with the ghostly ideal $\tilde{I}(U)$. 

To study this problem, we start with an extra picture for geometric and ghostly ideals using the associated groupoid $C^*$-algebras (see Lemma \ref{lem: groupoid C*-alg char for I(U)} and Proposition \ref{prop: char for Ker qU}). Based on these descriptions, we show that the amenability of the restriction $G(X)_{\partial_\beta X \setminus U}$ of the coarse groupoid ensures that $I(U) = \tilde{I}(U)$ (see Proposition \ref{prop: criteria to ensure I(U) = tilde I(U)}). Meanwhile, we also discuss the $K$-theory of the geometric and ghostly ideals and provide a criterion to ensure that $K_\ast(I(U)) = K_\ast(\tilde{I}(U))$ for $\ast=0,1$:

\begin{alphprop}[Proposition \ref{prop: criteria to ensure the same K-theory}]\label{introprop:K-theory}
Let $X$ be a discrete metric space of bounded geometry which can be coarsely embedded into some Hilbert space. Then for any invariant open subset $U \subseteq \beta X$, we have an isomorphism
\[
(\iota_U)_\ast: K_\ast(I(U)) \longrightarrow K_\ast(\tilde{I}(U))
\]
for $\ast =0,1$, where $\iota_U$ is the inclusion map. 
\end{alphprop}

Applying Proposition \ref{introprop:K-theory} to the case of $U=X$, we partially recover \cite[Proposition 35]{Fin14}. This is crucial in the constructions of counterexamples to the Baum-Connes type conjectures (see \cite{STY02} for the coarse version and \cite[Section 5]{FW14} for the boundary version, the latter of which is based on the example considered in \cite[Section 3]{Wan07}). Hence presumably Proposition \ref{introprop:K-theory} will find further applications in higher index theory.

Conversely, it is natural to ask whether $I(U) = \tilde{I}(U)$ implies that the restriction groupoid $G(X)_{\partial_\beta X \setminus U}$ is amenable. Note that when $U=X$, \cite[Theorem 1.3]{RW14} implies that $I(X) = \tilde{I}(X)$ is equivalent to that $X$ has Property A (see also Example \ref{eg: compact ideal} and \ref{eg: ghost ideal}), which is further equivalent to that the coarse groupoid $G(X)$ is amenable thanks to \cite[Theorem 5.3]{STY02}. Inspired by these works, we introduce the following partial version of Property A:

\begin{alphdefn}[Definition \ref{defn:partial Property A}]\label{introdefn:partial A}
Let $(X,d)$ be a discrete metric space of bounded geometry and $U \subseteq \beta X$ be an invariant open subset. We say that $X$ has \emph{partial Property A towards $\partial_\beta X \setminus U$} if $G(X)_{\partial_\beta X \setminus U}$ is amenable.
\end{alphdefn}

Finally we reach the following, which recovers \cite[Theorem 1.3]{RW14} when $U=X$:

\begin{alphthm}[Theorem \ref{thm:I(U)=tI(U)}]\label{introthm:I(U)=tI(U)}
Let $(X,d)$ be a strongly discrete metric space of bounded geometry and $U \subseteq \beta X$ be a countably generated invariant open subset.
Then the following are equivalent:
\begin{enumerate}
 \item $X$ has partial Property A towards $\beta X \setminus U$;
 \item $\tilde{I}(U) = I(U)$;
 \item the ideal $I_G$ of all ghost operators is contained in $I(U)$.
\end{enumerate}
\end{alphthm}


Note that there is a technical condition of countable generatedness (see Definition \ref{defn:countably generated}) used in Theorem \ref{introthm:I(U)=tI(U)}, which holds for a number of examples (see Example \ref{ex:countably generated}) including $X$ itself. However as shown in Example \ref{ex:noncountably generated ideal}, there does exist an invariant open subset which is \emph{not} countably generated. 

The proof of Theorem \ref{introthm:I(U)=tI(U)} follows the outline of the case that $U=X$ (cf. \cite[Theorem 1.3]{RW14}), and is divided into several steps. Firstly, we unpack the groupoid language of Definition \ref{introdefn:partial A} and provide a concrete geometric description similar to the definition of Property A (see Proposition \ref{prop:char for partial A}). Then we introduce a notion of partial operator norm localisation property (Definition \ref{defn:partial ONL}), which is a partial version of the operator norm localisation property (ONL) introduced in \cite{CTWY08}. Parallel to Sako's result that Property A is equivalent to ONL (\cite{Sak14}), we show that partial Property A is equivalent to partial ONL (see Proposition \ref{prop:partial A = partial ONL}). Finally thanks to the assumption of countable generatedness, we conclude Theorem \ref{introthm:I(U)=tI(U)}. 

We also remark that in the proof of Theorem \ref{introthm:I(U)=tI(U)}, we make use of the notion of ideals in spaces introduced by Chen and the first-named author in \cite{CW04} (see also Definition \ref{defn: ideals in space}) instead of using invariant open subsets of $\beta X$ directly. This has the advantage of playing within the given space rather than going to the mysterious Stone-\v{C}ech boundary, which allows us to step over several technical gaps (see, \emph{e.g.}, Remark \ref{rem:importance of ideals in space}).

The paper is organised as follows. In Section \ref{sec:pre}, we recall necessary background knowledge in coarse geometry and groupoid theory. In Section \ref{sec:limit sp. and op.}, we recall the theory of limit spaces and limit operators developed in \cite{SW17}, which will be an important tool used throughout the paper. Section \ref{sec:geometric ideals} is devoted to the notion of geometric ideals (Definition \ref{introdefn:geometric ideals}) studied in \cite{CW04, Wan07}, and we also discuss their minimality in the lattice of ideals $\mathfrak{I}_U$ from (\ref{EQ:set of ideals}). We introduce the key notion of ghostly ideals (Definition \ref{introdefn:ghostly ideals}) in Section \ref{sec:ghostly ideals}, prove Theorem \ref{introthm:inclusion of ideals} and provide several characterisations for later use. Then we discuss maximal ideals in uniform Roe algebras in Section \ref{sec:maximal ideals}, and prove Proposition \ref{introprop:maximal ideals} and Theorem \ref{introthm:maximal ideals for Z}. In Section \ref{sec:geometric vs ghostly ideals}, we study the problem when the geometric ideal coincides with the ghostly ideal, discuss their $K$-theories and prove Proposition \ref{introprop:K-theory}. Then in Section \ref{sec:partial A and partial ONL}, we introduce the notion of partial Property A (Definition \ref{introdefn:partial A}) and prove Theorem \ref{introthm:I(U)=tI(U)}. Finally, we list some open questions in Section \ref{sec:open questions}, and provide Appendix \ref{app:ultrafilters} to record the notion of ultrafilters and their properties used throughout the paper.

\noindent{\bf Acknowledgement.} We would like to thank Kang Li for reading an early draft of this paper and pointing out that the assumption of countable generatedness is redundant for Proposition \ref{introprop:K-theory} (see Remark \ref{rem:thanks to Kang}). We would also like to thank Baojie Jiang and J\'{a}n \v{S}pakula for some helpful discussions.

\section{Preliminaries}\label{sec:pre}

\subsection{Standard notation} Here we collect the notation used throughout the paper.

For a set $X$, denote by $|X|$ the cardinality of $X$. For a subset $A \subseteq X$, denote by $\chi_{A}$ the characteristic function of $A$, and set $\delta_x:=\chi_{\{x\}}$ for $x\in X$.

When $X$ is a locally compact Hausdorff space, we denote by $C(X)$ the set of complex-valued continuous functions on $X$, and by $C_b(X)$ the subset of bounded continuous functions on $X$. Recall that the \textit{support} of a function $f\in C(X)$ is the closure of $\{ x\in X: f(x)\neq 0\}$, written as $\supp(f)$, and denote by $C_c(X)$ the set of continuous functions with compact support. We also denote by $C_0(X)$ the set of continuous functions vanishing at infinity, which is the closure of $C_c(X)$ with respect to the supremum norm $\|f\|_\infty:=\sup\{|f(x)|: x\in X\}$.

When $X$ is discrete, denote $\ell^\infty(X):=C_b(X)$ and $\ell^2(X)$ the Hilbert space of complex-valued square-summable functions on $X$. Denote by $\B(\ell^2(X))$ the $C^*$-algebra of all bounded linear operators on $\ell^2(X)$, and by $\K(\ell^2(X))$ the $C^*$-subalgebra of all compact operators on $\ell^2(X)$. 

For a discrete space $X$, denote by $\beta X$ its Stone-\v{C}ech compactification and $\partial_\beta X:=\beta X \setminus X$ the Stone-\v{C}ech boundary.

\subsection{Notions from coarse geometry}\label{ssec:notions from coarse geometry}

Here we collect necessary notions from coarse geometry, and guide readers to \cite{NY12, Roe03} for more details. 

For a discrete metric space $(X,d)$, denote the closed ball by $B_{X}(x,r):=\{y\in X: d(x,y)\leq r\}$ for $x\in X$ and $r\geq 0$. For a subset $A\subseteq X$ and $r>0$, denote the \emph{$r$-neighbourhood of $A$ in $X$} by $\Nd_r(A):=\{x\in X: d_X(x,A) \leq r\}$. For $R>0$, denote the \emph{$R$-entourage} by $E_R :=\{(x,y) \in X \times X: d(x,y) \leq R\}$. 

We say that $(X,d)$ has \emph{bounded geometry} if for any $r>0$, the number $\sup_{x\in X} |B_{X}(x,r)|$ is finite. Also say that $(X,d)$ is \emph{strongly discrete} if the set $\{d(x,y): x,y \in X\}$ is a discrete subset of $\RR$. 


\noindent \textbf{Convention.} We say that ``\emph{$X$ is a space}'' as shorthand for ``$X$ is a strongly discrete metric space of bounded geometry'' (as in \cite{SW17}) throughout the rest of this paper. 

We remark that although our results hold without the assumption of strong discreteness, we choose to add it so as to simplify the proofs. As discussed in \cite[Section 2]{SW17}, this will not lose any generality since one can always modify a discrete metric space (using a coarse equivalence) to satisfy this assumption.


Now we recall the notion of Property A introduced by Yu (see, \emph{e.g.}, \cite[Proposition 1.2.4]{Wil09} for the equivalence to Yu’s original definition):

\begin{defn}[\cite{Yu00}]\label{defn: Property A}
A space $(X,d)$ is said to have \emph{Property A} if for any $\varepsilon, R>0$ there exist an $S>0$ and a function $f: X \times X \to [0,+\infty)$ satisfying:
\begin{enumerate}
  \item $\supp(f) \subseteq E_S$;
  \item for any $x\in X$, we have $\sum_{z\in X} f(z,x) =1$;
  \item for any $x,y\in X$ with $d(x,y) \leq R$, then $\sum_{z\in X} |f(z,x) - f(z,y)| \leq \varepsilon$.
\end{enumerate}
\end{defn}

Using a standard normalisation argument, we have the following:

\begin{lem}\label{lem:char for A}
A space $(X,d)$ has Property A \emph{if and only if} for any $\varepsilon, R>0$ there exist an $S>0$ and a function $f: X \times X \to [0,+\infty)$ satisfying:
\begin{enumerate}
  \item $\supp(f) \subseteq E_S$;
  \item for any $x\in X$, we have $|\sum_{z\in X} f(z,x) - 1 | \leq \varepsilon$;
  \item for any $x,y\in X$ with $d(x,y) \leq R$, then $\sum_{z\in X} |f(z,x) - f(z,y)| \leq \varepsilon$.
\end{enumerate}
\end{lem}


We also need a characterisation for Property A using kernels. Recall that a \emph{kernel} on $X$ is a function $k\colon X\times X\rightarrow\RR$. We say that $k$ is of \textit{positive type} if for any $n\in\NN$, $x_{1},\ldots,x_{n}\in X$ and $\lambda_1,\ldots,\lambda_{n}\in\RR$, we have:
\[
    \sum_{i,j=1}^{n}\lambda_{i}\lambda_{j}k(x_{i},x_{j})\geq 0.
\]
The following is well-known (see, \emph{e.g.}, \cite[Proposition 1.2.4]{Wil09}):

\begin{lem}\label{lem:char for A}
A space $(X,d)$ has \textit{Property A} if and only if for any $R>0$ and $\varepsilon>0$, there exist $S>0$ and a kernel $k: X \times X \to \RR$ of positive type satisfying the following:
\begin{enumerate}
 \item for $x,y\in X$, we have $k(x,y)=k(y,x)$ and $k(x,x) =1$;
 \item for $x,y\in X$ with $d(x,y) \geq S$, we have $k(x,y)=0$;
 \item for $x,y\in X$ with $d(x,y) \leq R$, we have $|1-k(x,y)| \leq \varepsilon$.
\end{enumerate}
\end{lem}

We also recall the notion of coarse embedding:

\begin{defn}\label{defn: coarse equivalence}
Let $(X,d_X)$ and $(Y,d_Y)$ be metric spaces and $f: X \to Y$ be a map. We say that $f$ is a \emph{coarse embedding} if there exist functions $\rho_{\pm}: [0,\infty) \to [0,\infty)$ with $\lim_{t\to +\infty}\rho_{\pm}(t) = +\infty$ such that for any $x,y\in X$ we have
\[
\rho_-(d_X(x,y)) \leq d_Y(f(x),f(y)) \leq \rho_+(d_X(x,y)).
\]
If additionally there exists $C>0$ such that $Y=\Nd_C(f(X))$, then we say that $f$ is a \emph{coarse equivalence} and $(X,d_X), (Y,d_Y)$ are \emph{coarsely equivalent}.
\end{defn}

\subsection{Uniform Roe algebras}\label{ssec:uniform Roe alg}
Let $(X,d)$ be a discrete metric space. Each operator $T \in \B(\ell^2(X))$ can be written in the matrix form $T=(T(x,y))_{x,y\in X}$, where $T(x,y)=\langle T \delta_y, \delta_x \rangle \in \CC$. We also regard $T \in \B(\ell^2(X))$ as a bounded function on $X \times X$, \emph{i.e.}, an element in $\ell^\infty(X \times X)$. 
Denote by $\|T\|$ the operator norm of $T$ in $\B(\ell^2(X))$, and $\|T\|_\infty$ the supremum norm when regarding $T$ as a function in $\ell^\infty(X \times X)$. It is clear that $\|T\|_\infty \leq \|T\|$ for any $T \in \B(\ell^2(X))$.

Given an operator $T \in \B(\ell^2(X))$, we define the \emph{support} of $T$ to be
\[
\supp(T):=\{(x,y) \in X \times X: T(x,y) \neq 0\},
\]
and the \emph{propagation} of $T$ to be 
\[
\ppg(T):= \sup\{d(x,y): (x,y) \in \supp(T)\}.
\]

\begin{defn}\label{defn: unif. Roe alg.}
Let $(X,d)$ be a space. 
\begin{enumerate}
 \item The set of all finite propagation operators in $\B(\ell^2(X))$ forms a $\ast$-algebra, called the \emph{algebraic uniform Roe algebra of $X$} and denoted by $\CC_u[X]$. For each $R\geq 0$, denote the subset 
 \[
 \CC_u^R[X]:=\{T\in \B(\ell^2(X)): \ppg(T) \leq R\}.
 \]
 It is clear that $\CC_u[X]=\bigcup_{R \geq 0} \CC_u^R[X]$.
 \item The \emph{uniform Roe algebra of $X$} is defined to be the operator norm closure of $\CC_u[X]$ in $\B(\ell^2(X))$, which forms a $C^*$-algebra and is denoted by $C^*_u(X)$.
\end{enumerate}
\end{defn}

The following notion was originally introduced by Yu:

\begin{defn}\label{defn: ghost operator}
An operator $T \in C^*_u(X)$ is called a \emph{ghost} if $T \in C_0(X \times X)$ when regarding $T$ as a function in $\ell^\infty(X \times X)$. In other words, for any $\varepsilon$ there exists a finite subset $F \subseteq X$ such that for any $(x,y) \notin F \times F$, we have $|T(x,y)| < \varepsilon$.
\end{defn}


It is easy to see that all the ghost operators in $C^*_u(X)$ form an ideal in $C^*_u(X)$, denoted by $I_G$. Intuitively speaking, a ghost operator is locally invisible at infinity in all directions. This will be made more precise in the sequel.

\subsection{Groupoids and $C^*$-algebras} \label{ssec:groupoids}

We collect here some basic notions and terminology on groupoids. Details can be found in \cite{Ren80}, or \cite{Sim17} in the \'{e}tale case.

Recall that a \emph{groupoid} is a small category, in which every morphism is invertible. More precisely, a groupoid consists of a set $\G$, a subset $\Gz$ called the \emph{unit space}, two maps $\s,\r: \G \to \Gz$ called the \emph{source} and \emph{range} maps respectively, a \emph{composition law}:
\[
\G^{(2)}:=\{(\gamma_1,\gamma_2) \in \G \times \G: \s(\gamma_1)=\r(\gamma_2)\}\ni(\gamma_1,\gamma_2) \mapsto \gamma_1\gamma_2 \in \G,
\]	
and an \emph{inverse} map $\gamma \mapsto \gamma^{-1}$. These operations satisfy a couple of axioms, including associativity law and the fact that elements in $\Gz$ act as units. 

For $x \in \Gz$, denote $\G^x:=\r^{-1}(x)$ and $\G_x:=\s^{-1}(x)$. 
For $Y\subseteq \Gz$, denote $\G_Y^Y:=\r^{-1}(Y) \cap \s^{-1}(Y)$. 
Note that $\G_Y^Y$ is a subgroupoid of $\G$ (in the sense that it is stable under multiplication and inverse), called the \emph{reduction of $\G$ by $Y$}. 
A subset $Y$ is said to be \emph{invariant} if $\r^{-1}(Y)=\s^{-1}(Y)$, and we write $\G_Y$ instead of $\G_Y^Y$ in this case.

A \emph{locally compact Hausdorff} groupoid is a groupoid equipped with a locally compact and Hausdorff topology such that the structure maps (composition and inverse) are continuous with respect to the induced topologies. 
Such a groupoid is called \emph{\'{e}tale} (also called \emph{$r$-discrete}) if the range (hence the source) map is a local homeomorphism. Clearly in this case, each fibre $\G^x$ (and $\G_x$) is discrete with the induced topology, and $\Gz$ is clopen in $\G$. The notion of \'{e}taleness for a groupoid can be regarded as an analogue of discreteness in the group case.

\begin{ex}\label{eg: pair groupoid}
Let $X$ be a set. The \emph{pair groupoid} of $X$ is $X \times X$ as a set, whose unit space is $\{(x,x) \in X \times X: x\in X\}$ and identified with $X$ for simplicity. The source map is the projection onto the second coordinate and the range map is the projection onto the first coordinate. The composition is given by $(x,y) \cdot (y,z)=(x,z)$ for $x,y,z \in X$. When $X$ is a discrete Hausdorff space, then $X \times X$ is a locally compact Hausdorff \'{e}tale groupoid. 
\end{ex}

Now we introduce the algebras associated to groupoids. Here we only focus on the case of \'{e}taleness, and guide readers to \cite{Ren80} for the general case.

Let $\G$ be a locally compact, Hausdorff and \'{e}tale groupoid with unit space $\Gz$. Note that the space $C_c(\G)$ is a $\ast$-involutive algebra with respect to the following operations: for $f,g \in C_c(\G)$,
\begin{eqnarray*}
	(f \ast g)(\gamma) &=& \sum_{\alpha \in \G_{\s(\gamma)}} f(\gamma \alpha^{-1}) g(\alpha), \\
	f^*(\gamma) &=& \overline{f(\gamma^{-1})}.
\end{eqnarray*}
Consider the following algebraic norm on $C_c(\G)$ defined by:
\[
\|f\|_I:=\max\left\{\sup_{x \in \Gz} \sum_{\gamma \in \G_x} |f(\gamma)|, ~\sup_{x \in \Gz} \sum_{\gamma \in \G_x} |f^*(\gamma)|\right\}.
\]
The completion of $C_c(\G)$ with respect to the norm $\|\cdot\|_I$ is denoted by $L^1(\G)$.

The \emph{maximal (full) groupoid $C^*$-algebra} $C^*_{\max}(\G)$ is defined to be the completion of $C_c(\G)$ with respect to the norm:
\[
\|f\|_{\max}:=\sup \|\pi(f)\|,
\]
where the supremum is taken over all $\ast$-representations $\pi$ of $L^1(\G)$.

In order to define the reduced counterpart, we recall that for each $x \in \Gz$ the \emph{regular representation at $x$}, denoted by $\lambda_x: C_c(\G) \to \B(\ell^2(\G_x))$, is defined as follows:
\begin{equation}\label{reduced algebra defn}
\big(\lambda_x(f)\xi\big)(\gamma):=\sum_{\alpha \in \G_x} f(\gamma \alpha^{-1})\xi(\alpha), \quad \mbox{where}~ f \in C_c(\G)\mbox{~and~}\xi \in \ell^2(\G_x).
\end{equation}
It is routine work to check that $\lambda_x$ is a well-defined $\ast$-homomorphism. The \emph{reduced norm} on $C_c(\G)$ is
\[
\|f\|_r:=\sup_{x \in \Gz} \|\lambda_x(f)\|,
\]
and the \emph{reduced groupoid $C^*$-algebra} $C^*_r(\G)$ is defined to be the completion of the $\ast$-algebra $C_c(\G)$ with respect to this norm. Clearly, each regular representation $\lambda_x$ can be extended to a homomorphism $\lambda_x: C^*_r(\G) \to \B(\ell^2(\G_x))$ automatically. It is also routine to check that there is a canonical surjective homomorphism from $C^*_{\max}(\G)$ to $C^*_r(\G)$.

\subsection{Coarse groupoids} \label{ssec:coarse groupoid}
Let $(X,d)$ be a space as in Section \ref{ssec:notions from coarse geometry}. The coarse groupoid $G(X)$ on $X$ was introduced by Skandalis, Tu and Yu in \cite{STY02} (see also \cite[Chapter 10]{Roe03}) to relate coarse geometry to the theory of
groupoids. As a topological space,
\[
G(X):=\bigcup_{r>0}{\overline{E_r}}^{\beta (X \times X)} \subseteq \beta (X \times X).
\]
Recall from Example \ref{eg: pair groupoid} that $X \times X$ is the pair groupoid with source and range maps $\s(x,y)=y$ and $\r(x,y)=x$. These maps extend to maps $G(X) \to \beta X$, still denoted by $\r$ and $\s$.

Now consider $(\r,\s): G(X) \to \beta X \times \beta X$. It was shown in \cite[Lemma 2.7]{STY02} that the map $(\r,\s)$ is injective, and hence $G(X)$ can be endowed with a groupoid structure induced by the pair groupoid $\beta X \times \beta X$, called the \emph{coarse groupoid} of $X$. Therefore, $G(X)$ can also be equivalently defined by
\[
G(X):=\bigcup_{r>0}{\overline{E_r}}^{\beta X \times \beta X} \subseteq \beta X \times \beta X,
\]
with the weak topology. It was also shown in \cite[Proposition 3.2]{STY02} that the coarse groupoid $G(X)$ is locally compact, Hausdorff, \'{e}tale and principal. Clearly, the unit space of $G(X)$ can be identified with $\beta X$. 

Given $f \in C_c(G(X))$, then $f$ is a continuous function supported on $\overline{E_r}$ for some $r>0$; equivalently, we can interpret $f$ as a bounded function on $E_r$. Hence we define an operator $\theta(f)$ on $\ell^2(X)$ by setting its matrix coefficients to be $\theta(f)(x,y):=f(x,y)$ for $x,y\in X$. We have the following:

\begin{prop}[{\cite[Proposition 10.29]{Roe03}}]\label{prop: groupoid char for uniform Roe alg}
The map $\theta$ provides a $\ast$-isomorphism from $C_c(G(X))$ to $\CC_u[X]$, and extends to a $C^*$-isomorphism $\Theta: C^*_r(G(X)) \to C^*_u(X)$. Note that $\Theta$ maps the $C^*$-subalgebra $C^*_r(X \times X)$ onto the compact operators $\mathfrak{K}(\ell^2(X))$. 
\end{prop}

Recall from Section \ref{ssec:uniform Roe alg} that an operator $T \in \B(\ell^2(X))$ can be regarded as an element in $\ell^\infty(X \times X)$. Following the notation from \cite{CW04}, we denote by $\overline{T}$ the continuous extension of $T$ on $\beta(X \times X)$ when regarding $T \in \ell^\infty(X \times X)$. Then
\[
\supp(\overline{T}) = \overline{\supp(T)}.
\]
Note that $G(X)$ is open in $\beta(X \times X)$, hence $C_0(G(X))$ is a subalgebra in $C(\beta(X \times X))$. Restricting to $G(X)$, we also regard $\overline{T}$ as a function on $G(X)$ and hence we can talk of the value $\overline{T}(\alpha, \gamma)$ for $(\alpha, \gamma) \in G(X)$. Moreover, we have the following:

\begin{lem}\label{lem: unif. Roe belongs to C0}
For $T \in \CC_u[X]$, we have $\overline{T} \in C_c(G(X))$ and $\theta(\overline{T})=T$. For $T \in C^*_u(X)$, we have $\overline{T} \in C_0(G(X))$ and $\Theta(\overline{T}) = T$. 
\end{lem}

\begin{proof}
The first statement is a direct corollary of Proposition \ref{prop: groupoid char for uniform Roe alg}. Since $\|T\|_\infty \leq \|T\|$ for any $T \in \B(\ell^2(X))$, the second follows from the first.
\end{proof}



\subsection{Amenability and a-T-menability for groupoids}\label{ssec: amenability}
Amenable groupoids comprise a large class of groupoids with relatively nice properties, which are literally the analogue of amenable groups in the world of groupoids. Here we only focus on the case of \'{e}taleness, in which the notion of amenability behaves quite well. A standard reference is \cite{ADR00} and another reference for just \'{e}tale groupoids is \cite[Chapter 5.6]{BO08}.

\begin{defn}[\cite{ADR00}]\label{defn: amenable groupoid}
A locally compact, Hausdorff and \'{e}tale groupoid $\G$ is said to be \emph{(topologically) amenable} if for any $\varepsilon>0$ and compact $K \subseteq \G$, there exists $f \in C_c(\G)$ with range in $[0,1]$ such that for any $\gamma \in K$ we have
\[
\sum_{\alpha \in \G_{\r(\gamma)}} f(\alpha) =1  \quad \mbox{and} \quad \sum_{\alpha \in \G_{\r(\gamma)}} |f(\alpha) - f(\alpha\gamma)| < \varepsilon.
\]
\end{defn}

Similar to the case of Property A, we have the following by a standard normalisation argument:

\begin{lem}\label{lem:elementary char for amenable groupoid}
A locally compact, Hausdorff and \'{e}tale groupoid $\G$ is amenable \emph{if and only if} for any $\varepsilon>0$ and compact $K \subseteq \G$, there exists $f \in C_c(\G)$ with range in $[0,+\infty)$ such that for any $\gamma \in K$ we have
\[
\big| \sum_{\alpha \in \G_{\r(\gamma)}} f(\alpha) - 1 \big| < \varepsilon  \quad \mbox{and} \quad \sum_{\alpha \in \G_{\r(\gamma)}} |f(\alpha) - f(\alpha\gamma)| < \varepsilon.
\]
\end{lem}

Amenability for \'{e}tale groupoids enjoy similar permanence properties as in the case of groups. For example, open or closed subgroupoids of amenable \'{e}tale groupoids are amenable, and amenability is preserved under taking groupoid extensions. See \cite[Section 5]{ADR00} for details. Also recall that we have the following:

\begin{prop}[{\cite[Corollary 5.6.17]{BO08}}]\label{prop: etale amen}
Let $\G$ be a locally compact, Hausdorff, \'{e}tale and amenable groupoid. Then the natural quotient $C^*_{\max}(\G) \to C^*_r(\G)$ is an isomorphism.
\end{prop}

Now we recall the notion of a-T-menability for groupoids introduced by Tu \cite{Tu99}. Let $\G$ be a locally compact, Hausdorff and \'{e}tale groupoid. A continuous function $f: \G \to \RR$ is said to be of \emph{negative type} if 
\begin{enumerate}
 \item $f|_{\Gz}=0$;
 \item for any $\gamma\in \G$, $f(\gamma) = f(\gamma^{-1})$;
 \item Given $\gamma_1, \cdots, \gamma_n \in \G$ with the same range and $\lambda_1, \cdots, \lambda_n \in \RR$ with $\sum_{i=1}^n \lambda_i=0$, we have $\sum_{i,j} \lambda_i\lambda_jf(\gamma_i^{-1}\gamma_j) \leq 0$. 
\end{enumerate}
A continuous function $f: \G \to \RR$ is called \emph{locally proper} if for any compact subset $K \subseteq \Gz$, the restriction of $f$ on $\G_K^K$ is proper.

\begin{defn}[{\cite[Section 3.3]{Tu99}}]\label{defn: a-T-menable groupoids}
A locally compact, Hausdorff and \'{e}tale groupoid $\G$ is said to be \emph{a-T-menable} if there exists a continuous locally proper function $f: \G \to \RR$ of negative type on $\G$.
\end{defn}

Analogous to the case of groups, Tu \cite{Tu99} proved that a locally compact, $\sigma$-compact, Hausdorff and \'{e}tale groupoid $\G$ is a-T-menable \emph{if and only if} there exists a continuous field of Hilbert spaces over $\Gz$ with a proper affine action of $\G$. We also need the following significant result by Tu:

\begin{prop}[{\cite[Th\'{e}or\`{e}me 0.1]{Tu99}}]\label{prop: a-T-menable groupoids are K-amenable}
Let $\G$ be a locally compact, $\sigma$-compact, Hausdorff, \'{e}tale and a-T-menable groupoid. Then $\G$ is \emph{$K$-amenable}. In particular, the quotient map induces an isomorphism $K_\ast(C^*_{\max}(\G)) \to K_\ast(C^*_r(\G))$ for $\ast=0,1$.
\end{prop}

Finally, we record the following result for coarse groupoids:

\begin{prop}[{\cite[Theorem 5.3 and 5.4]{STY02}}]\label{prop:char for amen and a-T-men}
Let $(X,d)$ be a space and $G(X)$ be the associated coarse groupoid. Then:
\begin{enumerate}
 \item $X$ has Property A \emph{if and only if} $G(X)$ is amenable;
 \item $X$ can be coarsely embedded into Hilbert space \emph{if and only if} $G(X)$ is a-T-menable.
\end{enumerate}
\end{prop}

\section{Limit spaces and limit operators}\label{sec:limit sp. and op.}

In this section, we recall the theory of limit spaces and limit operators for metric spaces developed by \v{S}pakula and Willett in \cite{SW17}, which becomes an important tool for later use. 

Throughout the section, we always assume that $(X,d)$ is a space (see ``Convention'' in Section \ref{ssec:notions from coarse geometry}) and $G(X)$ is its coarse groupoid. We will freely use the notion of ultrafilters on $X$, and related materials are recalled in Appendix \ref{app:ultrafilters}.

\subsection{Limit spaces}\label{ssec:limit space} First recall that a function $t: D \to R$ with $D,R \subseteq X$ is called a \emph{partial translation} if $t$ is a bijection from $D$ to $R$, and $\sup_{x\in X}d(x,t(x))$ is finite. The \emph{graph} of $t$ is $\{(t(x),x): x\in D\}$, denoted by $\gr(t)$. It is well-known that each entourage $E$ on $X$ can be decomposed into finitely many graphs of partial translations (see, \emph{e.g.}, \cite[Lemma 4.10]{Roe03}) thanks to the bounded geometry of $X$.

\begin{defn}[{\cite[Definition 3.2 and 3.6]{SW17}}]\label{defn:limit space compatible}
Fix an ultrafilter $\omega \in \beta X$. A partial translation $t:D \to R$ on $X$ is \emph{compatible with $\omega$} if $\omega(D)=1$. In this case, regarding $t$ as a function from $D$ to $\beta X$, we define the following thanks to Lemma \ref{lem:ultralimit}:
\[
t(\omega):=\lim_\omega t \in \beta X.
\]
In other words, consider the extension $\overline{t}: \overline{D} \to \overline{R}$ then $t(\omega) = \overline{t}(\omega)$.

An ultrafilter $\alpha \in \beta X$ is \emph{compatible with $\omega$} if there exists a partial translation $t$ compatible with $\omega$ and $t(\omega)=\alpha$. Denote by $X(\omega)$ the collection of all ultrafilters on $X$ compatible with $\omega$. A \emph{compatible family for $\omega$} is a collection of partial translations $\{t_\alpha\}_{\alpha \in X(\omega)}$ such that each $t_\alpha$ is compatible with $\omega$ and $t_\alpha(\omega)=\alpha$.
\end{defn}

Fix an ultrafilter $\omega$ on $X$, and a compatible family $\{t_\alpha\}_{\alpha \in X(\omega)}$. Define a function $d_{\omega}: X(\omega) \times X(\omega) \to [0,\infty)$ by
$$d_{\omega}(\alpha, \beta):=\lim_{x\to \omega}d(t_\alpha(x),t_\beta(x)).$$
It is shown in \cite[Proposition 3.7]{SW17} that $d_\omega$ is a uniformly discrete metric of bounded geometry on $X(\omega)$ which does not depend on the choice of $\{t_{\alpha}\}$.

This leads to the following:

\begin{defn}[{\cite[Definition 3.8]{SW17}}]\label{defn:limit space}
For each non-principal ultrafilter $\omega$ on $X$, the metric space $(X(\omega), d_\omega)$ is called the \emph{limit space} of $X$ at $\omega$, which is a space in the sense of ``Convention'' in Section \ref{ssec:notions from coarse geometry}.
\end{defn}

It is shown in \cite[Proposition 3.9]{SW17} that for any $\alpha \in X(\omega)$, we have $X(\alpha) = X(\omega)$ as metric spaces. Also note that when $\omega$ is principal, \emph{i.e.}, $\omega \in X$, then it is clear that $(X(\omega), d_\omega) = (X,d)$. 

We recall the following result, which reveals that the local geometry of $X$ can be recaptured by those of the limit spaces.

\begin{prop}[{\cite[Proposition 3.10]{SW17}}]\label{prop:local geometry}
Let $\omega$ be a non-principal ultrafilter on $X$, and $\{t_\alpha:D_\alpha \to R_\alpha\}$ a compatible family for $\omega$. Then for each finite $F \subseteq X(\omega)$, there exists a subset $Y \subseteq X$ with $\omega(Y)=1$ such that for each $y\in Y$, there is a finite subset $G(y) \subseteq X$ such that the map
$$f_y: F \to G(y), \alpha \mapsto t_\alpha(y)$$
is a surjective isometry. Such a collection $\{f_y\}_{y\in Y}$ is called a \emph{local coordinate systerm } for $F$, and the maps $f_y$ are called \emph{local coordinates}.

Furthermore, if $F$ is a metric ball $B(\omega,r)$, then there exist $Y \subseteq X$ with $\omega(Y)=1$ and a local coordinate system $\{f_y: F \to G(y)\}_{y\in Y}$ such that each $G(y)$ is the ball $B(y,r)$.
\end{prop}

As shown in \cite[Appendix C]{SW17}, limit spaces can be described in terms of the coarse groupoid $G(X)$:

\begin{lem}[{\cite[Lemma C.3]{SW17}}]\label{lem:limit space via coarse groupoids}
Given a non-principal ultrafilter $ \omega \in \beta X$, the map
$$F : X(\omega) \rightarrow G(X)_\omega, \quad \alpha \mapsto (\alpha, \omega)$$
is a bijection. Hence $X(\omega)$ is the smallest invariant subset of $\beta X$ containing $\omega$. Here we consider $G(X)$ as a subset of $\beta X \times \beta X$, as explained in Section \ref{ssec:coarse groupoid}.
\end{lem}

Consequently, we obtain the following:

\begin{cor}\label{cor:char for G(X)}
As a set, we have
\[
G(X) = \big(X \times X\big) \sqcup \bigcup_{\omega \in \partial_\beta X} \big( X(\omega) \times X(\omega) \big).
\]
\end{cor}


Now we would like to provide a quantitative version for Corollary \ref{cor:char for G(X)}. First we record the following observation, whose proof is straightforward and almost identical to that of Lemma \ref{lem:limit space via coarse groupoids} (originally from \cite[discussion in 10.18-10.24]{Roe03}):

\begin{lem}\label{lem:compatibility via coarse groupoid}
Let $t: D \to R$ be a partial translation on $X$. Then we have:
\[
\overline{\gr(t)}^{\beta X \times \beta X} = \gr(t) \sqcup \bigcup_{\omega \in \partial_\beta X} \big\{(\alpha, \omega): \omega (D) =1 \text{ and } \alpha = t(\omega)\big\}.
\]
\end{lem}

In general, we have the following:

\begin{lem}\label{lem:closure of entourage}
For any $S\geq 0$, we have:
\[
\overline{E_S}^{\beta X \times \beta X} = E_S \sqcup \bigcup_{\omega \in \partial_\beta X} \big\{(\alpha, \gamma)\in X(\omega) \times X(\omega): d_\omega(\alpha, \gamma) \leq S \big\}.
\]
\end{lem}

\begin{proof}
As explained at the beginning of this subsection, we can decompose 
\[
E_S= \gr(t_1) \sqcup \cdots \sqcup \gr(t_N)
\]
where each $t_i: D_i \to R_i$ is a partial translation. Hence we have
\[
\overline{E_S} = \overline{\gr(t_1)} \cup \cdots \cup \overline{\gr(t_N)}.
\]
Applying Lemma \ref{lem:compatibility via coarse groupoid}, we obtain that $\overline{E_S}$ is contained in the right hand side in the statement. 

On the other hand, given $\omega \in \partial_\beta X$ and $\alpha, \gamma \in X(\omega)$ with $d_\omega(\alpha, \gamma) \leq S$, then we have $(X(\omega), d_\omega) = (X(\gamma),d_\gamma)$. Take a partial translation $t: D \to R$ such that $\gamma(D) =1$ and $\alpha = t(\gamma)$. Note that
\[
d_\gamma(\alpha, \gamma) = \lim_{x \to \gamma} d(t(x),x) \leq S.
\]
Hence for $D':=\{x\in D: d(t(x),x) \leq S\}$, we have $\gamma(D')=1$. Consider the restriction of $t$ on $D'$, denoted by $t'$. Then $t'$ is also a partial translation and $\alpha = t'(\gamma)$. By Lemma \ref{lem:compatibility via coarse groupoid}, we obtain that $(\alpha, \gamma) \in \overline{\gr(t')}$, which is contained in $\overline{E_S}$ as desired.
\end{proof}

Now we compute concrete examples of limit spaces. First we recall the case of groups from \cite[Appendix B]{SW17}. Let $\Gamma$ be a countable discrete group, equipped with a left-invariant bounded geometry and strongly discrete metric $d$. For each $g\in \Gamma$, denote 
\[
\rho_g: \Gamma \to \Gamma, h \mapsto hg
\]
the right translation map. Each $\rho_g$ is a partial translation with full domain, and hence is compatible with every $\omega \in \beta \Gamma$. Moreover, we have the following:

\begin{lem}[{\cite[Lemma B.1]{SW17}}]\label{lem:limit space for group case}
For each non-principal ultrafilter $\omega \in \beta \Gamma$, the map
\[
b_\omega: \Gamma \longrightarrow \Gamma(\omega), g \mapsto \rho_g(\omega)
\]
is an isometric bijection.
\end{lem}

Inspired by Lemma \ref{lem:limit space for group case}, we provide the following general method:

\begin{lem}\label{lem:limit space general case}
Let $\{t_\lambda: D_\lambda \to R_\lambda\}_{\lambda \in \Lambda}$ be a family of partial translations on $X$ satisfying the following: for each $S>0$, there exists a finite subset $\Lambda_S \subseteq \Lambda$ such that $\gr(t_{\lambda}) \cap \gr(t_\mu)$ is finite for $\lambda \neq \mu$ in $\Lambda_S$ and $E_S \setminus \big( \bigcup_{\lambda \in \Lambda_S} \gr(t_\lambda) \big)$ is finite. Then for any non-principal ultrafilter $\omega$ on $X$ and $\alpha \in X(\omega)$, there exists $\lambda \in \Lambda$ such that $\omega(D_\lambda)=1$ and $\alpha = t_\lambda(\omega)$.
\end{lem}

\begin{proof}
By definition, we assume that $\alpha = t(\omega)$ for some partial translation $t: D \to R$ with $\omega(D)=1$. For $\lambda \in \Lambda$, set $\widetilde{D}_\lambda:=\{x\in D \cap D_\lambda: t(x) = t_\lambda(x)\}$. Choose $S>0$ such that $\gr(t) \subseteq E_S$, and hence there exists a finite subset $F \subseteq E_S$ such that
\[
\gr(t) \subseteq \big( \bigcup_{\lambda \in \Lambda_S} \gr(t_\lambda) \big) \sqcup F.
\]
This implies that $D \subseteq \big( \bigcup_{\lambda \in \Lambda_S} \widetilde{D}_\lambda \big) \sqcup F'$ for some finite $F' \subseteq X$. Since $\omega$ is non-principal, $\Lambda_S$ is finite and $\widetilde{D}_\lambda \cap \widetilde{D}_\mu$ is finite for any $\lambda, \mu \in \Lambda_S$, there exists a unique $\lambda \in \Lambda_S$ such that $\omega(\widetilde{D}_\lambda)=1$. This implies that $\omega(D_\lambda)=1$ and $\alpha = t_\lambda(\omega)$, which concludes the proof.
\end{proof}

Back to the case of the group $\Gamma$, the set $\{\rho_g: \Gamma \to \Gamma\}_{g\in \Gamma}$ satisfies the condition in Lemma \ref{lem:limit space general case}, and hence the map $b_\omega$ in Lemma \ref{lem:limit space for group case} is surjective. It is straightforward to check that $b_\omega$ is isometric, which recovers the proof for Lemma \ref{lem:limit space for group case}.

\begin{ex}\label{eg:limit spaces for N}
Consider $X=\NN$ with the usual metric. For each $k \in \ZZ$ with $k \geq 0$, define a partial translation
\[
\rho_k: \NN \longrightarrow \NN, \quad n \mapsto n+k.
\]
For $k\in \ZZ$ with $k<0$, define a partial translation
\[
\rho_k: [-k, \infty) \cap \NN \longrightarrow \NN, \quad n \mapsto n+k.
\]
Then it is clear that the set $\{\rho_k\}_{k\in \ZZ}$ satisfies the condition in Lemma \ref{lem:limit space general case}. Now for a non-principal ultrafilter $\omega$ on $\NN$, consider the map 
\[
b_\omega: \ZZ \longrightarrow \NN(\omega), \quad k \mapsto \rho_k(\omega).
\]
Note that for $k,l\in \ZZ$, we have
\[
d_\omega(\rho_k(\omega), \rho_l(\omega)) = \lim_{n\to \omega} |(k+n) - (l+n)| = |k-l|,
\]
which implies that $b_\omega$ is isometric. Moreover, Lemma \ref{lem:limit space general case} shows that $b_\omega$ is surjective. Therefore, every limit space of $\NN$ is isometric to $\ZZ$. This provides a detailed proof for \cite[Example 3.14(2)]{SW17}.
\end{ex}

Similar to the analysis in Example \ref{eg:limit spaces for N}, we can also apply Lemma \ref{lem:limit space general case} to obtain proofs for \cite[Example 3.14(3)-(5)]{SW17}. Details are left to readers.

\subsection{Limit operators}

Now we recall the notion of limit operators for metric spaces introduced by \v{S}pakula and Willett:

\begin{defn}[{\cite[Definition 4.4]{SW17}}]\label{defn:limit operator}
For a non-principal ultrafilter $\omega$ on $X$, fix a compatible family $\{t_\alpha\}_{\alpha \in X(\omega)}$ for $\omega$ and let $T \in C^*_u(X)$. The \emph{limit operator of $T$ at $\omega$}, denoted by $\Phi_\omega(T)$, is an $X(\omega)$-by-$X(\omega)$ indexed matrix defined by
\[
\Phi_\omega(T)_{\alpha\gamma}:=\lim_{x \to \omega}T_{t_{\alpha}(x)t_{\gamma}(x)} \quad \text{for} \quad \alpha, \gamma \in X(\omega).
\]
\end{defn}

It was studied in \cite[Chapter 4]{SW17} that the above definition does not depend on the choice of the compatible family $\{t_\alpha\}_{\alpha \in X(\omega)}$ for $\omega$. Furthermore, the limit operator $\Phi_\omega(T)$ is indeed a bounded operator on $\ell^2(X(\omega))$, and belongs to the uniform Roe algebra $C^*_u(X(\omega))$.

Recall from Lemma \ref{lem: unif. Roe belongs to C0} that for an operator $T \in C^*_u(X)$, the continuous extension $\overline{T} \in C_0(G(X))$. We have the following, which was implicitly mentioned in the proof of \cite[Lemma C.3]{SW17}.

\begin{lem}\label{lem:limit operator using extension}
For a non-principal ultrafilter $\omega$ on $X$ and $T \in C^*_u(X)$, we have 
\[
\Phi_\omega(T)_{\alpha\gamma} = \overline{T}(\alpha, \gamma) \quad \text{for} \quad \alpha, \gamma \in X(\omega).
\]
\end{lem}

\begin{proof}
Choose partial translations $t_\alpha, t_\gamma$ compatible with $\alpha, \gamma$ such that $t_\alpha(\omega) = \alpha$, $t_\gamma(\omega) = \gamma$. By definition, we have
\[
\Phi_\omega(T)_{\alpha\gamma} = \lim_{x \to \omega}T_{t_{\alpha}(x)t_{\gamma}(x)} = \lim_{x \to \omega}T(t_{\alpha}(x),t_{\gamma}(x)) = \overline{T}(\alpha, \gamma)
\]
where the last equality comes from the discussion before Lemma \ref{lem: unif. Roe belongs to C0}.
\end{proof}

%
%
%

Since the limit operator $\Phi_\omega(T)$ contains the information of the asymptotic behaviour of $T$ ``in the $\omega$-direction'', we introduce the following:

\begin{defn}\label{defn:vanish at infinity}
For an $\omega\in \partial_\beta X$, we say that an operator $T \in C^*_u(X)$ is \emph{locally invisible (\emph{or} vanishes) in the $\omega$-direction} if $\Phi_\omega(T)=0$. For a subset $V \subseteq \partial_\beta X$, we say that $T$ is \emph{locally invisible (\emph{or} vanishes) in the $V$-direction} if $\Phi_\omega(T)=0$ for any $\omega \in V$.
\end{defn}

Finally, we recall from Proposition \ref{prop: groupoid char for uniform Roe alg} that there is a $C^*$-isomorphism $\Theta: C^*_r(G(X)) \to C^*_u(X)$. This allows us to relate limit operators to left regular representations of $C^*_r(G(X))$:

\begin{lem}[{\cite[Lemma C.3]{SW17}}]\label{lem:limit operators via regular representation}
For a non-principal ultrafilter $\omega$ on $X$, let $W_\omega: \ell^2(G(X)_\omega) \to \ell^2(X(\omega))$ be the unitary representation induced by $F$ in Lemma \ref{lem:limit space via coarse groupoids}. Then we have the following commutative diagram:
\begin{displaymath}
    \xymatrix@=3em{
        C^*_r(G(X)) \ar[r]^-{\textstyle \lambda_\omega} \ar[d]_-{\textstyle\Theta}^-{\textstyle\cong} & \B(\ell^2(G(X)_\omega)) \ar[d]^-{\textstyle \mathrm{Ad}_{W_\omega}}_-{\textstyle\cong} \\
        C^*_u(X) \ar[r]^-{\textstyle \Phi_\omega} & \B(\ell^2(X(\omega))),}
\end{displaymath}
where $\lambda_\omega$ is the left regular representation from (\ref{reduced algebra defn}).
\end{lem}

\section{Geometric ideals}\label{sec:geometric ideals}

In this section, we recall the notion of geometric ideals, which was originally introduced by the first-named author in \cite{Wan07} (see also \cite{CW04}). 
Throughout the section, let $X$ be a space in the sense of ``Convention'' in Section \ref{ssec:notions from coarse geometry}.

\begin{defn}[{\cite[Definition 3.1, 3.3]{CW04}}]\label{defn: epsilon supp}
For an operator $T \in C^*_u(X)$ and $\varepsilon>0$, the \emph{$\varepsilon$-support of $T$} is defined to be
\[
\supp_\varepsilon(T):=\{(x,y)\in X \times X: |T(x,y)| \geq \varepsilon\}.
\]
Also define the \emph{$\varepsilon$-truncation of $T$} to be
\[
T_\varepsilon(x,y):=
\begin{cases}
~T(x,y), & \mbox{if~} |T(x,y)| \geq \varepsilon; \\
~0, & \mbox{otherwise}.
\end{cases}
\]
\end{defn}

It is clear that $\supp(T_\varepsilon) = \supp_\varepsilon(T)$. We also record the following elementary result for later use. The proof is straightforward, hence omitted.
\begin{lem}\label{lem: support containment}
Given $T \in C^*_u(X)$ and $\varepsilon>0$, we have
\[
\overline{\supp(T_\varepsilon)} \subseteq \{\tilde{\omega} \in \beta(X \times X): |\overline{T}(\tilde{\omega})| \geq \varepsilon\} \subseteq  \overline{\supp(T_{\varepsilon/2})}.
\]
\end{lem}

%

The following is a key result in \cite{CW04}:

\begin{prop}[{\cite[Theorem 3.5]{CW04}}]\label{prop: Thm 3.5 in CW04}
Let $I$ be an ideal in the uniform Roe algebra $C^*_u(X)$. For each $T \in I$ and $\varepsilon>0$, we have $T_\varepsilon \in I \cap \CC_u[X]$. Moreover, we have
\[
\overline{I \cap \CC_u[X]} = \overline{\{T_\varepsilon: T \in I, \varepsilon>0\}}
\]
where the closures are taken with respect to the operator norm.
\end{prop}

Now we recall the notion of geometric ideals from \cite{Wan07}:

\begin{defn}\label{defn: geometric ideals}
An ideal $I$ in the uniform Roe algebra $C^*_u(X)$ is called \emph{geometric} if $I \cap \CC_u[X]$ is dense in $I$.
\end{defn}

In \cite{CW04}, Chen and the first-named author provide a full description for geometric ideals in $C^*_u(X)$ in terms of invariant open subsets of $G(X)^{(0)}=\beta X$.
To outline their work, let us start with the following elementary observation:


\begin{lem}\label{lem: inv. open contains X}
Let $U$ be a non-empty invariant open subset of $\beta X$, then $X \subseteq U$. 
\end{lem}

\begin{proof}
Since $X$ is dense in $\beta X$ and $U$ is open and non-empty, we obtain that $U \cap X \neq \emptyset$. Take an $x \in U \cap X$, then the pair $(x,y) \in G(X)$ for any $y\in X$. Thanks to the invariance of $U$, we obtain that $y\in U$. This implies that $X \subseteq U$.
\end{proof}

Given an invariant open subset $U \subseteq \beta X$, denote $G(X)_U:=G(X) \cap s^{-1}(U)$. Following \cite{CW04}, we define
\begin{align*}
I_c(U):&=\{f\in C_c(G(X)): f(\tilde{\omega})=0 \mbox{~for~any~} \tilde{\omega} \notin G(X)_U\}\\
&= \{T \in \CC_u[X]: \overline{T}(\tilde{\omega})=0 \mbox{~for~any~} \tilde{\omega} \notin G(X)_U\}.
\end{align*}
Obviously, $I_c(U)$ is a two-sided ideal in $C_c(G(X))$. Denote its closure in $C^*_r(G(X))$ by $I(U)$, which is a geometric ideal in $C^*_r(G(X)) \cong C^*_u(X)$ from the definition (see also \cite[Lemma 5.1]{CW04}). This leads to the following:

\begin{defn}\label{defn:geometric ideal}
For an invariant open subset $U \subseteq \beta X$, the ideal $I(U)$ is called the \emph{geometric ideal associated to $U$}.
\end{defn}

For later use, we record the following alternative description for the geometric ideal $I(U)$:

\begin{lem}\label{lem: groupoid C*-alg char for I(U)}
Let $U$ be an invariant open subset of $\beta X$. Then the ideal $I(U)$ is isomorphic to the reduced groupoid $C^*$-algebra $C^*_r(G(X)_U)$.
\end{lem}

\begin{proof}
This was implicitly contained in the proof of \cite[Proposition 5.5]{CW04}. For convenience to the readers, we include a proof here. By definition, $C^*_r(G(X)_U)$ is isomorphic to the norm closure of $C_c(G(X)_U)$ in $C^*_r(G(X))$. Note that
\[
C_c(G(X)_U) = \{T \in \CC_u[X]: \supp(\overline{T}) \subseteq G(X)_U\} \subseteq I_c(U). 
\]
On the other hand, for $T \in I_c(U)$ we have $T = \lim_{\varepsilon \to 0} T_\varepsilon$ since $T$ has finite propagation. Note that $\supp(\overline{T_\varepsilon}) \subseteq G(X)_U$, which implies $T_\varepsilon \in C_c(G(X)_U)$. Hence we obtain that $C_c(G(X)_U)$ and $I_c(U)$ have the same closure in $C^*_r(G(X))$, which concludes the proof.
\end{proof}

\begin{ex}\label{eg: compact ideal}
For $U=X$, it follows directly from definition that $G(X)_X = X \times X$. Hence combining Proposition \ref{prop: groupoid char for uniform Roe alg} and Lemma \ref{lem: groupoid C*-alg char for I(U)}, we obtain that the geometric ideal associated to $X$ is $I(X)=\K(\ell^2(X))$. On the other hand, for $U=\beta X$ it is clear that $I(\beta X) = C^*_u(X)$.
\end{ex}

Conversely, following \cite[Section 4 and 5]{CW04} we can associate an invariant open subset of $\beta X$ to any ideal in the uniform Roe algebra. More precisely, let $I$ be an ideal in the uniform Roe algebra $C^*_u(X)$. Define:
\begin{equation}\label{EQ: defn for U(I)}
U(I) :=\bigcup_{T \in I, \varepsilon >0} \overline{\r(\supp_\varepsilon(T))} = \bigcup_{T \in I \cap \CC_u[X], \varepsilon >0} \overline{\r(\supp_\varepsilon(T))},
\end{equation}
where the second equality follows directly from Proposition \ref{prop: Thm 3.5 in CW04}. Also \cite[Lemma 5.2]{CW04} implies that $U(I)$ is an invariant open subset of $\beta X$. Furthermore, as a special case of \cite[Theorem 6.3]{CW04}, we have the following:

\begin{prop}\label{prop: 1-1 btw ideals and open inv subsets}
For a space $(X,d)$, the map $I \mapsto U(I)$ provides an isomorphism between the lattice of all geometric ideals in $C^*_u(X)$ and the lattice of all invariant open subsets of $\beta X$, with the inverse map given by $U \mapsto I(U)$.
\end{prop}

Proposition \ref{prop: 1-1 btw ideals and open inv subsets} shows that geometric ideals in $C^*_u(X)$ can be fully determined by invariant open subsets of $\beta X$. In contrast, general ideals in $C^*_u(X)$ cannot be characterised merely by the associated subsets of $\beta X$. For example, direct calculations show that the associated invariant open subsets for the ideal $I_G$ defined in Section \ref{ssec:uniform Roe alg} and for the ideal of compact operators in $\B(\ell^2(X))$ are the same, both of which equal $X$ (see also Example \ref{eg: compact ideal} and \ref{eg: ghost ideal} below).

Hence as pointed out in Section \ref{sec:intro}, the study of the ideal structure for the uniform Roe algebra can be reduced to analyse the lattice (where the order is given by inclusion)
\[
\mathfrak{I}_U=\{I \mbox{~is~an~ideal~in~} C^*_u(X): U(I)=U\}
\]
in (\ref{EQ:set of ideals}) for each invariant open subset $U \subseteq \beta X$. The following result busts the smallest element in $\mathfrak{I}_U$:

\begin{prop}\label{cor: geometric ideals are smallest}
Let $U$ be an invariant open subset of $\beta X$.
Then the geometric ideal $I(U)$ is the smallest element in the lattice $\mathfrak{I}_U$ in (\ref{EQ:set of ideals}).
\end{prop}

The proof of Proposition \ref{cor: geometric ideals are smallest} follows directly from the following lemma:

\begin{lem}\label{prop: geometric ideals are smallest}
Let $(X,d)$ be a space and $I$ an ideal in $C^*_u(X)$. Then we have 
\[
I(U(I)) = \overline{I \cap \CC_u[X]},
\]
where the closure is taken in $C^*_u(X)$. Hence we have $I(U(I)) \subseteq I$.
\end{lem}

\begin{proof}
Denoting $\mathring{I}:=\overline{I \cap \CC_u[X]}$, it is clear that $\mathring{I} \cap \CC_u[X] = I \cap \CC_u[X]$. This implies that $\mathring{I}$ is a geometric ideal in $C^*_u(X)$, and hence $\mathring{I} = I(U(\mathring{I}))$ by Proposition \ref{prop: 1-1 btw ideals and open inv subsets}. By definition, we have
\[
U(\mathring{I}) =  \bigcup_{T \in \mathring{I} \cap \CC_u[X], \varepsilon >0} \overline{\r(\supp_\varepsilon(T))} = \bigcup_{T \in I \cap \CC_u[X], \varepsilon >0} \overline{\r(\supp_\varepsilon(T))} = U(I).
\]
Therefore, we obtain that $I(U(I)) = I(U(\mathring{I})) = \mathring{I}=\overline{I \cap \CC_u[X]}$ as required.
\end{proof}

We would like to recall another description for geometric ideals based on the notion of ideals in spaces introduced in \cite{CW04}. It has the advantage of playing within the given metric space, rather than going to the \emph{mysterious} Stone-\v{C}ech boundary, and hence will help us to step over several technical gaps in Section \ref{sec:partial A and partial ONL}.

\begin{defn}[{\cite[Definition 6.1]{CW04}}]\label{defn: ideals in space}
An \emph{ideal} in a space $(X,d)$ is a collection $\L$ of subsets of $X$ satisfying the following:
\begin{enumerate}
 \item if $Y \in \L$ and $Z \subseteq Y$, then $Z \in \L$;
 \item if $R\geq 0$ and $Y \in \L$, then $\Nd_R(Y) \in \L$;
 \item if $Y,Z \in \L$, then $Y \cup Z \in \L$.
\end{enumerate}
\end{defn}

For an ideal $\L$ in $X$, we define
\[
U(\L):= \bigcup_{Y \in \L} \overline{Y}^{\beta X}.
\]
Conversely, given an invariant open subset $U$ of $\beta X$, we define
\[
\L(U):=\{Y \subseteq X: \overline{Y}^{\beta X} \subseteq U\}.
\]
As a special case of \cite[Theorem 6.3]{CW04}, we have the following:

\begin{prop}\label{prop: 1-1 btw ideals in spaces and open inv subsets}
For a space $(X,d)$, the map $\L \mapsto U(\L)$ provides an isomorphism between the lattice of all ideals in $X$ and the lattice of all invariant open subsets of $\beta X$, with the inverse map given by $U \mapsto \L(U)$.
\end{prop}

Combining Proposition \ref{prop: 1-1 btw ideals and open inv subsets} and \ref{prop: 1-1 btw ideals in spaces and open inv subsets}, we obtain an isomorphism between the lattice of all ideals in $X$ and the lattice of all geometric ideals in $C^*_u(X)$. Direct calculation shows (see also \cite[Theorem 6.4]{CW04}):
\begin{equation}\label{EQ: I(U(L))}
I(U(\L))=\overline{\{T \in \CC_u[X]: \supp(T) \subseteq Y \times Y \mbox{~for~some~}Y \in \L\}},
\end{equation}
where the closure is taken in $C^*_u(X)$.

Now we consider a special class of geometric ideals coming from subspaces. Given a subspace $A \subseteq X$, recall from \cite[Section 5]{HRY93} that there is an associated ideal $I_A$ in $C^*_u(X)$ whose $K$-theory is isomorphic to that of the uniform Roe algebra $C^*(A)$. More precisely, recall that an operator $T \in \B(\ell^2(X))$ is \emph{near $A$} if there exists $R>0$ such that $\supp(T) \subseteq \Nd_R(A) \times \Nd_R(A)$, and the ideal $I_A$ is defined to be the operator norm closure of all operators in $\CC_u[X]$ near $A$.

The ideal $I_A$ is called \emph{spatial} in \cite{CW04b} since it is related to a subspace in $X$. However, as shown in \cite[Example 2.1]{CW04b}, there exist non-spatial ideals in general. On the other hand, spatial ideals play an important role in the computation of the $K$-theory of Roe algebras via the Mayer-Vietoris sequence argument (see \cite{HRY93}).

To show that spatial ideals are geometric, we observe that the smallest ideal in $X$ containing $A$ is $\L_A:=\{Z \subseteq \Nd_R(A): R>0\}$. Hence applying Proposition \ref{prop: 1-1 btw ideals in spaces and open inv subsets}, we immediately obtain the following:


\begin{lem}\label{lem: invariant subset}
Let $A$ be a subset of $X$. Then the set
\begin{equation}\label{EQ:UA}
U_A:=U(\L_A) = \bigcup_{R>0} \overline{\Nd_R(A)} 
\end{equation}
is an invariant open subset of $\beta X$. Moreover, if $U$ is an invariant open subset of $\beta X$ containing $A$, then $U_A \subseteq U$.
\end{lem}




Consequently, combining with (\ref{EQ: I(U(L))}) we reach the following:

\begin{cor}\label{cor: ideals to subspace}
Let $A$ be a subset of $X$, then $I(U_A)=I_A$. Hence the spatial ideal $I_A$ is geometric.
\end{cor}


For later use, we record the following result concerning the set $U_A$ defined in (\ref{EQ:UA}). Recall from Corollary \ref{cor:subspace in SC comp.} that for a subset $Z \subseteq X$, the closure $\overline{Z}$ in $\beta X$ is homeomorphic to $\beta Z$. 

\begin{lem}\label{lem:dec for groupoids UA}
Let $A$ be a subset of $X$. Then we have:
\[
G(X)_{U_A} = \bigcup_{R>0} G(\Nd_R(A)).
\]
\end{lem}

\begin{proof}
By definition, we have $G(X)_{U_A} = \bigcup_{R>0} G(X)_{\overline{\Nd_R(A)}}$. Note that for each $R>0$ and $(\alpha, \omega) \in G(X)_{\overline{\Nd_R(A)}}$, we have $\omega \in \overline{\Nd_R(A)}$ and there exists $S>0$ such that $(\alpha, \omega) \in \overline{E_S}$ due to Lemma \ref{lem:closure of entourage}. Hence the pair $(\alpha, \omega)$ in $\overline{\Nd_{R+S}(A)} \times \overline{\Nd_{R+S}(A)}$ belongs to $G(\Nd_{R+S}(A))$, which concludes the proof.
\end{proof}


To end this section, we remark that for a given ideal $I$ in $C^*_u(X)$, it is usually hard to compute the associated $U(I)$ directly from definition. However, this is always achievable for principal ideals:

\begin{lem}\label{lem: U(I) for principal ideals}
Let $I=\langle T \rangle$ be the principal ideal in $C^*_u(X)$ generated by $T \in C^*_u(X)$. Denote
\[
U:= \bigcup_{\varepsilon >0, R>0} \overline{\Nd_R(\r(\supp_\varepsilon(T)))}.
\]
Then $U$ is an invariant open subset of $\beta X$, and we have $U(I)=U$.
\end{lem}

\begin{proof}
By Lemma \ref{lem: invariant subset}, it is clear that $U$ is an invariant open subset of $\beta X$. By (\ref{EQ: defn for U(I)}), $U(I)$ contains $\overline{\r(\supp_\varepsilon(T))}$ for any $\varepsilon>0$. Since $U(I)$ is invariant, we obtain that $U(I)$ contains $U$ again by Lemma \ref{lem: invariant subset}.

On the other hand, we consider $S=\sum_{i=1}^n a_iTb_i$ where $a_i, b_i$ are non-zero with supports being partial translations contained in $E_R$ for some $R>0$. Hence for any $\varepsilon>0$, we have
\[
\r\left(\supp_\varepsilon(S)\right) \subseteq \bigcup_{i=1}^n \r\left(\supp_{\frac{\varepsilon}{n}}(a_iTb_i)\right) \subseteq \bigcup_{i=1}^n \Nd_R\left(\r\left(\supp_{\frac{\varepsilon}{n\|a_i\|\cdot\|b_i\|}}(T)\right)\right),
\]
which implies that $\overline{\r\left(\supp_\varepsilon(S)\right)} \subseteq U$. Note that operators of the form $\sum_{i=1}^n a_iTb_i$ as above are dense in $I$. Hence for a general element $\tilde{S} \in I$ and $\varepsilon>0$, there exists $S=\sum_{i=1}^n a_iTb_i$ where $a_i, b_i$ are non-zero with supports being partial translations such that $\|\tilde{S} - S\| < \varepsilon/2$. Hence $\supp_{\varepsilon}(\tilde{S}) \subseteq \supp_{\varepsilon/2}(S)$, which concludes the proof.
\end{proof}

\section{Ghostly ideals}\label{sec:ghostly ideals}

In the previous section, we observe that for an invariant open subset $U$ of $\beta X$, the associated geometric ideal $I(U)$ is the smallest element in the lattice
\[
\mathfrak{I}_U=\{I \mbox{~is~an~ideal~in~} C^*_u(X): U(I)=U\}.
\]
In this section, we would like to explore the largest element in $\mathfrak{I}_U$. As revealed in Section \ref{sec:intro}, a natural idea is to include all operators $T \in C^*_u(X)$ sitting in some ideal $I$ with $U(I)=U$, which leads to the following:

\begin{defn}\label{defn: ghostly ideals}
For a space $(X,d)$ and an invariant open subset $U$ of $\beta X$, denote
\[
\tilde{I}(U):=\{T \in C^*_u(X): \overline{\r(\supp_\varepsilon(T))} \subseteq U \mbox{~for~any~}\varepsilon>0\}.
\]
We call $\tilde{I}(U)$ the \emph{ghostly ideal associated to $U$}.
\end{defn}

The terminology will become clear later (see Proposition \ref{prop: geometric char for ghostly ideals} and Remark \ref{rem: terminology exp. for ghostly ideals} below). First let us verify that $\tilde{I}(U)$ is indeed an ideal in $C^*_u(X)$.

\begin{lem}\label{lem: ghostly ideals are ideals}
For an invariant open subset $U$ of $\beta X$, $\tilde{I}(U)$ is an ideal in $C^*_u(X)$.
\end{lem}

\begin{proof}
For $T,S \in C^*_u(X)$ and $\varepsilon>0$, we have $\supp_\varepsilon(T+S) \subseteq \supp_{\varepsilon/2}(T) \cup \supp_{\varepsilon/2}(S)$.
It follows that $\tilde{I}(U)$ is a linear space in $C^*_u(X)$. Given $T \in \tilde{I}(U)$ and $\varepsilon>0$, note that $\r(\supp_\varepsilon(T^*)) = \s(\supp_\varepsilon(T))$ and hence $T^* \in \tilde{I}(U)$ since $U$ is invariant. Now given a sequence $\{T_n\}_n$ in $\tilde{I}(U)$ converging to $T \in C^*_u(X)$ and an $\varepsilon>0$, there exists $n\in \NN$ such that $\supp_\varepsilon(T) \subseteq \supp_{\varepsilon/2}(T_n)$. Hence we obtain that $T \in \tilde{I}(U)$.

Finally, given $T \in \tilde{I}(U)$ and $S \in C^*_u(X)$ we need show that $TS$ and $ST$ belong to $\tilde{I}(U)$. Note that $ST=(T^*S^*)^*$ and $\tilde{I}(U)$ is closed under taking the $\ast$-operation, hence it suffices to show that $TS \in \tilde{I}(U)$. 
Since $\tilde{I}(U)$ is closed in $C^*_u(X)$ with respect to the operator norm, it suffices to consider the case that $S \in \CC_u[X]$.
We can further assume that $S$ is a partial translation since $\tilde{I}(U)$ is closed under taking addition. In this case, we have $\r(\supp_\varepsilon(TS)) \subseteq \r(\supp_{\varepsilon/\|S\|}(T))$ for any $\varepsilon>0$. Hence we conclude the proof.
\end{proof}

Using the language of ideals in $X$, we record that for an ideal $\L$ in $X$ and $T \in C^*_u(X)$, then $T \in \tilde{I}(U(\L))$ \emph{if and only if} for any $\varepsilon>0$ there exist $R>0$ and $Y \in \L$ such that for any $(x,y) \notin E_R \cap (Y \times Y)$ then $|T(x,y)| < \varepsilon$.

The following result shows that the ghostly ideal $\tilde{I}(U)$ is indeed the largest element in the lattice $\mathfrak{I}_U$:

\begin{lem}\label{lem: ghostly ideas are largest}
Given an invariant open subset $U$ of $\beta X$, we have the following:
\begin{enumerate}
 \item for an ideal $I$ in $C^*_u(X)$ with $U(I)=U$, then $I \subseteq \tilde{I}(U)$.
 \item $U(\tilde{I}(U))=U$.
\end{enumerate}
\end{lem}

\begin{proof}
(1). Given $T \in I$, the condition $U(I)=U$ implies that $\overline{\r(\supp_\varepsilon(T))} \subseteq U$ for each $\varepsilon>0$. Hence by definition, we have $T \in \tilde{I}(U)$. 

(2). Note that
\[
U(\tilde{I}(U)) = \bigcup_{T \in \tilde{I}(U), \varepsilon >0} \overline{\r(\supp_\varepsilon(T))} \subseteq U.
\]
On the other hand, Proposition \ref{prop: 1-1 btw ideals and open inv subsets} shows that $U(I(U))=U$, which implies that $\tilde{I}(U) \supseteq I(U)$ thanks to (1). Hence we have $U(\tilde{I}(U)) \supseteq U(I(U)) =U$ again by Proposition \ref{prop: 1-1 btw ideals and open inv subsets}, which concludes the proof.
\end{proof}


Combining with Proposition \ref{cor: geometric ideals are smallest}, we reach the following desired result:
\begin{thm}\label{prop: ideals containment}
Let $(X,d)$ be a space as in Section \ref{ssec:notions from coarse geometry}, and $U$ be an invariant open subset of $\beta X$. Then any ideal $I$ in $C^*_u(X)$ with $U(I)=U$ sits between $I(U)$ and $\tilde{I}(U)$. More precisely, the geometric ideal $I(U)$ is the smallest element while the ghostly ideal $\tilde{I}(U)$ is the largest element in the lattice $\mathfrak{I}_U$ in  (\ref{EQ:set of ideals}).
\end{thm}

As mentioned in Section \ref{sec:intro}, Theorem \ref{prop: ideals containment} draws the border of the lattice $\mathfrak{I}_U$. Therefore, once we can bust every ideal between $I(U)$ and $\tilde{I}(U)$ for each invariant open subset $U \subseteq \beta X$, then we will obtain a full description for the ideal structure of the uniform Roe algebra $C^*_u(X)$ (see Question \ref{Q1}).

Now we aim to provide a geometric description for ghostly ideals, which helps to explain the terminology. Let us start with an easy example.

\begin{ex}\label{eg: ghost ideal}
Taking $U=X$, then $\tilde{I}(X)$ is the ideal $I_G$ defined in Section \ref{ssec:uniform Roe alg}. Indeed, $T \in \tilde{I}(X)$ if and only if for any $\varepsilon>0$, $\r(\supp_\varepsilon(T))$ is finite. This is equivalent to that $T \in C_0(X \times X)$ since $T \in C^*_u(X)$. On the other hand, it is clear that $\tilde{I}(\beta X) = C^*_u(X)$.
\end{ex}

More generally, we have the following result. Note that for any invariant open subset $U$ of $\beta X$, $G(X)_U$ is an open subset of $\beta(X \times X)$. Hence both $C_0(G(X)_U)$ and $I_c(U)$ can be regarded as subalgebras in $C(\beta(X \times X)) \cong \ell^\infty(X \times X)$.

\begin{prop}\label{prop: geometric char for ghostly ideals}
For an invariant open subset $U \subseteq \beta X$ and $T \in C^*_u(X)$, the following are equivalent:
\begin{enumerate}
 \item $T \in \tilde{I}(U)$;
 \item $\overline{T} \in C_0(G(X)_U)$;
 \item $\overline{T} \in \overline{I_c(U)}^{\|\cdot\|_\infty}$;
 \item $T$ vanishes in the $(\beta X \setminus U)$-direction, \emph{i.e.}, $\Phi_\omega(T) = 0$ for any $\omega \in \beta X \setminus U$.
\end{enumerate}
\end{prop}

\begin{proof}
``(1) $\Rightarrow$ (2)'': By definition, for any $\varepsilon>0$ we have 
\[
\r(\supp(\overline{T_\varepsilon})) = \r(\overline{\supp(T_\varepsilon)}) = \overline{\r(\supp_\varepsilon(T))} \subseteq U.
\]
Consider the compact set $K:=\{\tilde{\omega} \in \beta(X \times X): |\overline{T}(\tilde{\omega})| \geq 2\varepsilon\}$. By Lemma \ref{lem: support containment}, we have $\r(K) \subseteq \r(\supp(\overline{T_\varepsilon})) \subseteq U$, which implies that $K\subseteq G(X)_U$. 
Moreover, we have $|\overline{T}(\tilde{\omega})|<2\varepsilon$ for any $\tilde{\omega} \in \beta(X \times X) \setminus K$, which concludes that $\overline{T} \in C_0(G(X)_U)$.

``(2) $\Rightarrow$ (1)'': Assume that $\overline{T} \in C_0(G(X)_U) \subseteq C(\beta(X \times X))$. Then for any $\varepsilon>0$ there exists a compact subset $K \subseteq G(X)_U$ such that for any $\tilde{\omega} \in \beta(X \times X) \setminus K$, we have $|\overline{T}(\tilde{\omega})| < \varepsilon$. This implies that $\{\tilde{\omega} \in \beta(X \times X): |\overline{T}(\tilde{\omega})| \geq \varepsilon\} \subseteq K$. Using Lemma \ref{lem: support containment}, we obtain that $\overline{\supp(T_\varepsilon)} \subseteq K$, which implies that $\overline{\r(\supp_\varepsilon(T))} \subseteq U$. Hence $T \in \tilde{I}(U)$.

``(2) $\Leftrightarrow$ (3)'': This is due to the fact that 
\[
\overline{I_c(U)}^{\|\cdot\|_\infty} = \{f \in C_0(G(X)): f(\tilde{\omega})=0 \text{ for } \tilde{\omega} \notin G(X)_U\}.
\]

``(3) $\Leftrightarrow$ (4)'': By Lemma \ref{lem:limit operator using extension}, (4) holds if and only if $\overline{T}(\alpha, \gamma) = 0$ for any $\alpha, \gamma \in X(\omega)$ and $\omega \in \beta X \setminus U$. Applying Lemma \ref{lem: unif. Roe belongs to C0} and Lemma \ref{lem:limit space via coarse groupoids}, this holds if and only if $\overline{T}(\tilde{\omega}) = 0$ whenever $\tilde{\omega} \notin G(X)_U$, which describes elements in $\overline{I_c(U)}^{\|\cdot\|_\infty}$. Hence we conclude the proof.
\end{proof}

%

Note that $G(X)_X = X \times X$. Hence as a direct corollary, we recover the following characterisation for ghost operators:

\begin{cor}[{\cite[Proposition 8.2]{SW17}}]\label{cor:char for ghost}
An operator $T \in C^*_u(X)$ is a ghost \emph{if and only if} $\Phi_\omega(T)=0$ for any non-principal ultrafilter $\omega$ on $X$.
\end{cor}

\begin{rem}\label{rem: terminology exp. for ghostly ideals}
In other words, Corollary \ref{cor:char for ghost} shows that a ghost in $C^*_u(X)$ is locally invisible in all directions. This suggests us to consider operators in $\tilde{I}(U)$ as ``partial'' ghosts, which clarifies the terminology of ``ghostly ideals''.
\end{rem}


As an application of Proposition \ref{prop: geometric char for ghostly ideals}, we now provide another description for ghostly ideals in terms of operator algebras, which will be used later in Section \ref{sec:geometric vs ghostly ideals}. Let us start with the short exact sequences studied in \cite{HLS02} (see also \cite[Section 2]{FW14}). 


Given an invariant open subset $U \subseteq \beta X$, notice that $U^c=\beta X \setminus U$ is also invariant. Denote by $G(X)_{U^c}:=G(X) \cap \s^{-1}(U^c)$ and clearly, we have a decomposition:
\[
G(X) = G(X)_U \sqcup G(X)_{U^c}.
\]
Note that $G(X)_U$ is open in $G(X)$, hence the above induces the following short exact sequence of $\ast$-algebras:
\begin{equation}\label{EQ: short exact seq. for alg.}
0 \longrightarrow C_c(G(X)_U) \longrightarrow C_c(G(X)) \longrightarrow C_c(G(X)_{U^c}) \longrightarrow 0
\end{equation}
where the map $C_c(G(X)_U) \longrightarrow C_c(G(X))$ is the inclusion and the map $C_c(G(X)) \longrightarrow C_c(G(X)_{U^c})$ is the restriction. 

We may complete the sequence (\ref{EQ: short exact seq. for alg.}) with respect to the maximal groupoid $C^*$-norms and obtain the following sequence:
\begin{equation}\label{EQ: short exact seq. for max. alg.}
0 \longrightarrow C^*_{\rm max}(G(X)_U) \longrightarrow C^*_{\rm max}(G(X)) \longrightarrow C^*_{\rm max}(G(X)_{U^c}) \longrightarrow 0,
\end{equation}
which is easy to check by definition to be automatically exact (see, \emph{e.g.}, \cite[Lemma 2.10]{MRW96}). We may also complete this sequence with respect to the reduced groupoid $C^*$-norms and obtain the following sequence:
\begin{equation}\label{EQ: short exact seq. for red. alg.}
0 \longrightarrow C^*_r(G(X)_U) \stackrel{i_U}{\longrightarrow} C^*_r(G(X)) \stackrel{q_U}{\longrightarrow} C^*_r(G(X)_{U^c}) \longrightarrow 0.
\end{equation}
By construction, $i_U$ is injective, $q_U$ is surjective and $q_U \circ i_U =0$. Also recall from Lemma \ref{lem: groupoid C*-alg char for I(U)} that $i_U(C^*_r(G(X)_U)) = I(U)$, the geometric ideal associated to $U$. However in general, (\ref{EQ: short exact seq. for red. alg.}) fails to be exact at the middle item. This is crucial in \cite{HLS02} to provide a counterexample to the Baum-Connes conjecture with coefficients.

More precisely when $U=X$, it is proved in \cite[Proposition 2.11]{Roe05} for the group case and \cite[Proposition 4.4]{FW14} for the Roe algebraic case (see also \cite[Proposition 8.2]{SW17}) that ${\rm Ker}(q_X)=I_G$, \emph{i.e.}, the ideal consisting of all ghost operators in $C^*_u(X)$. Hence from Example \ref{eg: compact ideal} and Example \ref{eg: ghost ideal}, the sequence (\ref{EQ: short exact seq. for red. alg.}) is exact for $U=X$ \emph{if and only if} $I(X) = \tilde{I}(X)$. 
More generally, we have the following:

\begin{prop}\label{prop: char for Ker qU}
Given an invariant open subset $U \subseteq \beta X$, the kernel of $q_U: C^*_r(G(X)) \to C^*_r(G(X)_{U^c})$ coincides with the ghostly ideal $\tilde{I}(U)$. Hence the sequence (\ref{EQ: short exact seq. for red. alg.}) is exact \emph{if and only if} $I(U) = \tilde{I}(U)$.
\end{prop}

\begin{proof}
It is easy to see that an operator $T \in C^*_u(X) \cong C^*_r(G(X))$ belongs to the kernel of $q_U$ if and only if $\lambda_\omega(T)=0$ for any $\omega \in \beta X \setminus U$, where $\lambda_\omega$ is the left regular representation from (\ref{reduced algebra defn}). Now Lemma \ref{lem:limit operators via regular representation} implies that $\Phi_\omega(T) = 0$ for any $\omega \in \beta X \setminus U$. Finally, we conclude the proof thanks to Proposition \ref{prop: geometric char for ghostly ideals}.
\end{proof}

We end this section with an illuminating example from \cite[Section 3]{Wan07} (see also \cite[Section 5]{FW14}), which is important to construct counterexamples to Baum-Connes type conjectures:

\begin{ex}\label{ex: Wan07}
Let $\{X_i\}_{i\in \NN}$ be a sequence of expander graphs or pertubed expander graphs (see \cite{Wan07} for the precise definition). Let $Y_{i,j}=X_i$ for all $j\in \NN$ and set $Y:=\bigsqcup_{i,j} Y_{i,j}$. We endow $Y$ with a metric $d$ such that it is the graph metric on each $Y_{i,j}$ and satisfies $d(Y_{i,j}, Y_{k,l}) \to \infty$ as $i+j+k+l \to \infty$.

Let $P_{i,j} \in \B(\ell^2(Y_{i,j}))$ be the orthogonal projection onto constant functions on $Y_{i,j}$, and we set $P$ to be the direct sum of $P_{i,j}$ in the strong operator topology. By the assumption on the expansion of $\{X_i\}_{i\in \NN}$, it is clear that $P \in C^*_u(Y)$. It is explained in \cite[Section 3]{Wan07} (see also \cite[Lemma 5.1]{FW14}) that $P$ is \emph{not} a ghost, \emph{i.e.}, $P \notin \tilde{I}(X)$. However intuitively, $P$ should vanish ``in the $i$-direction''. We will make it more precisely in the following.

Recall from \cite[Section 5.1]{FW14} that we have a surjective map $\beta Y \to \beta X \times \beta \NN$ induced by the bijection of $Y$ with $X \times \NN$ and the universal property of $\beta Y$. Define:
\[
f: \beta Y \longrightarrow \beta X \times \beta \NN \longrightarrow \beta X
\]
where the second map is just the projection onto the first coordinate. Denote $U=f^{-1}(X)$, which is open in $\beta Y$. Note that $U=\bigsqcup_i f^{-1}(X_i)$, where each $f^{-1}(X_i)$ is homeomorphic to $X_i \times \beta \NN$. On the other hand, note that
\[
U= \bigsqcup_i f^{-1}(X_i) = \bigcup_i \overline{\bigsqcup_j Y_{i,j}} = \bigcup_{\varepsilon >0} \overline{\r(\supp_\varepsilon(P))} = \bigcup_{\varepsilon >0, R>0} \overline{\Nd_R(\r(\supp_\varepsilon(P)))}.
\]
Hence it follows from Lemma \ref{lem: U(I) for principal ideals} that $U$ is invariant (comparing with \cite[Lemma 5.2]{FW14}), and $U(\langle P \rangle) = U$. (Note that $P \in \tilde{I}(U)$ was already implicitly proved in \cite[Theorem 5.5]{FW14}, thanks to Proposition \ref{prop: char for Ker qU}.) Since $U$ contains $Y$ as a proper subset, we reprove that $P$ is not a ghost. 

Moreover, it follows from Proposition \ref{prop: geometric char for ghostly ideals} that $P$ vanishes in the $(\partial_\beta Y \setminus U)$-direction. In particular, fixing an index $j_0 \in \NN$ and taking a sequence $\{x_i \in Y_{i,j_0}\}_{i\in \NN}$, we choose a cluster point $\omega \in \overline{\{x_i\}_i}$. It is clear that $\omega \notin U$. Intuitively speaking, this means that $P$ vanishes ``in the $i$-direction''. 

We remark that the first-named author proved in \cite[Section 3]{Wan07} that the principal ideal $\langle P \rangle$ cannot be decomposed into $I(U) + (I_G\cap \langle P \rangle)$, which provided a couterexample to the conjecture at the end of \cite{CW04}. Our explanation above reveals that the reason behind this counterexample is that the ghostly part of $\langle P \rangle$ could not be ``exhausted'' merely by ghostly elements associated to $X$ (rather than $U$). 

Finally, we remark that the groupoid $G(Y)_U$ also plays a key role in constructing a counterexample to the boundary coarse Baum-Connes conjecture introduced in \cite{FW14} (see Section 5.2 therein).
\end{ex}


\section{Maximal ideals}\label{sec:maximal ideals}

In this section, we would like to study maximal ideals in uniform Roe algebras using the tools developed in Section \ref{sec:ghostly ideals}. Throughout this section, let $(X,d)$ be a space as in Section \ref{ssec:notions from coarse geometry}.

\subsection{Minimal points in the boundary}\label{ssec:min points}

Recall from previous sections that ideals are closely related to invariant open subsets of the unit space $\beta X$. Hence we introduce the following:

\begin{defn}\label{defn:maximal subset}
An invariant open subset $U \subseteq \beta X$ is called \emph{maximal} if $U \neq \beta X$ and $U$ is not properly contained in any proper invariant open subset of $\beta X$. Similarly, an invariant closed subset $K \subseteq \beta X$ is called \emph{minimal} if $K \neq \emptyset$ and $K$ does not properly contain any non-empty invariant closed subset of $\beta X$.
\end{defn}

\begin{prop}\label{lem:maximal ideals using maximal open subsets}
For any maximal invariant open subset $U \subset \beta X$, the ghostly ideal $\tilde{I}(U)$ is a maximal ideal in the uniform Roe algebra $C^*_u(X)$. Conversely for any maximal ideal $I$ in $C^*_u(X)$, the associated invariant open subset $U(I)$ is maximal and we have $I=\tilde{I}(U(I))$.
\end{prop}

\begin{proof}
For any ideal $J$ in $C^*_u(X)$ containing $\tilde{I}(U)$, we have $U(J) \supseteq U(\tilde{I}(U)) = U$ by Lemma \ref{lem: ghostly ideas are largest}(2). Since $U$ is maximal, then either $U(J) = \beta X$ or $U(J) = U$. If $U(J)= \beta X$, it follows from Lemma \ref{prop: geometric ideals are smallest} that $J$ contains $I(U(J)) = C^*_u(X)$, which implies that $J=C^*_u(X)$. If $U(J) = U$, then it follows from Lemma \ref{lem: ghostly ideas are largest}(1) that $J \subseteq \tilde{I}(U(J)) = \tilde{I}(U)$, which implies that $J=\tilde{I}(U)$. This concludes that $\tilde{I}(U)$ is maximal.

Conversely for any maximal ideal $I$ in $C^*_u(X)$, we have $U(I) \neq \beta X$. For any open invariant subset $V \neq \beta X$ containing $U$, we have $I \subseteq \tilde{I}(U) \subseteq \tilde{I}(V)$ by Theorem \ref{prop: ideals containment}, and $\tilde{I}(V) \neq C^*_u(X)$. Hence due to the maximality of $I$, we obtain that $I = \tilde{I}(U) = \tilde{I}(V)$. This implies that $U=V$ and also $I = \tilde{I}(U)$ as required.
\end{proof}

Taking complements, we obtain the following:

\begin{cor}\label{cor:maximal ideals using minimal closed subsets}
For any minimal invariant closed subset $K \subseteq \beta X$, the ghostly ideal $\tilde{I}(\beta X \setminus K)$ is a maximal ideal in the uniform Roe algebra $C^*_u(X)$. Moreover, every maximal ideal in $C^*_u(X)$ arises in this form.
\end{cor}

Therefore, in order to describe maximal ideals in the uniform Roe algebra, it suffices to study minimal invariant closed subsets of the unit space $\beta X$. Recall from Lemma \ref{lem:limit space via coarse groupoids} that for each $\omega \in \partial_\beta X$, the limit space $X(\omega)$ is the smallest invariant subset of $\beta X$ containing $\omega$. However, note that $X(\omega)$ might not be closed in general.

\begin{defn}\label{defn:minimal point}
A point $\omega \in \partial_\beta X$ is called \emph{minimal} if the closure of the limit space $X(\omega)$ in $\beta X$ is minimal in the sense of Definition \ref{defn:maximal subset}.
\end{defn}

The following result is straightforward, hence we omit the proof. It suggests us to study minimal points in the boundary.

\begin{lem}\label{lem:property of minimal points}
For a minimal invariant closed subset $K \subseteq \beta X$, there exists a minimal point $\omega \in \partial_\beta X$ such that $K= \overline{X(\omega)}$. Conversely for any minimal point $\omega \in \partial_\beta X$, the set $\overline{X(\omega)}$ is minimal.
\end{lem}

One might wonder whether every $\omega \in \partial_\beta X$ is minimal. However, things become very complicated after taking closures and we will show later that this does not hold even in the case of $X=\ZZ$. Firstly, we notice the following:

\begin{lem}\label{lem:existence of min point}
Let $K$ be an invariant closed subset of $\beta X$. Then $K$ contains a minimal point. In particular, for any $\omega \in \partial_\beta X$ there exists a minimal point $\omega'$ such that $\omega' \in \overline{X(\omega)}$.
\end{lem}

\begin{proof}
This follows directly from the Zorn's lemma together with the fact that $\beta X$ is compact. Details are left to readers.
\end{proof}

Consequently, we obtain:

\begin{cor}\label{cor:char for min point}
For $\omega \in \partial_\beta X$, $\omega$ is minimal \emph{if and only if} for any $\alpha \in \overline{X(\omega)}$, we have $\overline{X(\alpha)} = \overline{X(\omega)}$. Writing $\overline{X(\omega)} = \bigsqcup_{\lambda \in \Lambda}X(\omega_\lambda)$ for certain $\omega_\lambda \in \partial_\beta X$, then $\omega$ is minimal \emph{if and only if} $\overline{X(\omega_\lambda)} = \overline{X(\omega)}$ for any $\lambda \in \Lambda$.
\end{cor}

As promised, now we study the case of $X=\ZZ$ and show that it admits a number of non-minimal points. The following is the main result:

\begin{thm}\label{thm:non-min point for Z}
For the integer group $\ZZ$ with the usual metric, there exist non-minimal points in the boundary $\partial_\beta \ZZ$. More precisely, for any sequence $\{h_n\}_{n\in \NN}$ in $\ZZ$ tending to infinity such that $|h_n -h_m| \to +\infty$ when $n+m \to \infty$ and $n\neq m$, and any $\omega \in \partial_\beta \ZZ$ with $\omega(\{h_n\}_{n\in \NN})=1$, then $\omega$ is \emph{not} a minimal point.
\end{thm}

To prove Theorem \ref{thm:non-min point for Z}, we need some preparations. For future use, we record the following result in the context of a countable discrete group $\Gamma$ equipped with a left-invariant word length metric.

\begin{lem}\label{lem:char for closure of limit space}
For $\omega, \alpha \in \partial_\beta \Gamma$, then $\alpha \in \overline{\Gamma(\omega)}$ \emph{if and only if} for any $S \subseteq \Gamma$ with $\alpha(S)=1$, there exists $g_S \in \Gamma$ such that $\omega(S \cdot g_S^{-1}) =1$.
\end{lem}

\begin{proof}
Recall from Lemma \ref{lem:limit space for group case} that $\Gamma(\omega) = \{\rho_g(\omega): g\in \Gamma\}$, and from Appendix \ref{app:ultrafilters} that $\{\overline{S}: S \subseteq X\}$ forms a basis for $\beta X$. Hence by definition, we obtain that $\alpha \in \overline{X(\omega)}$ if and only if for any $S \subseteq X$ with $\alpha \in \overline{S}$, there exists $g_S \in \Gamma$ such that $\rho_{g_S}(\omega) \in \overline{S}$. Equivalently, this means that for any $S \subseteq X$ with $\alpha(S)=1$, there exists $g_S \in \Gamma$ such that $\rho_{g_S}(\omega)(S) = \omega(S \cdot g_S^{-1}) =1$, which concludes the proof.
\end{proof}

%

Now we return to the case of $\Gamma=\ZZ$ and prove Theorem \ref{thm:non-min point for Z}:

\begin{proof}[Proof of Theorem \ref{thm:non-min point for Z}]
Fix a subset $H=\{h_n\}_{n\in \NN} \subseteq \ZZ$ tending to infinity such that $|h_n -h_m| \to +\infty$ when $n+m \to \infty$ and $n\neq m$. For any non-zero $g\in \ZZ$, note that $h\in (g+H) \cap H$ if and only if there exists $h' \in H$ such that $h-h'=g$. Since $g$ is fixed and distances between different points in $H$ tend to infinity, we obtain that $(g+H) \cap H$ is finite. Hence for any $g_1 \neq g_2$ in $\ZZ$, $(g_1+H) \cap (g_2 + H)$ is finite.

Fixing a non-principal ultrafilter $\omega \in \partial_\beta \ZZ$ with $\omega(H)=1$, we denote 
\[
\U:=\{B \subseteq H: \omega(B)=1\}. 
\]
We claim that for each $n\in \NN$, there exists $g_n \in \ZZ$ and $B_n \in \U$ such that $\{B_n+g_n\}_{n\in \NN}$ are mutually disjoint. Indeed, we take $g_0=0$ and $B_0 = H$. Set $g_1=1$ and $B_1:=H \setminus (H-g_1)$. Since $H \cap (H-g_1)$ is finite by the previous paragraph, then $\omega(B_1)=1$, \emph{i.e.}, $B_1 \in \U$. Similarly for each $n \in \NN$, we take $g_n=n$ and $B_n:=H \setminus \big( (H-g_1) \cup (H-g_2) \cup \cdots \cup (H-g_{n}) \big)$, which concludes the claim.

Consider $\widetilde{H}:=\bigsqcup_{n\in \NN} (B_n + g_n)$, and denote $\U_n:=\{B \subseteq B_n: \omega(B)=1\}$ for each $n\in \NN$. By Lemma \ref{lem:localisation ultrafilter}, $\U_n$ is an ultrafilter on $B_n$. Choose a non-principal ultrafilter $\omega_0$ on $\NN$, and we consider:
\[
\widetilde{\U}:=\big\{ \bigsqcup_{n\in \NN} (A_n + g_n) \subseteq \bigsqcup_{n\in \NN} (B_n + g_n) = \widetilde{H}: \exists~J \subseteq \NN \text{ with } \omega_0(J)=1 \text{ s.t. } \forall n\in J, A_n \in \U_n \big\}.
\]
Lemma \ref{lem:ultraunion} implies that $\widetilde{\U}$ is an ultrafilter on $\widetilde{H}$. We define a function $\alpha: \P(\ZZ) \to \{0,1\}$ by setting $\alpha(S)=1$ if and only if $S \cap \widetilde{H} \in \widetilde{\U}$, which is indeed an ultrafilter on $\ZZ$ thanks to Lemma \ref{lem:ultraextension}. Also note that $\alpha$ is non-principal and $\alpha(\widetilde{H})=1$.

Now we show that $\alpha \in \overline{\ZZ(\omega)}$ while $\omega \notin \overline{\ZZ(\alpha)}$, and hence conclude the proof. To see that $\alpha \in \overline{\ZZ(\omega)}$, we will consult Lemma \ref{lem:char for closure of limit space}. For any $S \subseteq \ZZ$ with $\alpha(S)=1$, by definition we have $S \cap \widetilde{H} \in \widetilde{\U}$. Writing $S \cap \widetilde{H} = \bigsqcup_{n\in \NN}(A_n + g_n)$ with $A_n \subseteq B_n$, then there exists $J \in \omega_0$ such that for any $n\in J$ we have $A_n \in \U_n$. Hence for any $n \in J$, we have $S \supseteq S \cap \widetilde{H} \supseteq A_n + g_n$ and $\omega(A_n)=1$, which implies that $\omega(S - g_n)=1$. Applying Lemma \ref{lem:char for closure of limit space}, we conclude that $\alpha \in \overline{\ZZ(\omega)}$.

Finally, it remains to check that $\omega \notin \overline{\ZZ(\alpha)}$. Assume the contrary, then Lemma \ref{lem:char for closure of limit space} implies that there exists $g\in \ZZ$ and $\widetilde{B} \subseteq \widetilde{H}$ with $\alpha(\widetilde{B})=1$ such that $H \supseteq \widetilde{B} - g$. Writing $\widetilde{B} = \bigsqcup_{n\in \NN} (A_n + g_n)$ with $A_n \subseteq B_n$, then there exists $J \in \omega_0$ such that for any $n\in J$, $\omega(A_n) = 1$. This implies that
\[
H + g \supseteq \widetilde{B} \supseteq \bigsqcup_{n\in J} (A_n + g_n) \supseteq A_{n_0} + g_{n_0}
\]
for some $n_0 \in J$ with $g_{n_0} \neq g$ (this can be achieved since $J$ is infinite). However, it follows from the first paragraph that $(H + g) \cap (H + g_{n_0})$ is finite. While this intersection contains $A_{n_0} + g_{n_0}$, which is infinite since $\omega(A_{n_0}) =1$. Therefore, we reach a contradiction and conclude the proof.
\end{proof}

\subsection{Maximal ideals via limit operators}

We already see that maximal ideals in the uniform Roe algebra are related to minimal points in the Stone-\v{C}ech boundary of the underlying space and in Section \ref{ssec:min points}, we use topological methods to show the existence of non-minimal points. Now we turn to a $C^*$-algebraic viewpoint, and use the tool of limit operators to provide an alternative description for these ideals.

First recall from Corollary \ref{cor:maximal ideals using minimal closed subsets} and Lemma \ref{lem:property of minimal points} that maximal ideals in $C^*_u(X)$ arise in the form of $\tilde{I}(\beta X \setminus \overline{X(\omega)})$ for some boundary point $\omega \in \partial_\beta X$. Moreover according to Proposition \ref{prop: char for Ker qU}, $\tilde{I}(\beta X \setminus \overline{X(\omega)})$ is the kernel of the following surjective homomorphism:
\[
q_{\beta X \setminus \overline{X(\omega)}}:  C^*_r(G(X)) \longrightarrow C^*_r(G(X)_{\overline{X(\omega)}}).
\]
Hence we obtain the following:

\begin{cor}\label{cor:simple char for max ideal}
A point $\omega \in \partial_\beta X$ is minimal \emph{if and only if} $C^*_r(G(X)_{\overline{X(\omega)}})$ is simple.
\end{cor}

\begin{ex}\label{ex:group case minimal}
Consider the case of a countable discrete group $\Gamma$. For a point $\omega \in \partial_\beta \Gamma$, it follows from Lemma \ref{lem:limit space for group case} that the limit space $\Gamma(\omega)$ is identical to $\Gamma \omega$, and hence $C^*_r(G(\Gamma)_{\overline{\Gamma(\omega)}})$ is $C^*$-isomorphic to the reduced crossed product $C(\overline{\Gamma \omega}) \rtimes \Gamma$. Thanks to Corollary \ref{cor:simple char for max ideal}, we obtain that $\omega$ is minimal if and only if $C(\overline{\Gamma \omega}) \rtimes \Gamma$ is simple. Moreover, note that the action of $\Gamma$ on $\beta \Gamma$ is free (which is also a consequence of Lemma \ref{lem:limit space for group case}). Hence it follows from \cite[Corollary 4.6]{Ren91} that $C(\overline{\Gamma \omega}) \rtimes \Gamma$ is simple if and only if the action of $\Gamma$ on $\overline{\Gamma \omega}$ is minimal. In conclusion, we reach the following:

\begin{cor}
In the case of a countable discrete group $\Gamma$, a point $\omega \in \partial_\beta \Gamma$ is minimal \emph{if and only if} the action of $\Gamma$ on $\overline{\Gamma \omega}$ is minimal.
\end{cor}
\end{ex}

Corollary \ref{cor:simple char for max ideal} suggests an approach to distinguish minimal points via the simplicity of the reduced groupoid $C^*$-algebra. However, this $C^*$-algebra is still not easy to handle since it requires to consider all points in the $\overline{X(\omega)}$-direction (see Proposition \ref{prop: geometric char for ghostly ideals}). 
Now we show that this can be simplified by merely considering the $\omega$-direction:

\begin{lem}\label{lem:simplification of maximal ideal}
For $\omega \in \partial_\beta X$, an operator $T\in C^*_u(X)$ belongs to the ideal $\tilde{I}(\beta X \setminus \overline{X(\omega)})$ \emph{if and only if} $T$ vanishes in the $\omega$-direction, \emph{i.e.}, $\Phi_\omega(T)=0$.
\end{lem}

\begin{proof}
We assume that $\Phi_\omega(T)=0$, and it suffices to show that $\Phi_\alpha(T)=0$ for any $\alpha \in \overline{X(\omega)}$. Fixing such an $\alpha$, we take a net $\{\omega_\lambda\}_{\lambda \in \Lambda}$ in $X(\omega)$ such that $\omega_\lambda \to \alpha$ and it follows that $\Phi_{\omega_\lambda}(T)=0$. For any $\gamma_1, \gamma_2 \in X(\alpha)$, Lemma \ref{lem:limit space via coarse groupoids} and the fact that the coarse groupoid $G(X)$ is \'{e}tale imply that there exist $\gamma_{1,\lambda}$ and $\gamma_{2,\lambda}$ in $X(\omega_\lambda)$ for each $\lambda \in \Lambda$ such that $\gamma_{1,\lambda} \to \gamma_1$ and $\gamma_{2,\lambda} \to \gamma_2$. Now Lemma \ref{lem:limit operator using extension} implies that
\[
\Phi_\alpha(T)_{\gamma_1\gamma_2} = \overline{T}((\gamma_1,\gamma_2)) = \lim_{\lambda \in \Lambda} \overline{T}((\gamma_{1,\lambda},\gamma_{2,\lambda})) = \lim_{\lambda \in \Lambda} \Phi_{\omega_\lambda}(T)_{\gamma_{1,\lambda}\gamma_{2,\lambda}} = 0,
\]
which concludes the proof.
\end{proof}

Hence for $\omega \in \partial_\beta X$, Lemma \ref{lem:simplification of maximal ideal} implies that the associated ideal $\tilde{I}(\beta  \setminus \overline{X(\omega)})$ coincides with the kernel of the following limit operator homomorphism (see also \cite[Theorem 4.10]{SW17}):
\begin{equation}\label{EQ:limit op homo}
\Phi_\omega: C^*_u(X) \longrightarrow C^*_u(X(\omega)), \quad T \mapsto \Phi_\omega(T).
\end{equation}
Consequently, we reach the following:

\begin{cor}\label{cor:simple char for max ideal2}
A point $\omega \in \partial_\beta X$ is minimal \emph{if and only if} the image $\mathrm{Im}(\Phi_\omega)$ is simple. Hence when $X(\omega)$ is infinite and $\Phi_\omega$ is surjective, the point $\omega$ is \emph{not} minimal.
\end{cor}

\begin{proof}
For the last statement, it suffices to note that the ideal of compact operators is always contained in $C^*_u(X(\omega))$, which concludes the proof.
\end{proof}

\begin{rem}
We remark that ideals of the form $\tilde{I}(\beta  \setminus \overline{X(\omega)})$ already appeared in the work of Georgescu \cite{Geo11} (with the notation $\mathscr{G}_\omega(X)$) to study the limit operator theory on general metric spaces.
\end{rem}

Thanks to Corollary \ref{cor:simple char for max ideal2}, a special case of Theorem \ref{thm:non-min point for Z} can also be deduced from a recent work by Roch \cite{Roc22}. More precisely, combining \cite[Proposition B.6]{SW17} with \cite[Lemma 2.1]{Roc22}, we have the following:

\begin{prop}\label{prop:Roc22}
Let $\{h_n\}_{n\in \NN}$ be a sequence in $\ZZ^N$ tending to infinity such that 
\[
\|h_n - h_k\|_\infty \geq 3k \quad \text{for any} \quad k>n. 
\]
Then for any $\omega \in \partial_\beta \ZZ^N$ with $\omega(\{h_n\}_{n\in \NN}) =1$, the map $\Phi_\omega: C^*_u(\ZZ^N) \longrightarrow C^*_u(\ZZ^N(\omega))$ is surjective.
\end{prop}

Note that the limit space $\ZZ^N(\omega)$ is bijective to $\ZZ^N$ by Lemma \ref{lem:limit space for group case}, and hence infinite. Therefore applying Corollary \ref{cor:simple char for max ideal2}, we obtain the following (when $N=1$, it partially recovers Theorem \ref{thm:non-min point for Z}).

\begin{cor}\label{cor:Thm reproved}
For any sequence $\{h_n\}_{n\in \NN}$ in $\ZZ^N$ tending to infinity such that $\|h_n - h_k\|_\infty \geq 3k$ for any $k>n$, and any $\omega \in \partial_\beta \ZZ^N$ with $\omega(\{h_n\}_{n\in \NN})=1$, then $\omega$ is \emph{not} a minimal point.
\end{cor}

\begin{rem}
We notice from the discussion above that for $\omega \in \partial_\beta X$, there is a $C^*$-monomorphism:
\begin{equation}\label{EQ:embedding}
C^*_r(G(X)_{\overline{X(\omega)}}) \cong C^*_u(X) / \tilde{I}(\beta X \setminus \overline{X(\omega)}) \longrightarrow C^*_u(X(\omega))
\end{equation}
where the first comes from Proposition \ref{prop: char for Ker qU} and the second comes from (\ref{EQ:limit op homo}) together with Lemma \ref{lem:simplification of maximal ideal}. 

Now we provide another explanation for this map in terms of groupoids. Firstly, Corollary \ref{cor:char for G(X)} implies that $G(X)_{X(\omega)} = X(\omega) \times X(\omega)$ and hence $G(X)_{\overline{X(\omega)}} = \overline{X(\omega) \times X(\omega)}^{G(X)}$. Now we have
\[
G(X)_{\overline{X(\omega)}} = \bigcup_{S>0} \big( \overline{E_S}^{G(X)} \cap \overline{X(\omega) \times X(\omega)}^{G(X)}\big) = \bigcup_{S>0}  \overline{\overline{E_S}^{G(X)} \cap (X(\omega) \times X(\omega))}^{G(X)},
\]
where the last inequality is due to the fact that $\overline{E_S}^{G(X)}$ is clopen in $G(X)$. By Lemma \ref{lem:closure of entourage}, we have
\[
\overline{E_S}^{G(X)} \cap (X(\omega) \times X(\omega)) = \{(\alpha, \gamma) \in X(\omega) \times X(\omega): d_\omega(\alpha, \gamma) \leq S\}=:E_S(X(\omega), d_\omega)
\]
\emph{i.e.}, the $S$-entourage in the limit space $(X(\omega), d_\omega)$. Therefore, we conclude that
\[
G(X)_{\overline{X(\omega)}} = \bigcup_{S>0} \overline{E_S(X(\omega), d_\omega)}^{G(X)}.
\]
On the other hand, we have (by definition) that
\[
G(X(\omega)) = \bigcup_{S>0} \overline{E_S(X(\omega), d_\omega)}^{\beta (X(\omega) \times X(\omega))}.
\]
By the universal property of the Stone-\v{C}ech compactification, there is a surjective continuous map $\overline{E_S(X(\omega), d_\omega)}^{\beta (X(\omega) \times X(\omega))} \longrightarrow \overline{E_S(X(\omega), d_\omega)}^{G(X)}$, which induces an injective map $C_c(G(X)_{\overline{X(\omega)}}) \longrightarrow  C_c(G(X(\omega)))$. Moreover, it is routine to check that it induces a $C^*$-monomorphism
\[
C^*_r(G(X)_{\overline{X(\omega)}}) \longrightarrow  C^*_r(G(X(\omega))) \cong C^*_u(X(\omega)),
\]
which can be verified to coincide with the map (\ref{EQ:embedding}). Details are left to readers. 

In particular, we consider a countable discrete group $X =\Gamma$. Fixing a point $\omega \in \partial_\beta \Gamma$, we mentioned in Example \ref{ex:group case minimal} that $C^*_r(G(\Gamma)_{\overline{\Gamma(\omega)}}) \cong C(\overline{\Gamma \omega}) \rtimes \Gamma$. On the other hand, we know from Lemma \ref{lem:limit space for group case} that $C^*_u(\Gamma(\omega)) \cong \ell^\infty(\Gamma \omega) \rtimes \Gamma$. In this case, one can check that the map (\ref{EQ:embedding}) is induced by the natural embedding $C(\overline{\Gamma \omega}) \hookrightarrow \ell^\infty(\Gamma \omega)$.
\end{rem}

\section{Geometric ideals vs ghostly ideals}\label{sec:geometric vs ghostly ideals}

In this section, we return to the lattice $\mathfrak{I}_U$ in (\ref{EQ:set of ideals}). Recall from Theorem \ref{prop: ideals containment} that for an invariant open subset $U \subseteq \beta X$, the geometric ideal $I(U)$ and the ghostly ideal $\tilde{I}(U)$ are the smallest and the largest elements in $\mathfrak{I}_U$. Also as noticed before, generally $\mathfrak{I}_U$ consists of more than one element. 

Now we would like to study when $I(U)=\tilde{I}(U)$ or equivalently, when $\mathfrak{I}_U$ consists of a single element. Moreover, we will discuss their $K$-theories and provide a sufficient condition to ensure $K_\ast(I(U))=K_\ast(\tilde{I}(U))$ for $\ast=0,1$.


First recall from Proposition \ref{prop: char for Ker qU} that we already have a characterisation for $I(U) = \tilde{I}(U)$ using short exact sequences, while the condition therein still seems hard to check. Now we aim to search for a more practical criterion to ensure $I(U) = \tilde{I}(U)$.
We start with the following result, which combines \cite[Theorem 4.4]{CW05} and \cite[Theorem 1.3]{RW14}. 

\begin{prop}\label{prop: property A case}
For a space $(X,d)$, the following are equivalent:
\begin{enumerate}
 \item $X$ has Property A;
 \item $G(X)$ is amenable;
 \item $G(X)_{\partial_\beta X}$ is amenable;
 \item $I(U)=\tilde{I}(U)$ for any invariant open subset $U \subseteq \beta X$;
 \item $I(X) = \tilde{I}(X)$.
\end{enumerate}
\end{prop}

We remark that ``(1) $\Rightarrow$ (4)'' was originally proved in \cite{CW05} using approximations by kernels. Here we will take a shortcut, and the idea will also be used later.

\begin{proof}[Proof of Proposition \ref{prop: property A case}]
``(1) $\Leftrightarrow$ (2)'' was proved in \cite[Theorem 5.3]{STY02}. ``(2) $\Leftrightarrow$ (3)'' is due to the permanence properties of amenability (see Section \ref{ssec: amenability}) together with the fact that $G(X)_X \cong X \times X$ is always amenable.


``(2) $\Rightarrow$ (4)'': Let $U \subseteq \beta X$ be an invariant open subset. As open/closed subgroupoids, both $G(X)_U$ and $G(X)_{U^c}$ are amenable as well. Consider the following commutative diagram coming from (\ref{EQ: short exact seq. for max. alg.}) and (\ref{EQ: short exact seq. for red. alg.}):
\[ 
\xymatrix{
		0 \ar[r] & C^*_{{\rm max}}(G(X)_U) \ar[r] \ar[d] & C^*_{{\rm max}}(G(X)) \ar[r] \ar[d] & C^*_{{\rm max}}(G(X)_{U^c}) \ar[r] \ar[d] & 0\\
		0 \ar[r] & C^*_{r}(G(X)_U) \ar[r] & C^*_{r}(G(X)) \ar[r] & C^*_{r}(G(X)_{U^c}) \ar[r] & 0.
	}
\]
By Proposition \ref{prop: etale amen}, all three vertical lines are isomorphisms. Hence the exactness of the first row implies that the second row is exact as well. Therefore, we conclude (4) thanks to Proposition \ref{prop: char for Ker qU}.

``(4) $\Rightarrow$ (5)'' holds trivially, and ``(5) $\Rightarrow$ (1)'' comes from \cite[Theorem 1.3]{RW14} together with Example \ref{eg: compact ideal} and Example \ref{eg: ghost ideal}. Hence we conclude the proof.
\end{proof}

Proposition \ref{prop: property A case} provides a coarse geometric characterisation for $I(X) = \tilde{I}(X)$ using Property A. However, we notice that assuming Property A is often too strong to ensure that $I(U) = \tilde{I}(U)$ for merely a \emph{specific} invariant open subset $U \subseteq \beta X$. A trivial example is that $I(\beta X) = \tilde{I}(\beta X)$ holds for any space $X$. This suggests us to explore a weaker criterion for $I(U) = \tilde{I}(U)$, and we reach the following:

\begin{prop}\label{prop: criteria to ensure I(U) = tilde I(U)}
Let $(X,d)$ be a space and $U$ be an invariant open subset of $\beta X$. If the canonical quotient map $C^*_{{\rm max}}(G(X)_{U^c}) \to C^*_r(G(X)_{U^c})$ is an isomorphism, then $I(U)=\tilde{I}(U)$. In particular, if the groupoid $G(X)_{U^c}$ is amenable then $I(U)=\tilde{I}(U)$.
\end{prop}

\begin{proof}
We consider the following commutative diagram:
\begin{equation}\label{EQ: comm. diagram}
\xymatrix{
		0 \ar[r] & C^*_{{\rm max}}(G(X)_U) \ar[r] \ar[d]_{\pi_U} & C^*_{{\rm max}}(G(X)) \ar[r] \ar[d] & C^*_{{\rm max}}(G(X)_{U^c}) \ar[r] \ar[d] & 0\\
		0 \ar[r] & \tilde{I}(U) \ar[r] & C^*_{r}(G(X)) \ar[r] & C^*_{r}(G(X)_{U^c}) \ar[r] & 0.
	}
\end{equation}
Here the map $\pi_U$ is the composition:
\begin{equation}\label{EQ: composition}
C^*_{{\rm max}}(G(X)_U) \to C^*_r(G(X)_U) \cong I(U) \hookrightarrow \tilde{I}(U),
\end{equation}
where the middle isomorphism comes from Lemma \ref{lem: groupoid C*-alg char for I(U)}. Note that the top horizontal line is automatically exact, while the bottom one is also exact thanks to Proposition \ref{prop: char for Ker qU}. Also note that the middle vertical map is always surjective and by assumption, the right vertical map is an isomorphism. Hence via a diagram chasing argument, we obtain that the left vertical map is surjective. This concludes that $I(U)=\tilde{I}(U)$ thanks to (\ref{EQ: composition}).
\end{proof}

\begin{rem}\label{rem: special case of U=X}
When $U=X$, Proposition \ref{prop: criteria to ensure I(U) = tilde I(U)} recovers ``(3) $\Rightarrow$ (5)'' in Proposition \ref{prop: property A case}. Readers might wonder whether the converse of Proposition \ref{prop: criteria to ensure I(U) = tilde I(U)} holds as in the case of $U=X$. We manage to provide a partial answer in Section \ref{ssec:char for I(U)=tI(U)} below.
\end{rem}

Now we move to discuss the $K$-theory of geometric and ghostly ideals:

\begin{prop}\label{prop: criteria to ensure the same K-theory}
Let $X$ be a space which can be coarsely embedded into some Hilbert space. Then for any invariant open subset $U \subseteq \beta X$, we have an isomorphism
\[
(\iota_U)_\ast: K_\ast(I(U)) \longrightarrow K_\ast(\tilde{I}(U))
\]
for $\ast =0,1$, where $\iota_U$ is the inclusion map. Therefore for any ideal $I$ in $C^*_u(X)$, we have an injective homomorphism
\[
(\iota_I)_\ast: K_\ast(\overline{I \cap \CC_u[X]}) \longrightarrow K_\ast(I)
\]
for $\ast =0,1$, where $\iota_I$ is the inclusion map.
\end{prop}

\begin{proof}
Fixing such a $U\subseteq \beta X$, we consider the commutative diagram (\ref{EQ: comm. diagram}) from the proof of Proposition \ref{prop: criteria to ensure I(U) = tilde I(U)}, where both of the horizontal lines are exact. This implies the following commutative diagram in $K$-theories:
\begin{scriptsize}
\[
\xymatrix{
		\cdots \ar[r] & K_\ast(C^*_{{\rm max}}(G(X)_U)) \ar[r] \ar[d]_{(\pi_U)_\ast} & K_\ast(C^*_{{\rm max}}(G(X))) \ar[r] \ar[d] & K_\ast(C^*_{{\rm max}}(G(X)_{U^c})) \ar[r] \ar[d] & K_{\ast+1}(C^*_{{\rm max}}(G(X)_U)) \ar[r] \ar[d]_{(\pi_U)_{\ast +1}} &\cdots\\
		\cdots \ar[r] & K_\ast(\tilde{I}(U)) \ar[r] & K_\ast(C^*_r(G(X))) \ar[r] & K_\ast(C^*_r (G(X)_{U^c})) \ar[r] & K_{\ast+1}(\tilde{I}(U)) \ar[r] &\cdots
	}
\]
\end{scriptsize}
where both horizontal lines are exact. Recall from \cite[Theorem 5.4]{STY02} that $X$ is coarsely embeddable if and only if $G(X)$ is a-T-menable. It follows directly from Definition \ref{defn: a-T-menable groupoids} that as subgroupoids, both $G(X)_U$ and $G(X)_{U^c}$ are a-T-menable. Moreover, it is clear that $G(X)$ and $G(X)_{U^c}$ are $\sigma$-compact by definition. Hence Proposition \ref{prop: a-T-menable groupoids are K-amenable} implies that $G(X)$ and $G(X)_{U^c}$ are $K$-amenable. This shows that the middle two vertical maps in the above diagram are isomorphisms, which implies that $(\pi_U)_\ast$ is an isomorphism by the Five Lemma. Therefore from (\ref{EQ: composition}), we obtain that the composition
\[
K_\ast(C^*_{{\rm max}}(G(X)_U)) \longrightarrow K_\ast(C^*_r(G(X)_U)) \cong K_\ast(I(U)) \stackrel{(\iota_U)_\ast}{\longrightarrow} K_\ast(\tilde{I}(U))
\]
is an isomorphism for $\ast=0,1$. 

On the other hand, we claim that $K_\ast(C^*_{{\rm max}}(G(X)_U)) \longrightarrow K_\ast(C^*_r(G(X)_U))$ is an isomorphism, and conclude that the map $(\iota_U)_\ast$ is an isomorphism for $\ast=0,1$. To see, recall from \cite[Lemma 3.3]{STY02} that there is a locally compact, second countable, ample and \'{e}tale groupoid $\G'$ such that $G(X) = \beta X \rtimes \G'$, which implies that $G(X)_U = U \rtimes \G'$. Hence $C^*_{{\rm max}}(G(X)_U)$ coincides with the maximal crossed product $C_0(U) \rtimes_{{\rm max}} \G'$, and $C^*_r(G(X)_U)$ with the reduced one $C_0(U) \rtimes_{r} \G'$. Note that the $\G'$-$C^*$-algebra $C_0(U)$ can be written as an inductive limit $\varinjlim_{i\in I} (A_i, \phi_i)$ of separable $\G'$-$C^*$-algebras. Moreover, it follows from \cite[Theorem 5.4]{STY02} that $\G'$ is a-T-menable since $G(X)$ is a-T-menable, and hence Proposition \ref{prop: a-T-menable groupoids are K-amenable} implies that $\G'$ is also $K$-amenable. It follows from \cite[Proposition 4.12]{Tu99} that the canonical map $K_\ast(A_i \rtimes_{{\rm max}} \G') \rightarrow K_\ast(A_i \rtimes_{r} \G')$ is an isomorphism for each $i \in I$. Taking limits, we obtain that $K_\ast(C_0(U) \rtimes_{{\rm max}} \G') \rightarrow K_\ast(C_0(U) \rtimes_{r} \G')$ is also an isomorphism, which concludes the claim.


For the last statement, we assume that $U=U(I)$. By Lemma \ref{prop: geometric ideals are smallest}, we have that $\overline{I \cap \CC_u[X]} = I(U)$. Also Lemma \ref{lem: ghostly ideas are largest}(1) shows that $I \subseteq \tilde{I}(U)$. Hence the inclusion map $\iota_U$ can be decomposed as follows:
\[
I(U) = \overline{I \cap \CC_u[X]} \stackrel{\iota_I}{\hookrightarrow} I \hookrightarrow \tilde{I}(U).
\]
Therefore, $(\iota_U)_\ast$ being an isomorphism implies that $(\iota_I)_\ast$ is injective for $\ast=0,1$. 
\end{proof}

\begin{rem}\label{rem:thanks to Kang}
In general, we do not know whether the restriction $G(X)_U$ is $\sigma$-compact, and hence we cannot apply Proposition \ref{prop: a-T-menable groupoids are K-amenable} to $G(X)_U$ directly. In an early version of Proposition \ref{prop: criteria to ensure the same K-theory}, we put the assumption of countably generatedness (see Definition \ref{defn:countably generated} below) to bypass the issue, while Kang Li pointed out to us that this is redundant by using the technique above. 
\end{rem}

Applying Proposition \ref{prop: criteria to ensure the same K-theory} to the case of $U=X$, we partially recover the following result by Finn-Sell (see \cite[Proposition 35]{Fin14}), which is crucial for the  counterexamples to the coarse Baum-Connes conjecture:
\begin{cor}\label{cor: FS's result}
Let $X$ be a space which can be coarsely embedded into some Hilbert space. Then the inclusion of $\K(\ell^2(X))$ into $I_G$ induces an isomorphism on the $K$-theory level.
\end{cor}

\begin{rem}
For a general ideal $I$ in $C^*_u(X)$, our method in the proof of Proposition \ref{prop: criteria to ensure the same K-theory} only provides the injectivity of the induced map $(\iota_I)_\ast$. We wonder whether this map is indeed an isomorphism under the same assumption (see Question \ref{Q2}).
\end{rem}

\section{Partial Property A and partial operator norm localisation property}\label{sec:partial A and partial ONL}

In the previous section, we find a sufficient condition (Proposition \ref{prop: criteria to ensure I(U) = tilde I(U)}) to ensure $\tilde{I}(U) = I(U)$ for a given invariant open subset $U \subseteq \beta X$. As promised in Remark \ref{rem: special case of U=X}, now we study its converse and show that this is indeed an equivalent condition under an extra assumption.

Our strategy is to follow the outline of the case that $U=X$. More precisely, we introduce a notion called \emph{partial Property A towards invariant subsets of the boundary $\partial_\beta X$}, and then consider its counterpart in the context of operator norm localisation property to provide the desired characterisation.

\subsection{Partial Property A}\label{ssec:partial property A}

Recall from Proposition \ref{prop: property A case} that a space $X$ has Property A if and only if the groupoid $G(X)_{\partial_\beta X}$ is amenable, which characterises $I(X) = \tilde{I}(X)$. Together with Proposition \ref{prop: criteria to ensure I(U) = tilde I(U)}, this inspires us to introduce the following:

\begin{defn}\label{defn:partial Property A}
Let $(X,d)$ be a space and $U \subseteq \beta X$ be an invariant open subset. We say that $X$ has \emph{partial Property A towards $\partial_\beta X \setminus U$} if $G(X)_{\partial_\beta X \setminus U}$ is amenable.
\end{defn}

It is clear from definition that $X$ has Property A if and only if it has partial Property A towards the whole boundary $\partial_\beta X$. On the other hand, it follows from Proposition \ref{prop: criteria to ensure I(U) = tilde I(U)} that if $X$ has partial Property A towards $\partial_\beta X \setminus U$, then we have $I(U) = \tilde{I}(U)$. The rest of this section is devoted to studying the converse.

Firstly, we aim to unpack the groupoid language and provide a concrete geometric description for partial Property A, which resembles the definition of Property A (see Definition \ref{defn: Property A}). The following is the main result:

\begin{prop}\label{prop:char for partial A}
Let $(X,d)$ be a space and $U \subseteq \beta X$ an invariant open subset. Then $X$ has partial Property A towards $U^c=\partial_\beta X \setminus U$ \emph{if and only if} for any $\varepsilon, R>0$, there exist $S>0$, a subset $D \subseteq X$ with $\overline{D} \supseteq U^c$ and a function $f: X \times X \to [0,1]$ satisfying:
\begin{enumerate}
  \item $\supp(f) \subseteq E_S$;
  \item for any $x\in D$, we have $\sum_{z\in X} f(z,x) =1$;
  \item for any $x, y\in D$ with $d(x,y) \leq R$, we have $\sum_{z\in X} |f(z,x) - f(z,y)| \leq \varepsilon$.
\end{enumerate}
\end{prop}

Comparing Proposition \ref{prop:char for partial A} with Definition \ref{defn: Property A}, it is clear that Property A implies partial Property A towards any invariant closed subset of $\partial_\beta X$.

To prove Proposition \ref{prop:char for partial A}, we start with the following lemma:

\begin{lem}\label{lem:char for partial A}
With the same notation as above, $X$ has partial Property A towards $U^c$ \emph{if and only if} for any $\varepsilon, R>0$, there exist $S>0$ and a function $f: X \times X \to [0,1]$ satisfying:
\begin{enumerate}
  \item $\supp(f) \subseteq E_S$;
  \item for any $\omega\in U^c$, we have $\sum_{\alpha \in X(\omega)} \overline{f}(\alpha, \omega) =1$;
  \item for $\omega\in U^c$ and $\alpha \in X(\omega)$ with $d_\omega(\alpha, \omega) \leq R$, then $\sum_{\gamma\in X(\omega)} |\overline{f}(\gamma, \alpha) - \overline{f}(\gamma, \omega)| \leq \varepsilon$,
\end{enumerate}
where $\overline{f} \in C_0(G(X))$ is the continuous extension from Lemma \ref{lem: unif. Roe belongs to C0}.
\end{lem}

\begin{proof}
By definition, $X$ has partial Property A towards $U^c$ if and only if for any $\varepsilon>0$ and compact $K \subseteq G(X)_{U^c}$, there exists $g \in C_c(G(X)_{U^c})$ with range in $[0,1]$ such that for any $\gamma \in K$ we have
\[
\sum_{\alpha \in \G_{\r(\gamma)}} g(\alpha) =1  \quad \mbox{and} \quad \sum_{\alpha \in \G_{\r(\gamma)}} |g(\alpha) - g(\alpha\gamma)| < \varepsilon.
\]
Recall from (\ref{EQ: short exact seq. for alg.}) that the restriction map $C_c(G(X)) \to C_c(G(X)_{U^c})$ is surjective, and hence $g$ can be regarded as a function in $C_c(G(X))$. Taking $f$ to be the restriction of $g$ on $X \times X$, then $f \in \ell^\infty(X \times X)$ and there exists $S>0$ such that $\supp(f) \subseteq E_S$ for some $S>0$. Using the notation from Lemma \ref{lem: unif. Roe belongs to C0}, we have $g = \overline{f}$. Note that $G(X)_{U^c} = \bigcup_{R>0} (\overline{E_R} \cap G(X)_{U^c})$, and hence compact subsets of $G(X)_{U^c}$ are always contained in those of the form $\overline{E_R} \cap G(X)_{U^c}$. Furthermore, Lemma \ref{lem:closure of entourage} implies 
\[
\overline{E_R} \cap G(X)_{U^c} = \bigcup_{\omega \in U^c} \{(\alpha, \gamma)\in X(\omega) \times X(\omega): d_\omega(\alpha, \gamma) \leq R\}.
\]
Combining with Lemma \ref{lem:limit space via coarse groupoids}, we conclude the proof.
\end{proof}

As a direct corollary (together with Lemma \ref{lem:closure of entourage}), we obtain:
\begin{cor}\label{cor:uniform A}
Assume that $X$ has partial Property A towards $U^c$. Then the family of metric spaces $\{(X(\omega), d_\omega)\}_{\omega \in U^c}$ has uniform Property A in the sense that the parameters in Definition \ref{defn: Property A} can be chosen uniformly. 
\end{cor}

\begin{rem}
It is unclear to us whether the converse of Corollary \ref{cor:uniform A} holds. Note that $\overline{X(\omega)}$ might contain points outside $X(\omega)$ as discussed in Theorem \ref{thm:non-min point for Z}, hence we do not know whether functions on $X(\omega) \times X(\omega)$ can be glued together to provide a continuous function on $G(X)_{U^c}$.
\end{rem}

%

\begin{proof}[Proof of Proposition \ref{prop:char for partial A}]
\emph{Sufficiency:} For any $\varepsilon, R>0$, choose $S>0$, $D \subseteq X$ and a function $g: X \times X \to [0,1]$ satisfying the conditions (1)-(3) for $\varepsilon$ and $3R$. Take a map $p: \Nd_R(D) \to D$ such that the restriction of $p$ on $D$ is the identity map and $d(p(x), x) \leq R$. Now we define:
\[
f(x,y) = 
\begin{cases}
~g(x,p(y)), & y\in \Nd_R(D); \\
~g(x,y), & \text{otherwise}.
\end{cases}
\] 
It is clear that $\supp(f) \subseteq E_{R+S}$. Moreover for any $y_1, y_2 \in \Nd_R(D)$ with $d(y_1, y_2) \leq R$, we have $d(p(y_1),p(y_2)) \leq 3R$ and hence
\[
\sum_{x\in X} |f(x,y_1) - f(x,y_2)| = \sum_{x\in X} |g(x,p(y_1)) - g(x,p(y_2))| \leq \varepsilon.
\]
Therefore (enlarging $S$ to $S+R$) we obtain that for any $\varepsilon, R>0$, there exist $S>0$, a subset $D \subseteq X$ with $\overline{D} \supseteq U^c$ and a function $f: X \times X \to [0,1]$ satisfying:
\begin{enumerate}
  \item $\supp(f) \subseteq E_S$;
  \item for any $x\in D$, we have $\sum_{z\in X} f(z,x) =1$;
  \item for any $x\in D$ and $y\in X$ with $d(x,y) \leq R$, we have $\sum_{z\in X} |f(z,x) - f(z,y)| \leq \varepsilon$.
\end{enumerate}
Now we fix $\varepsilon, R>0$ and take such $S, D$ and function $f$.

Given $\omega \in U^c$, we have $\omega(D)=1$. Choose $\{t_\alpha:D_\alpha \to R_\alpha\}$ to be a compatible family for $\omega$. Applying Proposition \ref{prop:local geometry}, there exists $Y_\omega \subseteq X$ with $\omega(Y)=1$ and a local coordinate system $\{h_y: B(\omega, R+S) \to B(y,R+S)\}_{y\in Y_\omega}$ such that the map 
\[
h_y: B(\omega, R+S) \to B(y,R+S), \quad \alpha \mapsto t_\alpha(y)
\]
is a surjective isometry for each $y\in Y_\omega$. Replacing $Y$ by $Y_\omega \cap D$, we assume that $Y_\omega \subseteq D$. 
Note that $\supp(f) \subseteq E_S$, and hence applying Lemma \ref{lem:closure of entourage} we have
\begin{align}
\sum_{\alpha \in X(\omega)} \overline{f}(\alpha, \omega) & = \sum_{\alpha \in B(\omega, S)} \overline{f}(\alpha, \omega) = \sum_{\alpha \in B(\omega, S)} \lim_{x\to \omega} f(t_\alpha(x),x) = \lim_{x\to \omega, x\in Y_\omega} \sum_{\alpha \in B(\omega, S)} f(t_\alpha(x),x) \nonumber \\ 
&= \lim_{x\to \omega, x\in Y_\omega} \sum_{z\in B(x,S)} f(z,x) = \lim_{x\to \omega, x\in Y_\omega} \sum_{z\in X} f(z,x), \label{EQ:cal1}
\end{align}
where the last item equals $1$ thanks to the assumption.

On the other hand, for any $\alpha \in B(\omega, R)$ we have $\lim_{x\to \omega} d(t_\alpha(x),x) \leq R$. Hence shrinking $Y_\omega$ if necessary, we can assume that $d(t_\alpha(x),x) \leq R$ for any $x\in Y_\omega$. Therefore, we have
\begin{align}
\sum_{\gamma\in X(\omega)} |\overline{f}(\gamma, \alpha) - \overline{f}(\gamma, \omega)| &= \sum_{\gamma\in B(\omega, R+S)} |\overline{f}(\gamma, \alpha) - \overline{f}(\gamma, \omega)| \label{EQ:cal2}\\
&= \lim_{x\to \omega, x\in Y_\omega}\sum_{\gamma\in B(\omega, R+S)}  |f(t_\gamma(x), t_\alpha(x)) - f(t_\gamma(x), x)|\nonumber\\
&= \lim_{x\to \omega, x\in Y_\omega}\sum_{z\in B(x, R+S)}  |f(z, t_\alpha(x)) - f(z, x)| \nonumber\\
&= \lim_{x\to \omega, x\in Y_\omega}\sum_{z\in X}  |f(z, t_\alpha(x)) - f(z, x)|, \nonumber
\end{align}
where the last item is no more than $\varepsilon$ by assumption. Therefore applying Lemma \ref{lem:char for partial A}, we conclude the sufficiency.

\emph{Necessity:} Given $\varepsilon, R>0$, Lemma \ref{lem:char for partial A} provides $S>0$ and a function $f: X \times X \to [0,1]$ satisfying the conditions (1)-(3) therein. Fix an $\omega \in U^c$ and we choose $\{t_\alpha:D_\alpha \to R_\alpha\}$ to be a compatible family for $\omega$. Applying Proposition \ref{prop:local geometry}, there exists $Y_\omega \subseteq X$ with $\omega(Y)=1$ and a local coordinate system $\{h_y: B(\omega, R+2S) \to B(y,R+S)\}_{y\in Y_\omega}$ such that the map 
\[
h_y: B(\omega, R+S) \to B(y,R+S), \quad \alpha \mapsto t_\alpha(y)
\]
is a surjective isometry for each $y\in Y_\omega$. By the calculations in (\ref{EQ:cal1}), we obtain that
\[
\lim_{x\to \omega, x\in Y_\omega} \sum_{z\in X} f(z,x) =1.
\]
Hence for the given $\varepsilon$, there exists $Y'_\omega \subseteq Y_\omega$ with $\omega(Y'_\omega)=1$ such that for any $x\in Y'_\omega$ we have 
\[
\sum_{z\in X} f(z,x) \in (1-\varepsilon, 1+\varepsilon).
\]

On the other hand, for any $\alpha \in B(\omega, R)$ we apply the calculations in (\ref{EQ:cal2}) and obtain:
\[
\lim_{x\to \omega, x\in Y'_\omega}\sum_{z\in X}  |f(z, t_\alpha(x)) - f(z, x)| = \sum_{\gamma\in X(\omega)} |\overline{f}(\gamma, \alpha) - \overline{f}(\gamma, \omega)| \leq \varepsilon.
\]
Note that for $x\in Y'_\omega$, Proposition \ref{prop:local geometry} implies that $\{t_\alpha(x): \alpha \in B(\omega, R)\} = B(x,R)$. Hence there exists $Y''_\omega \subseteq Y'_\omega$ with $\omega(Y''_\omega)=1$ such that for any $x \in Y''_\omega$ and $y\in B(x,R)$, we have
\[
\sum_{z\in X}  |f(z, y) - f(z, x)| < 2\varepsilon.
\]

Taking $D:=\bigcup_{\omega \in U^c} Y''_\omega$, then it is clear that $\overline{D} \supseteq U^c$. Moreover, the analysis above shows that:
\begin{itemize}
 \item for any $x\in D$ we have $\sum_{z\in X} f(z,x) \in (1-\varepsilon, 1+\varepsilon)$;
 \item for any $x\in D$ and $y\in B(x,R)$, we have $\sum_{z\in X}  |f(z, y) - f(z, x)| < 2\varepsilon$.
\end{itemize}
Finally using a standard normalisation argument (or equivalently, applying Lemma \ref{lem:elementary char for amenable groupoid} and modifying Lemma \ref{lem:char for partial A} accordingly), we conclude the proof.
\end{proof}

Setting $\xi_y(x):=f(x,y)$ for the function $f$ in Proposition \ref{prop:char for partial A}, we can rewrite Proposition \ref{prop:char for partial A} as follows:

\begin{propbis}{prop:char for partial A}\label{prop:char for partial A II}
Let $(X,d)$ be a space and $U \subseteq \beta X$ be an invariant open subset. Then $X$ has partial Property A towards $U^c$ \emph{if and only if} for any $\varepsilon, R>0$, there exist $S>0$, a subset $D \subseteq X$ with $\overline{D} \supseteq U^c$ and a function $\xi: D \to \ell^1(X)_{1,+}, x \mapsto \xi_x$ satisfying:
\begin{enumerate}
  \item $\supp(\xi_x) \subseteq B(x,S)$ for any $x\in D$;
  \item for any $x, y\in D$ with $d(x,y) \leq R$, we have $\|\xi_x - \xi_y\|_1 \leq \varepsilon$.
\end{enumerate}
\end{propbis}

\begin{rem}\label{rem:mod for char}
We remark that the function $\xi$ in Proposition \ref{prop:char for partial A II} can be made such that $\xi_x \in \ell^1(D)_{1,+}$. In fact, this is the same trick as in the case of Property A (see, \emph{e.g.}, \cite[Proposition 4.2.5]{NY12}). Moreover, we can further replace $\ell^1(D)_{1,+}$ by $\ell^2(D)_{1,+}$ using the Mazur map (see, \emph{e.g.}, \cite[Proposition 1.2.4]{Wil09} for the same trick).
\end{rem}

Now we provide an alternative picture for Proposition \ref{prop:char for partial A} using the notion of ideals in spaces (see Definition \ref{defn: ideals in space}). Recall that for an ideal $\L$ in $X$, we denote $U(\L):= \bigcup_{Y \in \L} \overline{Y}$. We need the following auxiliary lemma:

\begin{lem}\label{lem:relating ideals and D}
Let $\L$ be an ideal in $X$ and $D \subseteq X$. Then $\overline{D} \supseteq U(\L)^c$ \emph{if and only if} there exists $Y \in \L$ such that $D \supseteq Y^c$.
\end{lem}

\begin{proof}
Assume $Y \in \L$ such that $D \supseteq Y^c$. Note that $\overline{Y} \cap \overline{Y^c} = \emptyset$ and $\overline{Y} \cup \overline{Y^c} = \beta X$. Hence we have $\overline{D} \supseteq \overline{Y^c} = \beta X \setminus \overline{Y} \supseteq \beta X \setminus U(\L) = U(\L)^c$.

Conversely, assume that $\overline{D} \supseteq U(\L)^c$. Then $(\overline{D})^c \subseteq U(\L) = \bigcup_{Y \in \L} \overline{Y}$. Since $\overline{D}$ is clopen in the compact space $\beta X$, the set $(\overline{D})^c$ is compact as well. Hence there exists $Y_1, \cdots, Y_n \in \L$ such that $(\overline{D})^c \subseteq \overline{Y_1} \cup \cdots \cup \overline{Y_n} = \overline{Y_1 \cup \cdots \cup Y_n}$. Since $\L$ is an ideal, the set $Y:=Y_1 \cup \cdots \cup Y_n \in \L$. Then we have $(\overline{D})^c \subseteq \overline{Y}$, which implies that $\overline{D} \supseteq (\overline{Y})^c = \overline{Y^c}$. Finally we obtain $D = \overline{D} \cap X \supseteq \overline{Y^c} \cap X = Y^c$, which conclude the proof.
\end{proof}

Thanks to Lemma \ref{lem:relating ideals and D}, now we can rewrite Proposition \ref{prop:char for partial A} (combining with Remark \ref{rem:mod for char}) as follows:

\begin{propbiss}{prop:char for partial A}
Let $(X,d)$ be a space, $U \subseteq \beta X$ an invariant open subset and $\L=\L(U)$ the associated ideal in $X$. Then $X$ has partial Property A towards $U^c$ \emph{if and only if} for any $\varepsilon, R>0$, there exist $S>0$, a subset $Y \in \L(U)$ and a function $\xi: Y^c \to \ell^2(Y^c)_{1,+}, x \mapsto \xi_x$ satisfying:
\begin{enumerate}
  \item $\supp(\xi_x) \subseteq B(x,S)$ for any $x\in Y^c$;
  \item for any $x, y\in Y^c$ with $d(x,y) \leq R$, we have $\|\xi_x - \xi_y\|_1 \leq \varepsilon$.
\end{enumerate}
\end{propbiss}

Similar to the proof in the case of Property A, we also have the following:

\begin{cor}\label{cor:char for partial A}
Let $(X,d)$ be a space, $U \subseteq \beta X$ an invariant open subset and $\L=\L(U)$ the associated ideal in $X$. Then $X$ has partial Property A towards $U^c$ \emph{if and only if} for any $R>0$ and $\varepsilon>0$, there exist $S>0$, a subset $Y \in \L(U)$ and a kernel $k: Y^c \times Y^c \to \RR$ of positive type satisfying:
\begin{enumerate}
 \item for $x,y\in Y^c$, we have $k(x,y)=k(y,x)$ and $k(x,x) =1$;
 \item for $x,y\in Y^c$ with $d(x,y) \geq S$, we have $k(x,y)=0$;
 \item for $x,y\in Y^c$ with $d(x,y) \leq R$, we have $|1-k(x,y)| \leq \varepsilon$.
\end{enumerate}
\end{cor}

\subsection{Partial operator norm localisation property}

Recall that the notion of operator norm localisation property (ONL) was introduced by Chen, Tessera, Wang and Yu in \cite{CTWY08}, and proved by Sako in \cite{Sak14} that ONL is equivalent to Property A. Here we introduce a partial version of ONL, parallel to Definition \ref{defn:partial Property A}.

Let $\nu$ be a positive locally finite Borel measure on $X$ and $\H$ be a separable infinite-dimensional Hilbert space. For an operator $T \in \B(L^2(X,\nu) \otimes \H)$, we can also define its propagation as in Section \ref{ssec:uniform Roe alg}. We introduce the following:

\begin{defn}\label{defn:partial ONL}
Let $(X,d)$ be a space and $U \subseteq \beta X$ an invariant open subset. We say that $X$ has \emph{partial operator norm localisation property (partial ONL) towards $U^c=\partial_\beta X \setminus U$} if there exists $c \in (0,1]$ such that for any $R>0$ there exist $S>0$ and $D \subseteq X$ with $\overline{D} \supseteq U^c$ satisfying the following: for any positive locally finite Borel measure $\nu$ on $X$ with $\supp(\nu) \subseteq D$ and any $a \in \B(L^2(X,\nu) \otimes \H)$ with propagation at most $R$, there exists a non-zero $\zeta \in L^2(X,\nu) \otimes \H$ with $\diam(\supp(\zeta)) \leq S$ such that $c\|a\| \cdot \|\zeta\| \leq \|a \zeta\|$. 
\end{defn}

The aim of the rest of this subsection is to show that partial ONL is equivalent to partial Property A. We will follow the outline of \cite{Sak14}. 

To simplify the statement, denote $\CC^R_u[X;\H]:=\{T \in \B(\ell^2(X) \otimes \H): \ppg(a) \leq R\}$ for $R \geq 0$. For a subspace $Y \subseteq X$, it is clear that $\CC^R_u[Y] \cong \chi_Y\CC^R_u[X] \chi_Y$ (resp. $\CC^R_u[Y;\H]\cong \chi_Y\CC^R_u[X;\H] \chi_Y$), and hence can be regarded as a subset of $\CC^R_u[X]$ (resp. $\CC^R_u[X;\H]$) with support in $Y \times Y$. Similarly, $C^*_u(Y) \cong \chi_Y C^*_u(X) \chi_Y$ can be regarded as a $C^*$-subalgebra in $C^*_u(X)$. Hence we will not tell the difference in the sequel. For $S>0$, denote
\[
B^Y_S:=\prod_{x\in X} \B(\ell^2(B(x,S) \cap Y)),
\]
whose elements will be written as $b=([b^{x}(y,z)]_{y,z\in B(x,S) \cap Y})_{x\in X}$. We also consider the map
\[
\psi_S^Y: C^*_u(Y) \cong \chi_Y C^*_u(X) \chi_Y \longrightarrow B^Y_S \quad \text{by} \quad  a \mapsto ([a(y,z)]_{y,z \in B(x,S) \cap Y})_{x\in X}.
\]

Recall the notions of completely positive map and completely bounded map:
\begin{itemize}
  \item A self-adjoint closed subspace $F$ of a unital $C^*$-algebra $B$ such that $1_B \in F$ is called an \emph{operator system}. 
  \item A linear map $\phi$ from $F$ to a $C^*$-algebra $C$ is said to be \emph{completely positive} if the map $\phi^{(n)} = \phi \otimes \mathrm{id}: F \otimes M_n(\CC) \to C \otimes M_n(\CC)$ is positive for every $n$.
  \item A linear map $\theta: F \to C$ is said to be \emph{completely bounded} if the sequence $\{\|\theta^{(n)}:F \otimes M_n(\CC) \to C \otimes M_n(\CC)\|\}$ is bounded. Denote $\|\theta\|_{\cb}:= \sup_{n\in \NN}\|\theta^{(n)}\|$.
\end{itemize}

We have the following characterisation for partial ONL, which is analogous to \cite[Proposition 3.1]{Sak14}. The proof is almost identical, hence omitted.

\begin{lem}\label{lem:char for partial ONL}
Let $(X,d)$ be a space and $U \subseteq \beta X$ an invariant open subset. Then the following are equivalent:
\begin{enumerate}
 \item $X$ has partial ONL towards $U^c$;
 \item there exists $c\in (0,1]$ such that for any $R>0$ there exist $S>0$ and $D \subseteq X$ with $\overline{D} \supseteq U^c$ satisfying condition $(\alpha)$: for any $a\in \CC^R_u[D;\H]$ there exists a non-zero $\zeta \in \ell^2(X) \otimes \H$ with $\diam(\supp(\zeta)) \leq S$ and $c\|a\| \cdot \|\zeta\| \leq \|a \zeta\|$;
 \item for any $c\in (0,1)$ and $R>0$, there exist $S>0$ and $D \subseteq X$ with $\overline{D} \supseteq U^c$ satisfying  condition $(\alpha)$;
 \item for any $c\in (0,1)$ and $R>0$, there exist $S>0$ and $D \subseteq X$ with $\overline{D} \supseteq U^c$ satisfying  condition $(\beta)$: for any $a \in \CC^R_u[D]$ there exists a non-zero $\xi \in \ell^2(X)$ with $\diam(\supp(\xi)) \leq S$ and $c\|a\| \cdot \|\xi\| \leq \|a \xi\|$;
 \item for any $\varepsilon, R>0$ there exist $S>R$ and $D \subseteq X$ with $\overline{D} \supseteq U^c$ such that
 \[
 \|(\psi_S^D|_{\CC^R_u[D]})^{-1}: \psi_S^D(\CC^R_u[D]) \longrightarrow \CC^R_u[D]\| < 1+\varepsilon;
 \]
 \item for any $\varepsilon, R>0$ there exist $S>R$ and $D \subseteq X$ with $\overline{D} \supseteq U^c$ such that
 \[
 \|(\psi_S^D|_{\CC^R_u[D]})^{-1}: \psi_S^D(\CC^R_u[D]) \longrightarrow \CC^R_u[D]\|_{\cb} < 1+\varepsilon;
 \]
\end{enumerate}
\end{lem}

We record the following, which comes directly from Lemma \ref{lem:relating ideals and D} and \ref{lem:char for partial ONL}.


\begin{lem}\label{lem:simple char for partial ONL}
Let $(X,d)$ be a space, $U \subseteq \beta X$ an invariant open subset and $\L=\L(U)$ the associated ideal in $X$. Then $X$ has partial ONL towards $U^c$ \emph{if and only if} for any $c \in (0,1)$ and $R>0$ there exist $S>0$ and $Y \in \L(U)$ satisfying the following: for any $a \in \CC^R_u[Y^c]$ there exists a non-zero $\xi \in \ell^2(X)$ with $\diam(\supp(\xi)) \leq S$ and $c\|a\| \cdot \|\xi\| \leq \|a \xi\|$.
\end{lem}

%

Finally, we can mimic the proof of \cite[Theorem 4.1]{Sak14} using Proposition \ref{prop:char for partial A II}, Remark \ref{rem:mod for char}, Corollary \ref{cor:char for partial A} and Lemma \ref{lem:char for partial ONL} instead, and reach the following. The proof is almost identical, and hence omitted.

\begin{prop}\label{prop:partial A = partial ONL}
Let $(X,d)$ be a space and $U \subseteq \beta X$ be an invariant open subset. Then $X$ has partial Property A towards $U^c$ \emph{if and only if} $X$ has partial ONL towards $U^c$.
\end{prop}

To end this subsection, we study a permanence property of partial ONL, which will help to prove the main result in the next subsection. 

Let $X$ be a space and $\L$ an ideal in $X$. Assume that $X$ can be decomposed into $X = X_1 \cup X_2$. Consider 
\begin{equation}\label{EQ:restriction of ideals}
\L_i:=\{Y \cap X_i: Y \in \L\} \quad \text{for} \quad i=1,2.
\end{equation}
Then it is routine to check that $\L_i$ is an ideal in $X_i$ for $i=1,2$, and 
\[
\L=\{Y_1 \cup Y_2: Y_i \in \L_i, i=1,2\}.
\]
With respect to the decomposition above, we now show that partial ONL is preserved under finite unions.

\begin{prop}\label{prop:permanence for partial ONL}
With the notation as above, assume that $X_i$ has partial ONL towards $\beta X_i \setminus U(\L_i)$ for $i=1,2$. Then $X$ has partial ONL towards $\beta X \setminus U(\L)$.
\end{prop}

One way to prove Proposition \ref{prop:permanence for partial ONL} is to follow the proof of \cite[Lemma 3.3]{CWW09} with minor changes. Here we choose another approach using Proposition \ref{prop:partial A = partial ONL}.

Recall from Corollary \ref{cor:subspace in SC comp.} that for a subset $Z \subseteq X$, the closure $\overline{Z}$ in $\beta X$ is homeomorphic to $\beta Z$. Hence we will regard them as the same object in the sequel. Now we can easily transfer the restriction of ideals in (\ref{EQ:restriction of ideals}) to that of invariant open subsets (the proof is straightforward, hence omitted):

\begin{lem}\label{lem:transfer ideals to open subsets}
Let $X = X_1 \cup X_2$, $\L$ be an ideal in $X$ and $U = U(\L) \subseteq \beta X$ be the associated invariant open subset of $\beta X$. Then $U \cap \overline{X_i}$ is an invariant open subset of $\overline{X_i} = \beta X_i$, which corresponds to the ideal $\L_i$ in (\ref{EQ:restriction of ideals}) for $i=1,2$.
\end{lem}

Although $U \cap \overline{X_i}$ is invariant in $\beta X_i$, generally it is not invariant in $\beta X$. This coincides with the fact that $\chi_{X_i} C^*_u(X) \chi_{X_i} \cong C^*_u(X_i)$ is a just subalgebra in $C^*_u(X)$ rather than an ideal. Instead, we consider the spatial ideal $I_{X_i}$ recalled in Section \ref{sec:geometric ideals},
%
and prove the following permanence property for partial Property A: 

\begin{lem}\label{lem:permanence for partial A}
Let $X = X_1 \cup X_2$, $U \subseteq \beta X$ be an invariant open subset and $U_i=U \cap \overline{X_i}$ for $i=1,2$. If $X_i$ has partial Property A towards $\beta X_i \setminus U_i$ for $i=1,2$, then $X$ has partial Property A towards $\beta X \setminus U$.
\end{lem}

\begin{proof}
For any $R>0$ and $i=1,2$, set $U_i(R):=\overline{\bigcup_{Y \in \L(U_i)} \Nd_R(Y)}$. Clearly, $U_i(R)$ is an invariant open subset of $\overline{\Nd_R(X_i)} = \beta (\Nd_R(X_i))$. Moreover, it follows from Proposition \ref{prop:char for partial A} that $\Nd_R(X_i)$ has partial Property A towards $\overline{\Nd_R(X_i)} \setminus U_i(R)$. By definition, this means that the groupoid $G(\Nd_R(X_i))_{\overline{\Nd_R(X_i)} \setminus U_i(R)}$ is amenable. Hence as a subgroupoid, $G(\Nd_R(X_i))_{\overline{\Nd_R(X_i)} \setminus U}$ is amenable, which further implies that the groupoid $\bigcup_{R>0}G(\Nd_R(X_i))_{\overline{\Nd_R(X_i)} \setminus U}$ is amenable.

Note from Lemma \ref{lem: invariant subset} that $\bigcup_R \overline{\Nd_R(X_i)}\setminus U$ is invariant in $\beta X$ and Lemma \ref{lem:dec for groupoids UA} implies that 
\[
G(X)_{\bigcup_R \overline{\Nd_R(X_i)}\setminus U} = \bigcup_{R>0} G(\Nd_R(X_i))_{\overline{\Nd_R(X_i)} \setminus U},
\]
which is hence amenable. Note that $\bigcup_R \overline{\Nd_R(X_1)} \cup \bigcup_R \overline{\Nd_R(X_2)} = \beta X$, and hence due to the extension property we obtain that
\[
G(X)_{\beta X \setminus U} = G(X)_{\bigcup_R \overline{\Nd_R(X_1)}\setminus U} \cup G(X)_{\bigcup_R \overline{\Nd_R(X_2)}\setminus U}
\]
is amenable as required.
\end{proof}

Combining Proposition \ref{prop:partial A = partial ONL} and Lemma \ref{lem:permanence for partial A}, we conclude Proposition \ref{prop:permanence for partial ONL}.

\subsection{Countable generatedness}

Here we introduce the extra assumption we need to characterise $\tilde{I}(U) = I(U)$ later. 

\begin{defn}\label{defn:countably generated}
For a set $\S$ of subsets of $X$, denote $\L(\S)$ the smallest ideal in $X$ containing $\S$, and we say that $\L(\S)$ is \emph{generated by $\S$}. An ideal $\L$ in $X$ is called \emph{countably generated} if there exists a countable set $\S$ such that $\L=\L(\S)$. 

An invariant open subset $U \subseteq \beta X$ is called \emph{countably generated} if the associated ideal $\L(U)$ is countably generated.
\end{defn}

\begin{ex}\label{ex:countably generated}
For a space $X$ and a subspace $A \subseteq X$, it follows from Lemma \ref{lem: invariant subset} that the spatial ideal $I_A$ is countably generated. On the other hand, it follows from Lemma \ref{lem: U(I) for principal ideals} that principal ideals are always countably generated. In particular, the ideal $\langle P \rangle$ considered in \cite[Section 3]{Wan07} (see also Example \ref{ex: Wan07}) is countably generated.
\end{ex}

The property of countable generatedness leads to the following:

\begin{lem}\label{lem:description of countably generated}
Let $\L$ be a countably generated ideal in $X$. Then there exists a countable subset $\{Y_1, Y_2, \cdots, Y_n, \cdots\}$ in $\L$ such that 
\[
\L=\{Z\subseteq X: ~\exists~n\in \NN \text{ such that } Z \subseteq Y_n\}.
\]
\end{lem}

To prove Lemma \ref{lem:description of countably generated}, we need an auxiliary result on the structure of $\L(\S)$:

\begin{lem}\label{lem:construction of ideals}
For a set $\S$ of subsets of $X$, denote $\S^{(1)}:=\{A_1 \cup \cdots \cup A_n: A_i \in \S, n\in \NN\}$ and $\S^{(2)}:=\{\Nd_k(\tilde{A}): \tilde{A} \in \S^{(1)}, k \in \NN\}$. Then we have:
\[
\L(\S) = \{Z: \exists~Y \in \S^{(2)} \text{ such that } Z \subseteq Y\}.
\]
\end{lem}

\begin{proof}
Denote $\S^{(3)}:=\{Z: \exists~Y \in \S^{(2)} \text{ such that } Z \subseteq Y\}$. It is clear that $\S^{(3)} \subseteq \L(\S)$, and any ideal containing $\S$ must contain $\S^{(3)}$. Hence it remains to check that $\S^{(3)}$ is an ideal.

For $Z \in \S^{(3)}$ and $W \subseteq Z$, there exists $Y \in \S^{(2)}$ such that $Z \subseteq Y$. Hence $W \subseteq Y$, which implies that $W \in \S^{(3)}$. Also note that $Y \in \S^{(2)}$ implies that there exists $\tilde{A} \in \S^{(1)}$ and $k'\in \NN$ such that $Y=\Nd_{k'}(\tilde{A})$. Hence for any $k\in \NN$, we have $\Nd_k(Z) \subseteq \Nd_k(Y) \subseteq \Nd_{k+k'}(\tilde{A})$, which implies that $\Nd_k(Z) \in \S^{(3)}$.

Finally for $Z_1, Z_2 \in \S^{(3)}$, there exist $\tilde{A}_1, \tilde{A}_2 \in \S^{(1)}$ and $k_1, k_2 \in \NN$ such that $Z_i \subseteq \Nd_{k_i}(\tilde{A}_i)$ for $i=1,2$. Hence 
\[
Z_1 \cup Z_2 \subseteq \Nd_{k_1}(\tilde{A}_1) \cup \Nd_{k_2}(\tilde{A}_2) \subseteq \Nd_{k_1 + k_2}(\tilde{A}_1 \cup \tilde{A}_2) \in \S^{(2)},
\]
which implies that $Z_1 \cup Z_2 \in \S^{(3)}$. Therefore, we conclude the proof.
\end{proof}

\begin{proof}[Proof of Lemma \ref{lem:description of countably generated}]
By assumption, there exists a countable $\S$ such that $\L= \L(\S)$. Using the notation of Lemma \ref{lem:construction of ideals}, the set $\S^{(2)}$ is countable as well. Hence the first statement follows directly from Lemma \ref{lem:construction of ideals}. 
\end{proof}

The following example shows that \emph{not every} ideal in a space is countably generated.

\begin{ex}\label{ex:noncountably generated ideal}
Let $X = \NN \times \NN$, equipped with the metric induced from the Euclidean metric $d_E$ on $\RR^2$. For each $\theta \in [0, \frac{\pi}{2}]$ and $k\in \NN$, we define 
\[
\ell_\theta:=\{(x,y) \in \RR \times \RR: y=\tan(\theta)x\}
\]
and
\[
S_{\theta,k}:=\{(x,y) \in X: d_E((x,y), \ell_\theta) \leq k\} = \Nd_k(\ell_\theta) \cap X.
\]
Consider $\S:=\{S_{\theta,k}: \theta \in [0, \frac{\pi}{2}], k\in \NN\}$ and set $\S^{(1)}:=\{A_1 \cup \cdots \cup A_n: A_i \in \S, n\in \NN\}$ as in Lemma \ref{lem:construction of ideals}. For any $R>0$ and $A= A_1 \cup \cdots \cup A_n\in \S^{(1)}$ where $A_i \in \S$, we have $\Nd_R(A) = \Nd_R(A_1) \cup \cdots \cup \Nd_R(A_n)$, which is contained in some element in $\S$. Hence applying Lemma \ref{lem:construction of ideals}, the ideal $\L(\S)$ generated by $\S$ is:
\begin{equation}\label{EQ:proof}
\L(\S)= \{Z: \exists~Y \in \S^{(1)} \text{ such that } Z \subseteq Y\}.
\end{equation}

We claim that $\L(\S)$ is \emph{not} countably generated. Otherwise, there exists a countable subset $\S'$ generating $\L(\S)$. Moreover, according to (\ref{EQ:proof}) we can assume that
\[
\S' = \{Y_n: n\in \NN\} \quad \text{where} \quad Y_n=S_{\theta_{n,1}, k_n} \cup \cdots \cup S_{\theta_{n,p_n}, k_{n}} \in \S^{(1)}.
\]
Choose $\theta \in [0,\frac{\pi}{2}] \setminus \{\theta_{n,i}: i=1,2,\cdots, p_n; n\in \NN\}$, and consider $Y=S_{\theta,1}$. Since $\S'$ generates $\L(\S)$, it follows from Lemma \ref{lem:construction of ideals} that there exist $R>0$ and $Y_{m_1}, \cdots, Y_{m_l} \in \S'$ such that 
\[
S_{\theta,1}  =Y \subseteq \Nd_R(Y_{m_1}\cup \cdots\cup Y_{m_l}).
\]
Note that the right hand side is contained in a finite union of some $R'$-neighbourhood of lines (in $\RR^2$ crossing the origin) with slopes in the set 
\[
\big\{\tan(\theta_{n,i}): i=1,2,\cdots, p_n; n\in \NN\big\}.
\] 
This leads to a contradiction due to the choice of $\theta$, which concludes that $\L(\S)$ cannot be countably generated. 
\end{ex}

\subsection{Characterisation for $I(U) = \tilde{I}(U)$}\label{ssec:char for I(U)=tI(U)}

Having established all the necessary ingredients above, now we present the main result of this section:

\begin{thm}\label{thm:I(U)=tI(U)}
Let $(X,d)$ be a space as in Section \ref{ssec:notions from coarse geometry} and $U \subseteq \beta X$ be a countably generated invariant open subset.
Then the following are equivalent:
\begin{enumerate}
 \item $X$ has partial Property A towards $\beta X \setminus U$;
 \item $\tilde{I}(U) = I(U)$;
 \item the ideal $I_G$ of all ghost operators is contained in $I(U)$.
\end{enumerate}
\end{thm}

Note that $U=X$ is countably generated, and hence Theorem \ref{thm:I(U)=tI(U)} recovers \cite[Theorem 1.3]{RW14} (see Example \ref{eg: compact ideal} and Example \ref{eg: ghost ideal}). 
Borrowing the language of \cite{RW14}, condition (3) in Theorem \ref{thm:I(U)=tI(U)} says that all the ghosts can be busted in the geometric ideal $I(U)$.

We follow the outline of the proof for \cite[Theorem 1.3]{RW14}. Firstly, we need a modified version of \cite[Lemma 4.2]{RW14}:

\begin{lem}\label{lem:pre for I(U)=tI(U)}
Let $(X,d)$ be a space, $U \subseteq \beta X$ be a countably generated invariant open subset and $\L=\L(U)$ the associated ideal in $X$. Assume that $X$ does not have partial ONL towards $U^c$. Then there exist $\kappa \in (0,1)$, $R>0$, a sequence $(T_n)$ in $\CC_u[X]$, a sequence $(B_n)$ of finite subsets of $X$ and a sequence $(S_n)$ of positive real numbers such that:
\begin{enumerate}[(a)]
 \item $(S_n)$ is an increasing sequence tending to infinity as $n \to \infty$;
 \item each $T_n$ is positive and has norm $1$;
 \item for $n\neq m$, then $B_n \cap B_m = \emptyset$;
 \item each $T_n$ is supported in $B_n \times B_n$;
 \item for each $n$ and $\xi \in \ell^2(X)$ with $\|\xi\|=1$ and $\diam(\supp \xi) \leq S_n$, then $\|T_n \xi\| \leq \kappa$;
 \item for each $Y \in \L(U)$, there exists $n$ such that $B_n \cap Y = \emptyset$.
\end{enumerate}
\end{lem}

\begin{proof}
Fixing a basepoint $x_0 \in X$, consider the decomposition $X = X^{(1)} \cup X^{(2)}$ with
\[
X^{(1)}:= \bigsqcup_{m \text{ even}} \{x\in X: m^2 \leq d(x, x_0) < (m+1)^2\}\]
and
\[
X^{(2)}:= \bigsqcup_{m \text{ odd}} \{x\in X: m^2 \leq d(x, x_0) < (m+1)^2\}.
\]
Set $\L_i:=\{Y \cap X^{(i)}: Y \in \L\}$ for $i=1,2$. By assumption, $X$ does not have partial Property A towards $U^c$. Hence without loss of generality, we can assume that $X^{(1)}$ does not have partial Property A towards $\beta X^{(1)} \setminus U(\L_1)$ thanks to Proposition \ref{prop:permanence for partial ONL}. Note that this implies that $U(\L_1) \neq \beta X^{(1)}$. 
It is clear that $\L_1$ is also countably generated, and hence according to Lemma \ref{lem:description of countably generated} there exists a countable subset $\{Y_1, Y_2, \cdots, Y_n, \cdots\}$ of $\L_1$ such that 
\[
\L_1=\{Z\subseteq X^{(1)}: ~\exists~n\in \NN \text{ such that } Z \subseteq Y_n\}.
\]
In the sequel, we fix such a sequence $\{Y_1, Y_2, \cdots, Y_n, \cdots\}$.

Due to Lemma \ref{lem:simple char for partial ONL}, we know that there exist $c\in (0,1)$ and $R>0$ such that for any $Y \in \L_1$ and $S>0$, there exists $T \in \CC^R_u[X^{(1)} \setminus Y]$ with $\|T\|=1$ satisfying: for any $\xi \in \ell^2(X^{(1)})$ with $\diam(\supp \xi) \leq S$ and $\|\xi\|=1$, then $\|T \xi\| < c$. We call such an operator \emph{$(R,c,S,Y)$-localised}. Replacing $T$ by $T^*T$ (and $R$ by $2R$ and $c$ by $\sqrt{c}$), we see that there exist $c\in (0,1)$ and $R>0$ such that for any $Y \in \L_1$ and $S>0$, there exists a positive $(R,c,S,Y)$-localised operator  of norm one. Let us fix such $c$ and $R$ in the rest of the proof, and set $\kappa:=\frac{2c}{1+c} < 1$. 

Note that $X^{(1)}$ can be decomposed into:
\[
X^{(1)}:= \bigsqcup_{m\in \NN} X_m
\]
where each $X_m$ is finite and $d(X_m, X_n) > R$ for any $n \neq m$. Hence each $T \in \CC^R_u[X^{(1)}]$ splits as a block diagonal sum of finite rank operators $T = \bigoplus_m T^{(m)}$ where $T^{(m)} \in \B(\ell^2(X_m))$, with respect to this decomposition.

Take $S_1=1$. By assumption, there exists a positive $(R,c,S_1,Y_1)$-localised operator $T \in \CC^R_u[X^{(1)} \setminus Y_1]$ with norm $1$. Note that $\|T\| = \sup_m \|T^{(m)}\|$, and then there exists $m_1 \in \NN$ such that $\|T^{(m_1)}\| > \frac{1+c}{2}$. We set $T_1 := T^{(m_1)} / \|T^{(m_1)}\|$ and denote $B_1:=X_{m_1} \cap (X^{(1)} \setminus Y_1)$, which is nonempty since $T$ has support in $B_1 \times B_1$ by assumption. Then for any $\xi \in \ell^2(X^{(1)})$ with $\|\xi\|=1$ and $\diam(\supp(\xi)) \leq S_1$, we have 
\[
\|T_1 \xi\| \leq \frac{2c}{1+c} = \kappa <1.
\]

Now we take $S_2> \max\big\{\diam\big( \bigsqcup_{k \leq m_1} X_k \big), 2\big\}$. By assumption, there exists a positive $(R,c,S_2,Y_2)$-localised operator $T \in \CC^R_u[X^{(1)} \setminus Y_2]$ with norm $1$. Again there exists $m_2$ such that $\|T^{(m_2)}\| > \frac{1+c}{2}$, which forces $m_2 > m_1$. We set $T_2 := T^{(m_2)} / \|T^{(m_2)}\|$ and denote $B_2:=X_{m_2} \cap (X^{(1)} \setminus Y_2)$, which is nonempty since $T$ has support in $B_2 \times B_2$ by assumption. Similarly for any $\xi \in \ell^2(X^{(1)})$ with $\|\xi\|=1$ and $\diam(\supp(\xi)) \leq S_2$, we have $\|T_2 \xi\| \leq \kappa$.

Inductively, we can construct a sequence $(T_n)$ in $\CC_u[X^{(1)}] \subseteq \CC_u[X]$, a sequence $(B_n)$ of finite subsets of $X^{(1)} \subseteq X$ and a sequence $(S_n)$ of positive real numbers satisfying condition (a)-(e) in the statement. Furthermore for each $Z \in \L(U)$, there exists $Y_n$ containing $Z \cap X^{(1)}$ for some $n$. By construction, we know that $B_n\cap Y_n = \emptyset$ and $B_n \subseteq X^{(1)}$. Hence we have:
\[
B_n \cap Z = B_n \cap X^{(1)} \cap Z \subseteq B_n \cap Y_n = \emptyset,
\]
which provides condition (f) and concludes the proof.
\end{proof}

\begin{rem}\label{rem:importance of ideals in space}
Comparing Lemma \ref{lem:pre for I(U)=tI(U)} with \cite[Lemma 4.2]{RW14}, we note that condition (f) is the only extra condition added in Lemma \ref{lem:pre for I(U)=tI(U)}. It seems hard to write condition (f) in the language of the invariant open subset $U$ instead of the ideal $\L(U)$, which indicates the importance of using the notion of ideals in spaces as mentioned in Section \ref{sec:intro}.
\end{rem}

Now we are in the position to prove Theorem \ref{thm:I(U)=tI(U)}.

\begin{proof}[Proof of Theorem \ref{thm:I(U)=tI(U)}]
``(1) $\Rightarrow$ (2)'' is contained in Proposition \ref{prop: criteria to ensure I(U) = tilde I(U)}, and ``(2) $\Rightarrow$ (3)'' holds trivially since $I_G = \tilde{I}(X) \subseteq \tilde{I}(U)$. Hence it suffices to show ``(3) $\Rightarrow$ (1)'', and we follow the outline of the proof for \cite[Theorem 1.3]{RW14}. 

Assume that $X$ does not have partial Property A towards $U^c$, then it follows from Proposition \ref{prop:partial A = partial ONL} that $X$ does not have partial ONL towards $U^c$. Then from Lemma \ref{lem:pre for I(U)=tI(U)}, there exist $\kappa \in (0,1)$, $R>0$, a sequence $(T_n)$ in $\CC_u[X]$, a sequence $(B_n)$ of finite subsets of $X$ and a sequence $(S_n)$ of positive real numbers satisfying condition (a)-(f) therein. Now we consider the operator 
\[
T :=\bigoplus_n T_n,
\]
which is a positive operator in $\CC_u[X]$ of norm one. 

Now we take a continuous function $f: [0,1] \to [0,1]$ such that $\supp f \subseteq [\frac{1+\kappa}{2},1]$ and $f(1)=1$. Consider the operator $f(T) \in C^*_u(X)$, which is positive, norm one, and admits a decomposition
\[
f(T) = \bigoplus_n f(T_n),
\]
where each $f(T_n) \in \B(\ell^2(B_n))$. We will show that $f(T) \in \tilde{I}(X)\setminus I(U)$, and hence conclude a contradiction. 

First we show that $f(T) \notin I(U)$. Recall from (\ref{EQ: I(U(L))}) that 
\[
I(U)=\overline{\{T' \in \CC_u[X]: \supp(T') \subseteq Y \times Y \mbox{~for~some~}Y \in \L(U)\}},
\]
Now for any $T' \in \CC_u[X]$ with $\supp(T') \subseteq Y \times Y$ for some $Y \in \L(U)$, condition (f) in Lemma \ref{lem:pre for I(U)=tI(U)} implies that there exists $n$ such that $B_n \cap Y = \emptyset$. Hence we have:
\[
\|f(T) - T'\| \geq \|\chi_{B_n} f(T) \chi_{B_n} - \chi_{B_n} T' \chi_{B_n}\|  = \|f(T_n) - 0\| =1,
\]
which implies that $f(T) \notin I(U)$.

On the other hand, using the same argument as for \cite[Theorem 1.3]{RW14} (since Lemma \ref{lem:pre for I(U)=tI(U)} provides all the conditions required in \cite[Lemma 4.2]{RW14}), we obtain that $f(T)$ is a ghost operator. Hence according to Example \ref{eg: ghost ideal}, we have $f(T) \in \tilde{I}(X)$. Therefore, we conclude the proof.
\end{proof}

\section{Open questions}\label{sec:open questions}

Here we collect several open questions around this topic.

First recall from Theorem \ref{prop: ideals containment} that for a space $(X,d)$, any ideal $I$ in the uniform Roe algebra $C^*_u(X)$ must lie between $I(U)$ and $\tilde{I}(U)$ for $U=U(I)$. However, the structure of the lattice
\[
\mathfrak{I}_{U}=\{I \mbox{~is~an~ideal~in~} C^*_u(X): U(I)=U\}
\]
in (\ref{EQ:set of ideals}) is still unclear. Note that for any invariant open subset $V \supseteq U$ of $\beta X$, the ideal $I(V) \cap \tilde{I}(U)$ belongs to the lattice $\mathfrak{I}_{U}$. Unfortunately, we do not know whether these ideals can bust every element in $\mathfrak{I}_{U}$. Hence we pose the following:

\begin{Q}\label{Q1}
Let $(X,d)$ be a space and $U \subseteq \beta X$ be an invariant open subset. Can we describe elements in the lattice $\mathfrak{I}_{U}=\{I \mbox{~is~an~ideal~in~} C^*_u(X): U(I)=U\}$ in details? For $I \in \mathfrak{I}_{U}$, can we find an invariant open subset $V \supseteq U$ such that $I = I(V) \cap \tilde{I}(U)$?
\end{Q}

Note that an answer to the above question together with Theorem \ref{prop: ideals containment} will provide a full description for the ideal structure of the uniform Roe algebra.

Our next question concerns minimal points discussed in Section \ref{ssec:min points}. Recall that minimal points in the Stone-\v{C}ech boundary correspond to maximal ideals in the uniform Roe algebra. However as shown in Theorem \ref{thm:non-min point for Z}, there exist a number of non-minimal points in the boundary. Hence it would be interesting to explore a practical approach to distinguish minimal points. 

\begin{Q}
Given a space $(X,d)$, can we find a practical approach to distinguish minimal points in the Stone-\v{C}ech boundary $\partial_\beta X$? 
\end{Q}

Note from Theorem \ref{thm:non-min point for Z} that the answer might not be easy even in the elementary case that $X= \ZZ$.

The next question concerns the discussion on $K$-theory in Section \ref{sec:geometric vs ghostly ideals}. Recall that in the second part of Proposition \ref{prop: criteria to ensure the same K-theory}, we prove that for an ideal $I$ in $C^*_u(X)$, $(\iota_I)_\ast: K_\ast(\overline{I \cap \CC_u[X]}) \longrightarrow K_\ast(I)$ is injective when $X$ is coarsely embeddable. Hence we pose the following:

\begin{Q}\label{Q2}
For an ideal $I$ in $C^*_u(X)$ with $X$ coarsely embeddable, is the map $(\iota_I)_\ast: K_\ast(\overline{I \cap \CC_u[X]}) \longrightarrow K_\ast(I)$ surjective for $\ast =0,1$?
\end{Q}

Our final question is designed for Theorem \ref{thm:I(U)=tI(U)}. Recall that the assumption of countably generatedness plays an important role in Lemma \ref{lem:pre for I(U)=tI(U)}, which is in turn crucial in the proof of Theorem \ref{thm:I(U)=tI(U)}. We wonder whether this assumption is necessary, and ask the following:

\begin{Q}
Let $(X,d)$ be a space and $U \subseteq \beta X$ be an invariant open subset. If $\tilde{I}(U) = I(U)$, can we deduce that $X$ has partial Property A towards $\beta X \setminus U$ without the assumption of $U$ being countably generated?
\end{Q}

\appendix

\section{Ultrafilters}\label{app:ultrafilters}

Here we collect some basic knowledge on ultrafilters, which is used throughout the paper. The material should be fairly well-known (see, \emph{e.g.}, \cite[Appendix A]{BO08}, \cite[Chapter 7.4]{Roe03} or \cite[Appendix A]{SW17}). While some of the results might not be standard and we have not dug into the reference, we include the proofs for convenience to readers.

\begin{defn}\label{defn:ultrafilter}
Let $X$ be a set and $\P(X)$ be its power set. An \emph{ultrafilter} on $X$ is a family $\U \subseteq \P(X)$ satisfying the following:
\begin{enumerate}
 \item[(I.1)] $\emptyset \notin \U$;
 \item[(I.2)] for $A, B \in \U$, then $A \cap B \in \U$;
 \item[(I.3)] for $A \in \U$ and $A \subseteq B$, then $B \in \U$;
 \item[(I.4)] for any $A \subseteq X$, either $A \in \U$ or $X \setminus A \in \U$.
\end{enumerate}
\end{defn}

For an ultrafilter $\U$ on $X$, we can associate a function $\omega: \P(X) \to \{0,1\}$ by setting $\omega(A) =1$ if and only if $A \in \U$. It follows from (I.1)-(I.4) above that $\omega$ is a finitely additive $\{0,1\}$-valued probability measure on $\P(X)$. Conversely for such a function $\omega$ on $\P(X)$, we can associate a family $\U_\omega:=\{A \subseteq X: \omega(A)=1\}$. It is clear that $\U_\omega$ is an ultrafilter on $X$, and these two procedures are inverse to each other. Therefore throughout the paper, we slide between these two notions freely without further explanation.

For $a \in X$, it is clear that the family $\{A \in \P(X): a\in A\}$ is an ultrafilter on $X$. Such an ultrafiler is called \emph{principal}. An ultrafilter which is not principal is called \emph{non-principal}. An argument using Zorn’s lemma shows that non-principal ultrafilters always exist whenever $X$ is infinite.

The following is well-known (see, \emph{e.g.}, \cite[Lemma A.2]{SW17}):

\begin{lem}\label{lem:ultralimit}
Let $\omega$ be an ultrafilter on a set $X$, and $D \subseteq X$ with $\omega(D)=1$. Let $f: D \to Y$ be a function from $D$ to a compact Hausdorff topological space $Y$. Then there exists a unique point $y\in Y$ such that for any open neighbourhood $U$ of $y$, we have $\omega(f^{-1}(U))=1$. 
\end{lem}

\begin{defn}\label{defn:ultralimit}
The unique point in Lemma \ref{lem:ultralimit} is called the \emph{ultralimit} of $f$ along $\omega$ or the \emph{$\omega$-limit} of $f$, denoted by $\lim_\omega f$ or $\lim_{a\to \omega} f(a)$. 
\end{defn}

We record the following localisation result:

\begin{lem}\label{lem:localisation ultrafilter}
Let $\U$ be an ultrafilter on a set $X$, and $A \subseteq X$ with $A \in \U$. Then we have:
\begin{enumerate}
 \item $\{S \cap A: S \in \U\} = \{S \subseteq A: S \in \U\}$ is an ultrafilter on $A$, denoted by $\U_A$.
 \item $\U=\{S \subseteq X: S \cap A \in \U_A\} = \{S \subseteq X: \exists~ S' \in \U_A \text{ such that } S' \subseteq S\}$.
\end{enumerate}
\end{lem}

\begin{proof}
(1). It follows from Definition \ref{defn:ultrafilter} that $\{S \cap A: S \in \U\} = \{S \subseteq A: S \in \U\}$, and hence (I.1)-(I.3) hold for $\U_A$. Concerning (I.4): given $B \subseteq A$, if $B \in \U$ then $B \in \U_A$ as well; if $B \notin \U$ then $X \setminus B \in \U$, and hence $A \setminus B = A \cap (X \setminus B) \in \U_A$.


(2). This is straightforward, hence omitted.
\end{proof}

We can also extend an ultrafilter on a subset to the whole space. The proof is straightforward, hence omitted.

\begin{lem}\label{lem:ultraextension}
Let $Y$ be a subset of a set $X$, and $\U_0$ an ultrafilter on $Y$. Define 
\[
\U:=\{S \subseteq X: S \cap Y \in \U_0\}.
\]
Then $\U$ is an ultrafilter on $X$.
\end{lem}

The following result provides an approach to combine a family of ultrafilters into a single one:

\begin{lem}\label{lem:ultraunion}
Let $\{X_i\}_{i\in I}$ be a family of sets, and $\U_i$ be an ultrafilter on $X_i$ for each $i\in I$. Let $\omega_0$ be an ultrafilter on $I$. Consider the set $X:=\bigsqcup_{i\in I} X_i$ and define:
\[
\U:=\big\{ \bigsqcup_{i\in I} A_i \subseteq \bigsqcup_{i\in I} X_i: \exists~J \subseteq I \text{ with } \omega_0(J)=1 \text{ such that } \forall i\in J, A_i \in \U_i \big\}.
\]
Then $\U$ is an ultrafilter on $X$.
\end{lem}

\begin{proof}
Firstly, it is clear that $\emptyset \notin \U$. Assume that $\bigsqcup_{i\in I} A_i$ and $\bigsqcup_{i\in I} B_i \in \U$, \emph{i.e.}, there exist $J_A, J_B \subseteq I$ with $\omega_0(J_A) = \omega_0(J_B)=1$ such that $A_i \in \U_i$ for any $i \in J_A$ and $B_i \in \U_i$ for any $i\in J_B$. Consider $(\bigsqcup_{i\in I} A_i) \cap (\bigsqcup_{i\in I} B_i) = \bigsqcup_{i\in I} (A_i \cap B_i)$ and $J=J_A \cap J_B$. Then $\omega_0(J) = 1$ and for each $i\in J$, $A_i$ and $B_i$ are in $\U_i$. This implies that $A_i \cap B_i \in \U_i$, which concludes (I.2).

It is clear that (I.3) holds for $\U$ and finally, we consider (I.4). Given $A=\bigsqcup_{i\in I} A_i \subseteq X$, denote $J:=\{i \in I: A_i \in \U_i\}$. If $\omega_0(J)=1$, then it follows that $A \in \U$. Otherwise, assume that $\omega_0(J)=0$. Then we consider $X \setminus A = \bigsqcup_{i\in I} (X_i \setminus A_i)$. Then $I \setminus J = \{i\in I: X_i \setminus A_i \in \U_i\}$ and $\omega_0(I \setminus J) =1$, which implies that $X \setminus A \in \U$ and concludes the proof. 
\end{proof}

Recall that ultrafilters can also be characterised by the Stone-\v{C}ech compactification. More precisely, we have the following (see, \emph{e.g.}, \cite[Chapter 7.4]{Roe03}):

\begin{lem}\label{lem:ultrafilter to SC comp.}
Let $X$ be a set and $\beta X$ be the Stone-\v{C}ech compactification of $X$. 
\begin{enumerate}
 \item Given $\omega \in \beta X$, the family $\{A \subseteq X: \omega \in \overline{A}\}$ is an ultrafilter on $X$.
 \item Given an ultrafilter $\U$ on $X$, the intersection $\bigcap_{A \in \U} \overline{A}$ consists of a single point.
\end{enumerate}
The procedures above are inverse to each other, and hence $\beta X$ can be characterised by ultrafilters on $X$. Moreover, points in $\partial_\beta X$ correspond to non-principal ultrafilters.
\end{lem}

Thanks to Lemma \ref{lem:ultrafilter to SC comp.}, we will also use ultrafilters and points in the Stone-\v{C}ech compactification freely without further explanation throughout the paper. 

For convenience, we also record that for $D \subseteq X$, its closure $\overline{D}$ in $\beta X$ satisfies:
\[
\overline{D} = \{\omega \in \beta X: \omega(D)=1 \}
\]
and $\overline{D}$ is clopen in $\beta X$. The topology of $\beta X$ is generated by $\{\overline{D}: D \subseteq X\}$.

Finally we recall the following, which can be proved either directly using the universal property of the Stone-\v{C}ech compactification or deduced directly from Lemma \ref{lem:localisation ultrafilter}, \ref{lem:ultraextension} and \ref{lem:ultrafilter to SC comp.}:

\begin{cor}\label{cor:subspace in SC comp.}
For a subset $Z \subseteq X$, the closure $\overline{Z}$ in $\beta X$ is homeomorphic to $\beta Z$.
\end{cor}

\bibliographystyle{plain}
\bibliography{bib_ideals}

\begin{thebibliography}{10}

\bibitem{ADR00}
Claire Anantharaman-Delaroche and Jean Renault.
\newblock {\em Amenable groupoids}, volume~36 of {\em Monographies de
  L'Enseignement Math\'{e}matique}.
\newblock L'Enseignement Math\'{e}matique, Geneva, 2000.
\newblock With a foreword by Georges Skandalis and Appendix B by E. Germain.

\bibitem{BBFKVW22}
Florent~P. Baudier, Bruno~M. Braga, Ilijas Farah, Ana Khukhro, Alessandro
  Vignati, and Rufus Willett.
\newblock Uniform {R}oe algebras of uniformly locally finite metric spaces are
  rigid.
\newblock {\em Invent. Math.}, 230(3):1071--1100, 2022.

\bibitem{BC00}
Paul Baum and Alain Connes.
\newblock Geometric {$K$}-theory for {L}ie groups and foliations.
\newblock {\em Enseign. Math.}, 46(1/2):3--42, 2000 (firstly circulated in
  1982).

\bibitem{BCH94}
Paul Baum, Alain Connes, and Nigel Higson.
\newblock Classifying space for proper actions and {$K$}-theory of group
  {$C^*$}-algebras.
\newblock {\em Contemporary Mathematics}, 167:241--241, 1994.

\bibitem{BL20}
Christian B\"{o}nicke and Kang Li.
\newblock Ideal structure and pure infiniteness of ample groupoid
  {$C^*$}-algebras.
\newblock {\em Ergod. Theory Dyn. Syst.}, 40(1):34--63, 2020.

\bibitem{BCL20}
Bruno~M. Braga, Yeong~Chyuan Chung, and Kang Li.
\newblock Coarse {B}aum-{C}onnes conjecture and rigidity for {R}oe algebras.
\newblock {\em J. Funct. Anal.}, 279(9):108728, 21, 2020.

\bibitem{BF21}
Bruno~M. Braga and Ilijas Farah.
\newblock On the rigidity of uniform {R}oe algebras over uniformly locally
  finite coarse spaces.
\newblock {\em Trans. Amer. Math. Soc.}, 374(2):1007--1040, 2021.

\bibitem{BFV20}
Bruno~M. Braga, Ilijas Farah, and Alessandro Vignati.
\newblock Embeddings of uniform {R}oe algebras.
\newblock {\em Comm. Math. Phys.}, 377(3):1853--1882, 2020.

\bibitem{BFV22}
Bruno~M. Braga, Ilijas Farah, and Alessandro Vignati.
\newblock General uniform {R}oe algebra rigidity.
\newblock {\em Ann. Inst. Fourier (Grenoble)}, 72(1):301--337, 2022.

\bibitem{BCS22}
Kevin~Aguyar Brix, Toke~Meier Carlsen, and Aidan Sims.
\newblock Some results regarding the ideal structure of c*-algebras
  of$\backslash$'etale groupoids.
\newblock {\em arXiv preprint, arXiv:2211.06126}, 2022.

\bibitem{BO08}
Nathanial~P. Brown and Narutaka Ozawa.
\newblock {\em {$C^*$}-algebras and finite-dimensional approximations},
  volume~88 of {\em Graduate Studies in Mathematics}.
\newblock Amer. Math. Soc., Providence, RI, 2008.

\bibitem{CTWY08}
Xiaoman Chen, Romain Tessera, Xianjin Wang, and Guoliang Yu.
\newblock Metric sparsification and operator norm localization.
\newblock {\em Adv. Math.}, 218(5):1496--1511, 2008.

\bibitem{CW01}
Xiaoman Chen and Qin Wang.
\newblock Notes on ideals of {R}oe algebras.
\newblock {\em Q. J. Math.}, 52(4):437--446, 2001.

\bibitem{CW04}
Xiaoman Chen and Qin Wang.
\newblock Ideal structure of uniform {R}oe algebras of coarse spaces.
\newblock {\em J. Funct. Anal.}, 216(1):191 -- 211, 2004.

\bibitem{CW04b}
Xiaoman Chen and Qin Wang.
\newblock Ideal structure of uniform {R}oe algebras over simple cores.
\newblock {\em Chinese Ann. Math. Ser. B}, 25(2):225--232, 2004.

\bibitem{CW05}
Xiaoman Chen and Qin Wang.
\newblock Ghost ideals in uniform {R}oe algebras of coarse spaces.
\newblock {\em Arch. Math. (Basel)}, 84(6):519--526, 2005.

\bibitem{CW06}
Xiaoman Chen and Qin Wang.
\newblock Rank distributions of coarse spaces and ideal structure of {R}oe
  algebras.
\newblock {\em Bull. London Math. Soc.}, 38(5):847--856, 2006.

\bibitem{CWW09}
Xiaoman Chen, Qin Wang, and Xianjin Wang.
\newblock Operator norm localization property of metric spaces under finite
  decomposition complexity.
\newblock {\em J. Funct. Anal.}, 257(9):2938--2950, 2009.

\bibitem{CWY13}
Xiaoman Chen, Qin Wang, and Guoliang Yu.
\newblock The maximal coarse {B}aum--{C}onnes conjecture for spaces which admit
  a fibred coarse embedding into {H}ilbert space.
\newblock {\em Adv. Math.}, 249:88--130, 2013.

\bibitem{EM19}
Eske~Ellen Ewert and Ralf Meyer.
\newblock Coarse geometry and topological phases.
\newblock {\em Comm. Math. Phys.}, 366(3):1069--1098, 2019.

\bibitem{Fin14}
Martin Finn-Sell.
\newblock Fibred coarse embeddings, a-{T}-menability and the coarse analogue of
  the {N}ovikov conjecture.
\newblock {\em J. Funct. Anal.}, 267(10):3758--3782, 2014.

\bibitem{FW14}
Martin Finn-Sell and Nick Wright.
\newblock Spaces of graphs, boundary groupoids and the coarse {B}aum-{C}onnes
  conjecture.
\newblock {\em Adv. Math.}, 259:306--338, 2014.

\bibitem{Geo11}
Vladimir Georgescu.
\newblock On the structure of the essential spectrum of elliptic operators on
  metric spaces.
\newblock {\em J. Funct. Anal.}, 260(6):1734--1765, 2011.

\bibitem{HLS02}
Nigel Higson, Vincent Lafforgue, and Georges Skandalis.
\newblock Counterexamples to the {B}aum-{C}onnes conjecture.
\newblock {\em Geom. Funct. Anal.}, 12(2):330--354, 2002.

\bibitem{HR95}
Nigel Higson and John Roe.
\newblock On the coarse {B}aum-{C}onnes conjecture.
\newblock In {\em Novikov conjectures, index theorems and rigidity, {V}ol. 2
  ({O}berwolfach, 1993)}, volume 227 of {\em London Math. Soc. Lecture Note
  Ser.}, pages 227--254. Cambridge Univ. Press, Cambridge, 1995.

\bibitem{HR00}
Nigel Higson and John Roe.
\newblock Amenable group actions and the {N}ovikov conjecture.
\newblock {\em J. Reine Angew. Math.}, 519:143--153, 2000.

\bibitem{HRY93}
Nigel Higson, John Roe, and Guoliang Yu.
\newblock A coarse {M}ayer-{V}ietoris principle.
\newblock {\em Math. Proc. Cambridge Philos. Soc.}, 114(1):85--97, 1993.

\bibitem{KY06}
Gennadi Kasparov and Guoliang Yu.
\newblock The coarse geometric {N}ovikov conjecture and uniform convexity.
\newblock {\em Adv. Math.}, 206(1):1--56, 2006.

\bibitem{KS19}
Matthew Kennedy and Christopher Schafhauser.
\newblock Noncommutative boundaries and the ideal structure of reduced crossed
  products.
\newblock {\em Duke Math. J.}, 168(17):3215--3260, 2019.

\bibitem{LSZ20}
Kang Li, J{\'a}n {\v{S}}pakula, and Jiawen Zhang.
\newblock Measured asymptotic expanders and rigidity for roe algebras.
\newblock {\em arXiv:2010.10749, to appear in Int. Math. Res. Not.}, 2020.

\bibitem{MRW96}
Paul~S. Muhly, Jean.~N. Renault, and Dana~P. Williams.
\newblock Continuous-trace groupoid ${C}^*$-algebras. {III}.
\newblock {\em Trans. Amer. Math. Soc.}, 348:3621--3641, 1996.

\bibitem{NY12}
Piotr~W. Nowak and Guoliang Yu.
\newblock {\em Large scale geometry}.
\newblock EMS Textbooks in Mathematics. European Mathematical Society (EMS),
  Z\"{u}rich, 2012.

\bibitem{Oza00}
Narutaka Ozawa.
\newblock Amenable actions and exactness for discrete groups.
\newblock {\em C. R. Acad. Sci. Paris S\'{e}r. I Math.}, 330(8):691--695, 2000.

\bibitem{Ren80}
Jean Renault.
\newblock {\em A Groupoid Approach to ${C}^*$-Algebras}, volume 793 of {\em
  Lecture Notes in Mathematics}.
\newblock Springer-Verlag, 1980.

\bibitem{Ren91}
Jean Renault.
\newblock The ideal structure of groupoid crossed product {$C^\ast$}-algebras.
\newblock {\em J. Operator Theory}, 25(1):3--36, 1991.
\newblock With an appendix by Georges Skandalis.

\bibitem{Roc22}
Steffen Roch.
\newblock Ideals of band-dominated operators.
\newblock {\em Complex Var. Elliptic Equ.}, 67(3):701--715, 2022.

\bibitem{Roe88}
John Roe.
\newblock An index theorem on open manifolds. {I}, {II}.
\newblock {\em J. Differential Geom.}, 27(1):87--113, 115--136, 1988.

\bibitem{Roe93}
John Roe.
\newblock {\em Coarse cohomology and index theory on complete {R}iemannian
  manifolds}, volume 497.
\newblock Amer. Math. Soc., 1993.

\bibitem{Roe96}
John Roe.
\newblock {\em Index theory, coarse geometry, and topology of manifolds},
  volume~90.
\newblock Amer. Math. Soc., 1996.

\bibitem{Roe03}
John Roe.
\newblock {\em Lectures on coarse geometry}, volume~31 of {\em University
  Lecture Series}.
\newblock Amer. Math. Soc., Providence, RI, 2003.

\bibitem{Roe05}
John Roe.
\newblock Band-dominated {F}redholm operators on discrete groups.
\newblock {\em Integr. Equ. Oper. Theory}, 51(3):411--416, 2005.

\bibitem{RW14}
John Roe and Rufus Willett.
\newblock Ghostbusting and property {A}.
\newblock {\em J. Funct. Anal.}, 266(3):1674--1684, 2014.

\bibitem{Sak14}
Hiroki Sako.
\newblock Property {A} and the operator norm localization property for discrete
  metric spaces.
\newblock {\em J. Reine Angew. Math. (Crelle's Journal)}, 2014(690):207--216,
  2014.

\bibitem{Sie10}
Adam Sierakowski.
\newblock The ideal structure of reduced crossed products.
\newblock {\em M\"{u}nster J. Math.}, 3:237--261, 2010.

\bibitem{Sim17}
Aidan Sims.
\newblock \'{E}tale groupoids and their ${C}^*$-algebras.
\newblock {\em arxiv preprint arXiv:1710.10897}, 2017.

\bibitem{STY02}
Georges Skandalis, Jean-Louis Tu, and Guoliang Yu.
\newblock The coarse {B}aum-{C}onnes conjecture and groupoids.
\newblock {\em Topology}, 41(4):807--834, 2002.

\bibitem{Tu99}
Jean-Louis Tu.
\newblock La conjecture de {B}aum-{C}onnes pour les feuilletages moyennables.
\newblock {\em $K$-Theory}, 17(3):215--264, 1999.

\bibitem{Spa09}
J\'{a}n \v{S}pakula.
\newblock Uniform {$K$}-homology theory.
\newblock {\em J. Funct. Anal.}, 257(1):88--121, 2009.

\bibitem{SW13}
J\'{a}n \v{S}pakula and Rufus Willett.
\newblock On rigidity of {R}oe algebras.
\newblock {\em Adv. Math.}, 249:289--310, 2013.

\bibitem{SW17}
J\'{a}n \v{S}pakula and Rufus Willett.
\newblock A metric approach to limit operators.
\newblock {\em Trans. Amer. Math. Soc.}, 369(1):263--308, 2017.

\bibitem{Wan07}
Qin Wang.
\newblock Remarks on ghost projections and ideals in the {R}oe algebras of
  expander sequences.
\newblock {\em Arch. Math. (Basel)}, 89(5):459--465, 2007.

\bibitem{Wil09}
Rufus Willett.
\newblock Some notes on property {A}.
\newblock In {\em Limits of graphs in group theory and computer science}, pages
  191--281. EPFL Press, Lausanne, 2009.

\bibitem{WY12}
Rufus Willett and Guoliang Yu.
\newblock Higher index theory for certain expanders and {G}romov monster
  groups, {I}.
\newblock {\em Adv. Math.}, 229(3):1380--1416, 2012.

\bibitem{Yu00}
Guoliang Yu.
\newblock The coarse {B}aum-{C}onnes conjecture for spaces which admit a
  uniform embedding into {H}ilbert space.
\newblock {\em Invent. Math.}, 139(1):201--240, 2000.

\end{thebibliography}

\end{document}